\documentclass[10pt,a4paper]{article}

\usepackage{amsmath,mathrsfs}
\usepackage{amsthm,amssymb,amsfonts,dsfont}
\allowdisplaybreaks
\usepackage{xpatch}
\usepackage[utf8]{inputenc}
\usepackage{setspace, color}
\usepackage{array}
\usepackage{cite}
\usepackage{url}
\usepackage{multicol}
\usepackage{multirow}
\usepackage{array}
\usepackage{ragged2e}
\usepackage{graphicx}
\usepackage{pdflscape}
\usepackage{epstopdf} 
\usepackage{booktabs}
\usepackage{natbib}
\usepackage{hyperref}
\hypersetup{
    colorlinks=true,
    linkcolor=blue,
    filecolor=magenta,      
    urlcolor=cyan,
    pdftitle={Overleaf Example},
    pdfpagemode=FullScreen,
    }

\newtheorem{theorem}{Theorem}[section]
\newtheorem{definition}[theorem]{Definition}

\newtheorem{lemma}[theorem]{Lemma}
\newtheorem{proposition}[theorem]{Proposition}
\newtheorem{corollary}[theorem]{Corollary}
\newtheorem{remark}[theorem]{Remark}

\makeatletter
   \xpatchcmd{\@thm}{\fontseries\mddefault\upshape}{}{}{} 
\makeatother

\usepackage[margin=1in]{geometry}

\title{On the global dynamics and blow-up dichotomy for inhomogeneous coupled nonlinear Schrödinger systems}
\author{Mykael Araujo Cardoso\footnote{Department of Mathematics, UFPI, Brazil, e-mail: mykael@ufpi.edu.br. This author is partially supported by CAPES} \,,\, and  Lázaro Santos Gil\footnote{Department of Mathematics, IFMG and Department of Mathematics, UFPI, Brazil, e-mail: lazaro.gil@ifmg.edu.br. This author is partially supported by CAPES}}
\date{\today}

\begin{document}

\maketitle

\begin{abstract}
In this work, we investigate the dynamics of an inhomogeneous coupled nonlinear Schrödinger system with quadratic-type interactions. Such systems arise naturally in nonlinear dynamics and mathematical physics, particularly in nonlinear optics, plasma physics, and wave propagation in inhomogeneous dispersive media.
We establish a sharp criterion characterizing the dichotomy between global existence and finite-time blow-up of solutions to the associated initial value problem. This criterion is formulated in terms of conserved quantities, namely mass and energy, measured relative to the ground state solutions of the corresponding elliptic system.
The analysis combines variational methods, conservation laws, and sharp Gagliardo–Nirenberg-type inequalities to obtain local and global well-posedness results in both subcritical and intercritical regimes.
Our results extend and unify previous studies on single and multi-component nonlinear Schrödinger equations, providing a general analytical framework applicable to a broad class of coupled systems with spatially inhomogeneous nonlinearities and quadratic growth.


\end{abstract}


\tableofcontents

\section{Introduction}\label{sec1}

 We present here a study on the asymptotic behavior of the solution of the Initial Value Problem (IVP) associated with the system of nonlinear, nonhomogeneous Schrödinger equations
\begin{equation}\label{INLS}
\left\{
\begin{array}{l}
i \alpha_k\partial_t u_k +\gamma_k\Delta u_k - \beta_ku_k +|x|^{-b}f_k(u_1,\ldots,u_l)=0,\\[2mm]
(u_1(x,0), \ldots, u_l(x,0)) = (u_{10}(x),\ldots,u_{l0}(x)), \;\;\; k=1,\ldots,l,
\end{array}
\right.
\end{equation}
where $u_k = u_k(x,t)$ is a complex-valued function, $2\leq n \leq 5$, $0<b< \min\left\{2, \dfrac{n}{2}\right\}$, $\alpha_k$, $\gamma_k$ are positive real constants, $\beta_k$ is a nonnegative real constant, and the maps  $f_k$  have a quadratic-type growth. 
%

On the nonsingular case $b=0$, where \eqref{INLS} was introduced by Noguera and Pastor (see \cite{Noguera}),
the well-posedness in $L^{2}(\mathbb{R}^{n})\times \ldots \times L^{2}(\mathbb{R}^{n})$ when $n \leq 4$ and in $H^{1}(\mathbb{R}^{n})\times \ldots \times H^{1}(\mathbb{R}^{n})$ when $1 \leq n \leq 6$ was completely investigated in \cite{Noguera} by the contraction argument combined with Strichartz estimates. Moreover, the global well-posedness in $L^{2}(\mathbb{R}^{n})\times \ldots \times L^{2}(\mathbb{R}^{n})$ when $n \leq 4$, as well as the global existence in $H^{1}(\mathbb{R}^{n})\times \ldots \times H^{1}(\mathbb{R}^{n})$ where $n\leq 3$, were obtained due to the conservation of mass and energy. 
In addition, by conservation laws, they proved that for initial data in $H^{1}(\mathbb{R}^{n})\times \ldots \times H^{1}(\mathbb{R}^{n})$,
a condition for global existence was found which, depends on the size of the initial data when compared to the associated ground states.  The global existence, energy scattering, and the finite-time blow-up in the intercritical case were studied in \cite{Noguera},  while the finite-time blow-up on the critical case were shown in \cite{Noguera2}.

On the other hand, in case where the system \eqref{INLS} is an  equation with one component, namely  $l=1$, it has been extensively studied over the past several years. Genoud and Stuart \cite{Genoud} showed local well-posedness in. Farah \cite{Farah} established the global existence of solutions and proved the existence of finite-time blow-up solutions for initial data with finite variance. Dinh \cite{Dinh} extended the existence of blow-up solutions to radial data, while Cardoso and Farah \cite{Cardoso}, and Bai and Li \cite{Bai} did so for non-radial initial data. Energy scattering was initially proven for radial initial data by Farah and Guzmán \cite{Guzman}\cite{Guzman.2}, Campos \cite{Campos}, and Dinh \cite{Dinh.2}, and was later extended to non-radial initial data by Miao, Murphy, and Zheng \cite{Miao}, as well as Cardoso, Farah, Guzmán, and Murphy \cite{Cardoso-Murphy} (see also \cite{Murphy}).

Systems as in \eqref{INLS}, $b\geq 0$ and with power-like quadratic nonlinearities, appear in several areas in physics such as nonlinear optics, plasma physics, propagation in nonlinear fibers, among others.  In optical media, for instance, Dinh and Esfahani \cite{Doung}  studied  the following model
%
\begin{equation}\label{pvi.arxiv}
\left\{
\begin{array}{l}
i\partial_t u +\dfrac{1}{2}\Delta u +|x|^{-b}\overline{u}v=0,\\[2mm]
i\partial_t v +\dfrac{\kappa}{2}\Delta v - \gamma v +|x|^{-b}u^2=0,
\end{array}
\right.
\end{equation}
%
that describes the two-wave (degenerate, alias type-I) quadratic interactions for the complex fundamental - frequency (FF) and second-harmonic (SH) amplitudes in the presence of the spatial singular modulation of the $\chi^{(2)}$  nonlinearity. For other examples, when the propagation of optical beams in a nonlinear dispersive medium with quadratic respose is considered, the following three-wave interaction models appear (see Noynov and Saltiel \cite{Koynov} for the case with $b=0$):
%
\begin{equation}\label{example1}
\left\{
\begin{array}{l}
2i\partial_t w +\Delta w -\beta w = -\dfrac{1}{2}|x|^{-b}(u^2 + v^2),\\[2mm]
i\partial_t v +\Delta v - \beta_{1} v = -|x|^{-b} \overline{v}w, \\[2mm]
i\partial_t u +\Delta u -u = -|x|^{-b}\overline{u}w,
\end{array}
\right.
\end{equation}
and
\begin{equation}\label{example2}
\left\{
\begin{array}{l}
i\partial_t w +\Delta w - w = -|x|^{-b}(\overline{w}v + \overline{v}u),\\[2mm]
2i\partial_t v +\Delta v - \beta v = -|x|^{-b}\left( \dfrac{1}{2}w^{2} + \overline{w}u\right), \\[2mm]
3i\partial_t u +\Delta u -\beta_{1}u = -|x|^{-b}vw,
\end{array}
\right.
\end{equation}
where $\beta$, $\beta_{1} >0$ are real constants from the mathematical point of view, Pastor \cite{Pastor2} studied the above systems, with $b=0$, in the one-dimensional case. Global well-posedness in the energy space, existence of ground state solutions as well as their stability were established. Inspired by these works, in particular, following the ideas \cite{Noguera}, we establish sufficient conditions on the general interaction terms $|x|^{-b}f_k$ to analyze the dynamics of the system \eqref{INLS}. A key feature of our analysis is that we do not assume an explicit form for the nonlinearities $f_k$, treating instead a general class of functions with quadratic interactions.

Our initial discussion will focus on the local well-posedness ones. For this to hold, the following conditions are assumed
\begin{itemize}
    \item[\textbf{(H1)}] $$f_k(0,\ldots,0)=0, \;\;\;\;\;\;\; k=1,\ldots,l.$$
    \item[\textbf{(H2)}] There exists a constant $C>0$ such that for $(z_1,\ldots,z_l)$, $(z_{1}^{'},..,z_{l}^{'})\in \mathbb{C}^{l}$ we have
    \begin{eqnarray*}
        \left| \dfrac{\partial}{\partial z_m} [f_k(z_1,\ldots,z_l)-f_k(z_{1}^{'},..,z_{l}^{'})]\right| &\leq& C \sum_{j=1}^{l}|z_j-z_{j}^{'}|, \;\;\;\;\;\; k,m\,=\,1,\ldots,l; \\[2mm]
        \left| \dfrac{\partial}{\partial \overline{z}_m} [f_k(z_1,\ldots,z_l)-f_k(z_{1}^{'},..,z_{l}^{'})] \right| &\leq& C \sum_{j=1}^{l}|z_j-z_{j}^{'}|, \;\;\;\;\;\; k,m\,=\,1,\ldots,l. 
    \end{eqnarray*}
\end{itemize}
Next,to establish global well-posedness via conservation laws, we assume the following
\begin{itemize}
    \item[\textbf{(H3)}] There exists a function $F:\mathbb{C}^{l} \to \mathbb{C}$, such that
    $$
    f_k(z_1,\ldots,z_l) = \dfrac{\partial F}{\partial \overline{z}_m}(z_1,\ldots,z_{l}) + \overline{\dfrac{\partial F}{\partial {z}_m}} (z_1,\ldots,z_{l}).
    $$
    \item[\textbf{(H4)}] There exists positive constants $\sigma_1$,\ldots, $\sigma_{l}$ such that for any $\theta \in \mathbb{R}$ and $(z_1,\ldots,z_{l}) \in \mathbb{C}^{l}$
    $$
    \mbox{Re} F(e^{i \sigma_1 \theta} z_1, \ldots, e^{i \sigma_1 \theta} z_{l}) = 
    \mbox{Re} F( z_1, \ldots,  z_{l}). 
    $$
    \end{itemize}
\begin{remark}\label{remarkH3H4}
  Similar to \cite[Lemma~2.9.]{Noguera}, the properties \emph{(H3)}-\emph{(H4)} imply that, for any $(z_1,\ldots,z_{l}) \in \mathbb{C}^{l}$,
    $$
    \mbox{Im} \sum_{k=1}^{l} \sigma_k f_k(z_1,\ldots,z_{l})\overline{z}_k =0.
    $$
\end{remark}
    \begin{itemize}
    \item[\textbf{(H5)}] Function $F$ is homogeneous of degree $3$, that is, for any $\lambda >0$ and $(z_1,\ldots,z_{l}) \in \mathbb{C}^{l}$,
    $$
    F(\lambda z_1,\ldots, \lambda z_l) \;=\; \lambda^{3}     F( z_1,\ldots,z_l).
    $$
\end{itemize}
Finally, concerning the ground states, we impose the following condition
\begin{itemize}
    \item[\textbf{(H6)}] There holds
    $$
    \left| \mbox{Re} \int_{\mathbb{R}^{n}} |x|^{-b}F(u_1,\ldots,u_l)\, dx \right| \;\leq\;
    \int_{\mathbb{R}^{n}} |x|^{-b} F(|u_1|,\ldots,|u_l|)\, dx.
    $$
    \item[\textbf{(H7)}] Function $F$ is real valued on $\mathbb{R}^{l}$, that is, if $(y_1,\ldots,y_l) \in \mathbb{R}^{l}$ then
    $$
    F(y_1,\ldots,y_l) \in \mathbb{R}.
    $$
    Moreover, functions $f_k$ are non-negative on the positive cone in $\mathbb{R}^{l}$, that is, for $y_j \geq 0$, $j=1,\ldots,l$,
    $$
    f_k(y_1,\ldots,y_l) \geq 0.
    $$
    \item[\textbf{(H8)}] Function $F = F_1 + \ldots+ F_{m}$, where each $F_{s}$, $s=1,\ldots,m$ is super-
    modular on $\mathbb{R}_{+}^{d}$ for some $1\leq d\leq l$ and vanishes on hyperplanes, that is, for any $i$, $j \in \{1,\ldots,d\}$, $i \neq j$ and $k,h >0$, we have
    $$
    F_s(y+he_i +ke_j) + F_s(y) \;\geq\; F_s(y+he_i) + F_s(y+ke_j), \;\;\;\;\;\;\; y \in \mathbb{R}_{+}^{d},
    $$
    and $F_s(y_1,\ldots,y_{d}) =0$ if $y_j=0$ for some $j \in \{1,\ldots,d\}$.
\end{itemize}

Although not every result requires all the assumptions, we will assume (H1)–(H8) throughout the paper. In Section $2$, we indicate the specific assumptions used in each case.  
\begin{remark}
It is easy to see that the systems \eqref{pvi.arxiv}, \eqref{example1} and, \eqref{example2} satisfy \emph{(H1)}-\emph{(H8)} with
$$
F(z_1,z_2) \;=\; \overline{z}_{1}^{2}z_2, \quad F(z_1,z_2,z_3) = \dfrac{\overline{z}_1}{2}(z_{2}^{2} +z_{3}^{2}) \quad \mbox{and,} \quad
F(z_1,z_2,z_3) = \dfrac{z_{1}^{2}\overline{z}_{2}}{2} +z_1 z_{2} \overline{z}_{3},
$$    
respectively.
\end{remark}

The equation \eqref{INLS} (with $\beta_k = 0$) admits the following symmetries: if $\mathbf{u}(x,t)$ is a solution,  then  are also solutions:
\begin{itemize}
\item[a)] $\mathbf{w}(x,t)=\mathbf{u}(x,t-t_0)$, for all $t_0\in \mathbb{R}$ (\emph{time translation invariance})
\item[b)] $\mathbf{w}(x,t)=e^{i\gamma}\mathbf{u}(x,t)$ for all $\gamma\in \mathbb{R}$ (\emph{phase invariance})
\item[c)] $\mathbf{w}(x,t)= \lambda^{2-b}\mathbf{u}(\lambda x, \lambda^2t)$ (\emph{scaling invariance}).
\end{itemize}
%
%
Furthermore, computing the homogeneous Sobolev norm, we obtain
%
$$\|\mathbf{w}(\cdot, 0)\|_{\dot{\mathbf{H}}^s} = \lambda^{s-\frac{n}{2}+2-b}\|\mathbf{u}(\cdot, 0)\|_{\dot{\mathbf{H}}^s}.$$
The Sobolev index, which leaves the scaling symmetry, invariant is called the critical index and is defined as $s_c = \frac{n+2b -4}{2}$, as it represents the minimal regularity for which one can expect local well-posedness for the equations. Note that, if $s_c=0$ (or $n+2b=4$) the system \eqref{INLS} is known as mass-critical or $L^{2}$-critical; if $s_c =1 $ (or $n+2b=6$) it is called energy-critical or $H^{1}$-critical; if $0<s_c<1$ (or $4<n+2b<6$), the system \eqref{INLS} is known as mass-supercritical and energy-subcritical or intercritical. More precisely, we say that system \eqref{INLS} is:
$$
L^2 - 
\left\{ \begin{array}{lc}
\mbox{subcritical}, \;\;\;\, \textrm{ if } n+2b < 4\\
\mbox{critical},\;\;\;\;\;\;\;\;\; \textrm{ if } n+2b = 4\\
\mbox{supercritical}, \;\textrm{ if } n+2b > 4.
\end{array}\right.
\;\;\;\;\;
\mbox{and}
\;\;\;\;\;
\dot{H}^1 - 
\left\{ \begin{array}{lc}
\mbox{subcritical}, \;\;\;\, \textrm{ if } n+2b < 6\\
\mbox{critical},\;\;\;\;\;\;\;\;\; \textrm{ if } n+2b = 6\\
\mbox{supercritical}, \;\textrm{ if } n+2b > 6.
\end{array}\right.
$$

Concerning local well-posedness for the IVP \eqref{INLS}, the first result of this work is as follows

\begin{theorem}\emph{\textbf{(Existence of local $\mathbf{L^{2}}$-solutions: subcritical case)}} \label{l2-existence solutions}Let $1\leq n \leq 3$, $n+2b< 4$ and $0<b< \min\left\{2,\dfrac{n}{2}\right\}$. Assume that \emph{(H1)} and \emph{(H2)} hold. Then for any $r>0$ there exists $T=T(r) >0$ such that for any $\mathbf{u}_0 \in \mathbf{L}^{2}$     with $\|\mathbf{u}_{0}\|_{\mathbf{L}^2} < r$, system \eqref{INLS} has a unique solution $\mathbf{u} \in \mathbf{X}(I)$ with $I=[-T,T]$. 
\end{theorem}

\begin{remark}
    We will also prove that under assumptions \emph{(H3)}-\emph{(H4)} the solution of the Theorem \ref{l2-existence solutions} is always global, i.e., system \eqref{INLS} has a unique solution $\mathbf{u} \in \mathbf{X}(\mathbb{R})$. 
\end{remark}

After this result in $\mathbf{L}^2$, the work focuses on studying the solutions in $\mathbf{H}^1$. The corresponding well-posedness result is stated as follows

\begin{theorem}\label{Existence H1}\emph{\textbf{(Existence of local $\mathbf{H^{1}}$-solutions: subcritical case $\mathbf{ n \geq 2}$)}} Let $n\geq 2$, 
    $0<b<\min\left\{2, \dfrac{n}{2}\right\}$ and $n+2b < 6$. Assume that \emph{(H1)} and \emph{(H2)} hold. Then for any $r>0$ there exists $T=T(r) >0$ such that for any $\mathbf{u}_0 \in \mathbf{H}^{1}$ with $\|\mathbf{u}_{0}\|_{\mathbf{H}^1} < r$, system \eqref{INLS} has a unique solution $\mathbf{u} \in \mathbf{Y}(I)$ with $I=[-T,T]$. 
\end{theorem}
Assuming (H1)-(H4), the solution of Theorem \ref{Existence H1} satisfies, for all $t\in [-T,T]$, the conservation of charge:
%
    \begin{eqnarray}\label{Q(u)}
    Q(\mathbf{u}(t))\;=\;\sum_{k=1}^{l} {\alpha_k \sigma_k} \int_{\mathbb{R}^n} |u_k(x,t)|^2\, dx
    \end{eqnarray}
and the conservation of energy:
\begin{equation}\label{E(u)}
    E(\mathbf{u}(t)) = K(\mathbf{u}(t)) +L(\mathbf{u}(t))-2P(\mathbf{u}(t)),
\end{equation}
where 
\begin{eqnarray}\label{K,L,P}
    K(\mathbf{u}) = 
    \sum_{k=1}^{l} \gamma_k \|\nabla u_k \|_{L^2}^{2}, \;\;\;\; 
    L(\mathbf{u}) =
    \sum_{k=1}^{l} \beta_k \| u_k \|_{L^2}^{2} \;\;\;\mbox{and} \;\;\; 
    P(\mathbf{u}) =
    \mbox{Re} \int_{\mathbb{R}^{n}} |x|^{-b} F(\mathbf{u})\, dx  
\end{eqnarray}
are well defined for all  $\mathbf{u} = (u_1,\ldots,u_l) \in \mathbf{H}^{1}$.

%
The maximal existence time of the solution $T(\mathbf{u}_0) > 0$ depends only on the $\mathbf{H}^{1}$ norm of the initial data. This, in turn, together with energy conservation, implies the blow-up alternative:
$$
\mbox{if}\;\; T(\mathbf{u}_0) < \infty, \;\;\;
\mbox{then} \;\;\;     \lim_{t\to T(\mathbf{u}_0) }\|\mathbf{u}(t)\|_{\mathbf{H}^{1}} = +\infty.
$$
We say that a solution blows up in finite time if its maximal existence time is finite (and therefore the above limit holds).

Our main goal in the present paper is to establish a sharp criterion concerning the dichotomy between global existence and blow-up in finite time. Before presenting our results, we first investigate the existence of a special solution for system \eqref{INLS}. More precisely, a standing wave solution for \eqref{INLS} is a particular solution of the form
\begin{equation}\label{wave-solution}
u_k(x,t) \;=\; e^{i \sigma_k \omega t} \psi_{k}(x), \;\;\;\;\;\;\;\;\; k=1,\ldots,l,
\end{equation}
where $\omega \in \mathbb{R}$ and $\psi_{k}$ are real-valued functions decaying to zero at infinity. Substituting into \eqref{INLS}, under the assumpstions of (H3)-(H4) for $k=1,\ldots,l$ and any $\omega \in \mathbb{R}$, we find that $\psi_{k}$ satisfies the equation 
\begin{equation}\label{PDE-GroundState}
- \gamma_k \Delta \psi_k + b_k \psi_k =|x|^{-b}f_k(\boldsymbol{\psi}), \;\;\;\;\;\; k=1,\ldots,l,    
\end{equation}
where $b_k = \alpha_k \sigma_k \omega+ \beta_k \geq0$. The action function associated with \eqref{PDE-GroundState} is
\begin{equation}\label{I=SumH1-P}
I(\boldsymbol{\psi}) = \dfrac{1}{2} \left[ \sum_{k=1}^{l} \gamma_k \|\nabla \psi_k\|_{L^2}^{2} + \sum_{k=1}^{l}  b_k\|\psi_k\|_{L^2}^{2} \right] - \int_{\mathbb{R}^{n}} |x|^{-b} F(\boldsymbol{\psi}) dx.    
\end{equation}
A ground state solution for \eqref{PDE-GroundState} is a nontrivial critical point of the functional $I$ that also minimizes it. The preceding discussion leads to following definitions
\begin{definition}\label{weak-solutions}
    We say that $\boldsymbol{\psi} \in \mathbf{H}^{1}$ is a (weak) solution of \eqref{PDE-GroundState} if for any $\mathbf{g} \in \mathbf{H}^{1}$,
    \begin{equation*}\label{weak}
    \gamma_k \int_{\mathbb{R}^{n}} \nabla \psi_k\nabla g_k \, dx + b_k \int_{\mathbb{R}^{n}} \psi_k g_k \, dx \;=\; \int_{\mathbb{R}^{n}} |x|^{-b} f_k(\boldsymbol{\psi})g_k \,dx, \;\;\;\;\; k=1,\ldots,l.
    \end{equation*}
\end{definition}

\begin{remark} Note that  from equation \eqref{weak} we have  the following  conditions 
    \begin{itemize}
        \item[ $(i)$] System \eqref{PDE-GroundState} makes sense, since the right-hand side of the system is real, it is due from the properties of $f_k$ (see Lemma \ref{Ademir.lem2.10.2.11.2.13});
        \item[$(ii)$] Observe that $\boldsymbol{\psi} = \mathbf{0}$ is always a solution (trivial solution) of \eqref{PDE-GroundState}. Hence we will always be interested in non-trivial solutions; 
        \item[$(iii)$] In order to obtain non-trivial solutions, we will restrict to values of $\omega$ to those such that $b_k >0$. 
    \end{itemize}
\end{remark}

\begin{definition}
    Let $\mathcal{C}$ be the set of non-trivial critical points of $I$. We say that $\boldsymbol{\psi} \in \mathbf{H}^{1}$ is a ground state solution \eqref{PDE-GroundState} if
    $$
I(\boldsymbol{\psi}) = \inf \{ I(\phi);\;\;\; \phi \in \mathcal{C}\}.
$$
We denote by $\mathcal{G}(\omega, \boldsymbol{\beta})$ the set of all ground states for system \eqref{PDE-GroundState}, where $(\omega, \boldsymbol{\beta})$ indicates the dependence on the parameters $\omega$ and $\boldsymbol{\beta}$.
\end{definition}

The existence of this special solution is established in the following theorem.

\begin{theorem}\label{GS-existence}
    The set $\mathcal{G}(\omega, \boldsymbol{\beta})$ is non-empty, and thus there exists at least one ground state solution $\boldsymbol{\psi}$ for system \eqref{PDE-GroundState}.
\end{theorem}

Once the existence of a ground state is established, we can write the sufficient condition for global solutions.

\begin{theorem}\emph{\textbf{(Sufficient condition for global solutions)}}\label{thrm: Global}.  Let $n \geq 2$, $n+2b<6$ and 
    $0<b<\min\left\{2, \dfrac{n}{2}\right\}$. Assume that \emph{(H1)-(H8)} holds, $\mathbf{u}_0 \in \mathbf{H}^{1}(\mathbb{R}^{n})$ and let $\mathbf{u}$ be the solution of \eqref{INLS} defined in the maximal existence interval $I$. Let $\boldsymbol{\psi} \in \mathcal{G}(1,\mathbf{0})$. 
    \begin{itemize}
            \item[$(i)$] If $2 \leq n \leq 3$ with $n+2b<4$ , then for any $\mathbf{u}_0 \in \mathbf{H}^1$, system \eqref{INLS} has a unique solution $\textbf{u} \in \mathbf{H}^{1}(\mathbb{R}^{n})$.
        \item[$(ii)$] Assume $n=3$ and $2b=1$. If
        \begin{eqnarray}\label{Q<Q}
            Q(\mathbf{u}_0) < Q(\boldsymbol{\psi})
        \end{eqnarray}                 
        then the initial value problem \eqref{INLS} is globally well-posed in $\mathbf{H}^{1}(\mathbb{R}^{n})$.
        \item[$(iii)$] Assume $4 < n+2b < 6$, $0<b<1$ and in addition that
        $$
    E(\mathbf{u}_0)^{s_c} Q(\mathbf{u}_0)^{1-s_c} < \mathcal{E}(\boldsymbol{\psi})^{s_c} Q(\boldsymbol{\psi})^{1-s_c}
    $$
    where $s_c = \dfrac{n+2b-4}{2}$ and $\mathcal{E}(\boldsymbol{\psi}) = K(\boldsymbol{\psi}) -2P(\boldsymbol{\psi})$. If
\begin{eqnarray}\label{KQ<KQ}
    K(\mathbf{u}_0)^{s_c} Q(\mathbf{u}_0)^{1-s_c} < K(\boldsymbol{\psi})^{s_c} Q(\boldsymbol{\psi})^{1-s_c}
\end{eqnarray}
    then
    $$
    K(\mathbf{u}(t))^{s_c} Q(\mathbf{u}_0)^{1-s_c} < K(\boldsymbol{\psi})^{s_c} Q(\boldsymbol{\psi})^{1-s_c}
    $$
    In particular, the initial value problem \eqref{INLS} is globally well-posed in $\mathbf{H}^{1}(\mathbb{R}^{n})$.
    \end{itemize}
\end{theorem}

Finally, we establish the sharpness of condition $(iii)$ in Theorem \ref{thrm: Global} under specific assumptions on the coefficients of \eqref{INLS}. This is demonstrated by constructing initial data that fail to satisfy this condition, yet lead to finite-time blow-up. 
The final result is as follows:

\begin{theorem}\label{blow-up}
    Assume $\mathbf{u}_0 \in \mathbf{H}^{1}(\mathbb{R}^{n})$ and let $\mathbf{u}$ be the solution of \eqref{INLS} defined in the maximal existence interval $I$. Let $\boldsymbol{\psi} $ the ground state of the system \eqref{PDE-GroundState}. Assume $4 < n+2b <6$, $0<b<\min\left\{2, \dfrac{n}{2}\right\}$ and in addition that
    \begin{eqnarray}\label{E(u_0)Q(u_0)}
       E(\mathbf{u}_0)^{s_c} Q(\mathbf{u}_0)^{1-s_c} < \mathcal{E}(\boldsymbol{\psi})^{s_c} Q(\boldsymbol{\psi})^{1-s_c} 
    \end{eqnarray}
    where $s_c = \dfrac{n+2b-4}{2}$ and $\mathcal{E}(\boldsymbol{\psi}) = K(\boldsymbol{\psi}) -2P(\boldsymbol{\psi})$. If
     \begin{eqnarray}\label{K(u_0)Q(u_0)}
    K(\mathbf{u}_0)^{s_c} Q(\mathbf{u}_0)^{1-s_c} > K(\boldsymbol{\psi})^{s_c} Q(\boldsymbol{\psi})^{1-s_c}
    \end{eqnarray}
    Then, if $\mathbf{u}_0$ is radially symmetric we have that $I$ is finite.
\end{theorem}

The structure of the paper is as follows. Section 2 contains preliminary results derived from assumptions (H1)–(H8). Section 3 is devoted to the study of the local and global theories of system \eqref{INLS} in the spaces $L^2$ and $H^1$. Our approach follows the standard techniques for Schrödinger-type equations: the Strichartz estimates, Hölder's inequality and, Sobolev embedding combined with the contraction mapping principle are sufficient to prove the local well-posedness. On the other hand, the global results are proved in view of an a priori bound of the local solution in the spaces of interest. In particular, solutions are global for any initial data in the subcritical regimes; global existence holds provided that certain conditions on the charge and energy of the initial data are satisfied. In Section 4, we investigate the existence of ground state solutions for the associated elliptic system. This step is necessary since we want to obtain a sharp Gagliardo–Nirenberg-type inequality, whose best constant depends on such solutions. The existence of ground states is established by minimizing the so-called Weinstein functional over an appropriate set. Finally, in Section 5 we are interested in the dichotomy between global well-posedness and finite-time blow-up. In the critical case, we establish a sharp criterion for global existence, depending on the parameters of the system. In the supercritical case, we show that solutions remain global provided there is a suitable balance between the charge and the energy of the initial data, measured relative to those of the ground states.

\section{Notation and preliminaries}

In this section, we introduce some notation, definitions, and results necessary for a good understanding of the entire text.

\subsection{Notation} 

    We use $C$ to denote various constants that
    may represent different values in different lines of an estimate, even within the same proof.
    Given any set $A$, by $\mathbf{A}$ we denote the product $A \times \ldots \times A$ ($l$ times). In particular, if $A$ is a Banach space with norm $\| \cdot \|_{A}$, then $\mathbf{A}$ is also a Banach space with the standard norm given by the sum. 
    Given $f$ a complex function, its real part will be denoted by $\mbox{Re}\,\,f$ and the imaginary
part by  $\mbox{Im}\,\,f$. 
Similarly, given $z$ a complex number, $\mbox{Re}\,\,z$ and $\mbox{Im}\,\,z$ represent its  real and imaginary parts. Also, $\overline{z}$ denotes its complex conjugate. 
In $\mathbb{C}^{l}$ we frequently write $\mathbf{z}$ and $\mathbf{z}^{'}$ instead of $(z_1,\ldots,z_l)$ and $(z_{1}^{'},\ldots, z_{l}^{'})$. 
Given $\mathbf{z} = (z_1,\ldots,z_l) \in \mathbb{C}^{l}$, we write $z_m \;=\; x_m + i y_m$, where $x_m = \mbox{Re}\,\, z_m$ and $y_m = \mbox{Im}\,\, z_m$. As usual, the operators $\dfrac{\partial}{\partial z_m}$ and $\dfrac{\partial}{\partial \overline{z}_m}$ are defined by 
    $$
    \dfrac{\partial}{\partial z_m} = \dfrac{1}{2} \left( \dfrac{\partial}{\partial x_m} - i \dfrac{\partial}{\partial y_m}\right), \;\;\;\;\;\;\;\;\;\;
    \dfrac{\partial}{\partial \overline{z}_m} = \dfrac{1}{2} \left( \dfrac{\partial}{\partial x_m} +  i \dfrac{\partial}{\partial y_m}\right).
    $$

Let $X,Y$ be Banach spaces. $C^k(X,Y)$ will denote the space of all maps from $X$ to $Y$ with $k$ continuous derivatives. 
For each $p$, $1\leq p < \infty$, let $L^p = L^p(\mathbb{R}^{n})$ be the class of measurable functions, $f:\mathbb{R}^n\to \mathbb{C}$,   
such that the integral (of Lebesgue)
$$\|f\|_{L^p} :=\left(\int_{\mathbb{R}^{n}}|f(x)|^pdx\right)^{\frac{1}{p}},
$$ 
    is finite. Also, let $L^{\infty}=L^{\infty}(\mathbb{R}^n)$ be the collection of functions that are bounded almost everywhere, i.e., at all points up to a null set.
$$\|f\|_{L^{\infty}}=ess\sup\{|f(x)|;\,\,x\in \mathbb{R}^{n}\}.$$
    Equipped with the norms
    $\|\cdot\|_{L^p}$ and $\|\cdot\|_{L^{\infty}}$,
    each $L^p(\mathbb{R}^n)$, $1\leq p \leq \infty$, is a Banach space.
    Additionally, for any $b>0$,  we define the weighted 
$L_{b}^{p} = L_{b}^{p}(\mathbb{R}^{n})$ norm  by 
    %
$$\|f\|_{L_{b}^{p}} :=\left(\int_{\mathbb{R}^{n}}|x|^{-b}|f(x)|^pdx\right)^{\frac{1}{p}},\,\,\,\,\,\,1\leq p < \infty.$$ 
The Fourier transform in the Schwarz space $\mathcal{S}(\mathbb{R}^{n})$ is defined as $$\widehat{f}(\xi):= \int_{\mathbb{R}^{n}} e^{-2 \pi i x \cdot \xi} f(x) dx.$$
We make use of the fractional differentiation operators $|\nabla|^{s}$ defined by
$$\widehat{|\nabla|^{s}f}(\xi) = |\xi|^{s} \widehat{f}(\xi).$$ 
    %
%
For $s\in \mathbb{R}$, we define the  Sobolev space $H^s(\mathbb{R}^n)$ by
$$
H^s=H^s(\mathbb{R}^{n})=\left\{f\in \mathcal{S}'(\mathbb{R}^n): J^sf\in L^2(\mathbb{R}^{n}),\mbox{ where } J^sf=((1+|\xi|^2)^{\frac{s}{2}}\widehat{f})^{\vee}\right\},
$$
equipped with the norm
$$\|f\|_{H^s}=\|J^sf\|_{L^2}=\left(\int_{\mathbb{R}^n}(1+|\xi|^2)^s|\widehat{f}(\xi)|^2d\xi\right)^{\frac{1}{2}}=\|(1+|\xi|^2)^{\frac{s}{2}}\widehat{f}\|_{L^2}.$$
%
%
%
Let $1 \leq p,q\leq \infty$ and a time interval $I$, the mixed space $L_{I}^{q}L_{x}^{p}$  is defined by
$$
L_{I}^{q}L_{x}^{p} \;=\;
\left\{
f: \mathbb{R}^{n} \times I  \to \mathbb{C}
:\;\; \|f\|_{L_{I}^{q}L_{x}^{p}} \;=\; \left( \int_{I} 
\left( \int_{\mathbb{R}^{n}} |f(x,t)|^{p} \,dx \right)^{\frac{q}{p}}
\,dt \right)^{\frac{1}{q}} < \infty
\right\},
$$
with the obvious modifications if either $p=\infty$ or $q=\infty$.
%
%
%
\begin{definition}
    The pair $(q,r)$ is called $L^{2}$-admissible if it satisfies the condition
$$
\dfrac{2}{q} \;=\; \dfrac{n}{2} - \dfrac{n}{r}
$$
where
\begin{eqnarray*}
\left\{ \begin{array}{lc}
2 \leq r \leq \dfrac{2n}{n-2} \;\;\;\; &\mbox{if} \;\; n\geq 3, \\[2mm]
2 \leq r < +\infty &\mbox{if} \;\; n = 2, \\[2mm]
2 \leq r \leq +\infty  &\mbox{if} \;\; n =1. \\[2mm]
\end{array}\right.    
\end{eqnarray*}
\end{definition}%
%
\begin{definition}
let $\mathcal{A}_0 \;=\; \{(q,r); \;\; (q,r) \;\;\mbox{is}\;\; L^{2}\mbox{-admissible}  \}$ and $(q^{'}, r^{'})$ is such that $\frac{1}{q} + \frac{1}{q^{'}} =1$ and $\frac{1}{r} + \frac{1}{r^{'}} =1$ for $(q,r) \in \mathcal{A}_0$.   We define the following Strichartz norm
$$
\|f\|_{S(L^{2},I)} \;=\; \sup_{(q,r) \in \mathcal{A}_0} \|f\|_{L_{I}^{q} L_{x}^{r}}
$$
and the dual Strichartz norm
$$
\|f\|_{S^{'}(L^{2},I)} \;=\; \inf_{(q,r) \in \mathcal{A}_{0}} \|f\|_{L_{I}^{q^{'}} L_{x}^{r^{'}}}.
$$    
In the case when $f$ and $\nabla f$ is $L^2$-admissible, we will use the  Strichartz norm given by the sum
$$
\|f\|_{S(H^{1},I)} \;:=\; \|f\|_{S(L^{2},I)} + \|\nabla f\|_{S(L^{2},I)}.
$$    
and the dual Strichartz norm
as 
$$
\|f\|_{S^{'}(H^{1},I)} \;:=\; \|f\|_{S^{'}(L^{2},I)} + \|\nabla f\|_{S^{'}(L^{2},I)}.
$$ 
\end{definition}

\subsection{Preliminaries}

Let us now give some useful consequences of our assumptions. In the following two lemmas, we establish estimates for the nonlinearities $f_k$ and the potential function $F$ appearing in (H3).

\begin{lemma}\label{Ademir.cor2.3.lem2.4}
If \emph{(H1)} and \emph{(H2)} hold
\begin{itemize}
    \item[$(i)$] For any $\mathbf{z}$, $\mathbf{z}' \in \mathbb{C}^{l}$, we have
    \begin{equation}
        |f_k(\mathbf{z}) - f_k(\mathbf{z}^{'})| \leq C \sum_{m=1}^{l}\sum_{j=1}^{l} (|z_j| +|z_{j}^{'}|)|z_{m} - z_{m}^{'}|, \;\;\;\;\;\;\;\;\; k=1,\ldots,l.
    \end{equation}
    In particular,  $$|f_k(\mathbf{z}) | \leq C \sum_{j=1}^{l} |z_j|^{2}, \;\;\;\;\;\;\;\;\; k=1,\ldots,l.     $$
    \item[$(ii)$] Let $\mathbf{u}$ and $\mathbf{u}'$ be complex-valued functions defined on $\mathbb{R}^{l}$. Then, for  $ k=1,\ldots,l$
    \begin{equation}
        |\nabla [f_k(\mathbf{u}) - f_k(\mathbf{u}^{'})]| \leq C \sum_{m=1}^{l}\sum_{j=1}^{l} |u_j| |\nabla (u_{m} - u_{m}^{'})| + C \sum_{m=1}^{l}\sum_{j=1}^{l} (|u_j-u_{j}^{'}|)|\nabla u_{m}^{'}|.
   \end{equation}
\end{itemize}
\end{lemma}
\begin{proof}
    See Norman and Pastor \cite[Corollary~2.3. and Lemma~2.4.]{Noguera}.
\end{proof}

\begin{lemma}\label{Ademir.lem2.10.2.11.2.13} Assume that \emph{(H1)}-\emph{(H3)} and \emph{(H5)} hold.
\begin{itemize}
    \item[$(i)$] Let $\mathbf{z}$, $\mathbf{z}' \in \mathbb{C}^{l}$. Then     
    \begin{equation}
        |\emph{\mbox{Re}}\,F (\mathbf{z}) - \emph{\mbox{Re}}\,F(\mathbf{z}^{'})| \leq C \sum_{m=1}^{l}\sum_{j=1}^{l} (|z_j|^{2} +|z_{j}^{'}|^{2})|z_{m} - z_{m}^{'}|.
    \end{equation}
    In particular,
    $$
        |\emph{\mbox{Re}}\,F (\mathbf{z})| \leq C \sum_{m=1}^{l}\sum_{j=1}^{l} |z_j|^{3}.
    $$
    \item[$(ii)$] Let $\mathbf{u}$ be a complex-value function defined on $\mathbb{R}^{n}$. Then 
    $$ \emph{\mbox{Re}}\,\sum_{k=1}^{l}f_k(\mathbf{u}) \nabla \overline{u}_k \;=\; \emph{\mbox{Re}}\;[\nabla F(\mathbf{u})].$$
    \item[$(iii)$] Let $\mathbf{u}: \mathbb{R}^{l} \to \mathbb{C}$. Then $$\emph{\mbox{Re}}\,\sum_{k=1}^{l}f_k(\mathbf{u}) \overline{u}_k \;=\; \emph{\mbox{Re}}\;[3 F(\mathbf{u})].$$
    \item[$(iv)$] If $F$ satisfies \emph{(H7)} then $$ f_k(\mathbf{x}) \;=\; \dfrac{\partial F}{\partial x_k}(\mathbf{x}), \;\;\;\;\;\;\;\forall \mathbf{x}\in \mathbb{R}^{l}.
$$
In addition, $F$ is positive on the positive cone of $\mathbb{R}^{l}$.
\end{itemize}
\end{lemma}
\begin{proof}
    See Norman and Pastor \cite[Lemmas~2.10,~2.11~and,~2.13]{Noguera}.
\end{proof}

We now present some additional results that arise due to the presence of the weight. The first one is a lemma  inspired by the work of Burchard and Hajaiej \cite{Burchard}, in which it is essential to use rearrangement in the proof of the existence of the ground state.

\begin{lemma}\label{PropH8} If $F:\mathbb{R}_{+}^{d} \to \mathbb{R}$ satisfies for any $i$, $j \in \{1,\ldots,d\}$, $i \neq j$ and $k,h >0$, the following inequality
    \begin{eqnarray}\label{F+F>F+F}
        F(\mathbf{y}+he_i +ke_j) + F(\mathbf{y}) \;\geq\; F(\mathbf{y}+he_i) + F(\mathbf{y}+ke_j), \;\;\;\;\;\;\; \mathbf{y} \in \mathbb{R}_{+}^{d}.
    \end{eqnarray}
    
Then $D(y_0, \mathbf{y}) = |y_0|^{b}\left[F(\mathbf{y}) - \sum_{i=1}^{d}F(y_i \mathbf{e}_i)\right]$ also satisfies \eqref{F+F>F+F} for all $\mathbf{y} \in \mathbb{R}_{+}^{d}$,  $y_0>0$ and $b>0$. 
\end{lemma}
\begin{proof} 
It is not difficult to see that $F(\mathbf{y}) - \sum_{i=1}^{m}F(y_i \mathbf{e}_i)$ also satisfies \eqref{F+F>F+F} and  is non-decreasing in each variable. So, $D(y_{0},\mathbf{y})$ satisfies the  conditions \eqref{F+F>F+F} if, and only if, for any $j \in \{1,\ldots,d\}$, we have 
\begin{eqnarray*}
  & &  (y_0+h)^b\left[ F(\mathbf{y} +ke_j) - \sum_{i=1}^{d}F((y_i +\delta_{ij}k) \mathbf{e}_i) \right] +(y_0)^{b}\left[F(\mathbf{y}) - \sum_{i=1}^{d}F(y_i \mathbf{e}_i)\right] \\[2mm]
 & &   \geq (y_0+h)^b\left[F(\mathbf{y}) - \sum_{i=1}^{d}F(y_i \mathbf{e}_i)\right] +(y_0)^{b}\left[ F(\mathbf{y} +ke_j) - \sum_{i=1}^{d}F((y_i +\delta_{ij}k) \mathbf{e}_i)\right],
\end{eqnarray*}
or equivalently
$$
[(y_0+h)^b - y_0^b]\times\left[ \left(F(\mathbf{y} +ke_j) - \sum_{i=1}^{d}F((y_i +\delta_{ij}k) \mathbf{e}_i)\right) - \left(F(\mathbf{y}) - \sum_{i=1}^{d}F(y_i \mathbf{e}_i) \right) \right] \geq 0.
$$
The result follows from the fact that $F(\mathbf{y}) - \sum_{i=1}^{d}F(y_i \mathbf{e}_i)$ is non-decreasing in $j$.
\end{proof}

\begin{lemma}\label{Sobolev~multiplication} \emph{\textbf{(Sobolev multiplication law)}} Let $n\geq 1$. Assume that $s,s_1,s_2$ are real numbers satisfying either
\begin{itemize}
    \item[$(i)$] $s_1+s_2 \geq 0$, $s_1$,$s_2 \geq s$ and $s_1+s_2 > s +\dfrac{n}{2}$; or 
    \item[$(ii)$] $s_1+s_2 > 0$, $s_1$,$s_2 > s$ and $s_1+s_2 \geq s +\dfrac{n}{2}$.
\end{itemize}
Then, there is a continuous multiplication map
$$
H^{s_1}(\mathbb{R}^n) \times H^{s_2}(\mathbb{R}^n) \to H^{s}(\mathbb{R}^n) 
$$
taking 
$$
(u,v) \mapsto uv $$
and satisfying the estimate
$$
\|uv \|_{H^{s}} \leq C\|u \|_{H^{s_1}}\|v \|_{H^{s_2}}.$$
\end{lemma}
\begin{proof}
    See  \cite[Corollary~3.16.]{Noguera} 
\end{proof}
\begin{corollary}\label{corollario-sobolev law} Let $n \geq 2$, $0<b<\min\left\{2, \dfrac{n}{2}\right\}$ and $n+2b < 6$. Assume $u$, $v \in H^{1}(\mathbb{R}^n)$. Then holds
$$
\| |x|^{-b} uv \|_{H^{-1}} \leq C
\| u \|_{H^{1}}
\| v \|_{H^{1}}.
$$
\end{corollary}
\begin{proof}
    If $2b \in [0,2]$, taking $s_1=0$, $s_2=1$ and $s=-1$, as in Lemma \ref{Sobolev~multiplication}, we have
    \begin{equation}\label{lemma~Sobolev~1}
        \||x|^{-b}uv \|_{H^{-1}} \leq C\||x|^{-b} u \|_{L^{2}}\|v \|_{H^{1}}.
    \end{equation}
    On the other hand, from the Hardy inequalities
    \begin{equation}\label{lemma~Sobolev~2}
        \int_{\mathbb{R}^{n}} \dfrac{|u|^{2}}{|x|^{2b}} \,dx \leq C \left( \dfrac{2}{n-2b} \right)^{2b}\|u\|_{L^{2}}^{2-2b} \| \nabla u \|_{L^{2}}^{2b}.
    \end{equation}
    Replacing \eqref{lemma~Sobolev~2} in \eqref{lemma~Sobolev~1}, we obtain the desired inequality.

    If $2b>2$, in this case we have $n=3$ and $2b<3$. Denote by $B_1$ the open unit ball in $\mathbb{R}^{3}$. Consider the sets $B_{+}^{1} = B_{1}$ and $B_{-}^{1}= \mathbb{R}^{3}-B_{1}$.

    Taking $s_1=-\dfrac{1}{2}$, $s_{2}=1$ and $s=-1$ as in Lemma \ref{Sobolev~multiplication}, we have
    \begin{equation}\label{lemma~Sobolev~3}
        \|\chi_{B_{+}^{1}} |x|^{-b}uv \|_{H^{-1}} \leq C\| \chi_{B_{+}^{1}} |x|^{-b} u \|_{H^{-1/2}}\|v \|_{H^{1}}.
    \end{equation}
    Applying  Lemma \ref{Sobolev~multiplication} again, this time taking $s_1=0$, $s_2=1$ and $s=-\dfrac{1}{2}$
    \begin{equation}\label{lemma~Sobolev~4}
\|\chi_{B_{+}^{1}} |x|^{-b}u \|_{H^{-1/2}} \leq C\||x|^{-b} \|_{L^{2}(B_{+}^{1})}\|u \|_{H^{1}}.    \end{equation}
    Note that $\||x|^{-b} \|_{L^{2}(B_{+}^{1})} < \infty$, since $2<\dfrac{3}{b}$.   Replacing \eqref{lemma~Sobolev~4} in \eqref{lemma~Sobolev~3}, we obtain
    \begin{equation}\label{lemma~Sobolev~5}
        \| \chi_{B_{+}^{1}}|x|^{-b} uv \|_{H^{-1}} \leq C
\| u \|_{H^{1}}
\| v \|_{H^{1}}.
    \end{equation}
    Finally, since $|x|^{-b}$ and $\nabla(|x|^{-b})$ are bounded by $C$ in $B_{-}^{1}$. Applying Lemma \ref{Sobolev~multiplication} directly, we have
    \begin{equation}\label{lemma~Sobolev~6}
        \| \chi_{B_{-}^{1}}|x|^{-b} uv \|_{H^{-1}} \leq C
\| u \|_{H^{1}}
\| v \|_{H^{1}}.
    \end{equation}
    \eqref{lemma~Sobolev~5} and     \eqref{lemma~Sobolev~6} conclude the last case.
\end{proof}
\begin{lemma}\label{Ademir.lemma2.15} If $n \geq 2$, $n+2b<6$, and $0< b< \min\left\{2, \dfrac{n}{2} \right\}$. Assume that the nonlinearities $f_k$ satisfy \emph{(H1)} and \emph{(H2)}, then for all $k=1,\ldots,l$, we have $|x|^{-b}f_k \in \mathcal{C}(\mathbf{H}^{1}(\mathbb{R}^{n}),\,\mathbf{H}^{-1}(\mathbb{R}^{n}))$.
\end{lemma}
\begin{proof}
   Let $(\mathbf{u}_i) \subset \mathbf{H}^{1}(\mathbb{R}^n)$ be such that $\mathbf{u}_i \to \mathbf{u}$ in $\mathbf{H}^{1}(\mathbb{R}^n)$. In particular, there exists $M>0$ such that $\|\mathbf{u}_i\|_{\mathbf{H}^{1}(\mathbb{R}^n)} \leq M$.
   Lemma \ref{Ademir.cor2.3.lem2.4},
 and Corollario \ref{corollario-sobolev law} lead to
   \begin{eqnarray*}
       \| |x|^{-b}(f_k(\mathbf{u}_i)-f_k(\mathbf{u})) \|_{\mathbf{H}^{-1}(\mathbb{R}^n)}
       &\leq&
       C \sum_{m=1}^{l}\sum_{j=1}^{l} \| |x|^{-b}(|u_{ji}|-|u_{j}|)|u_{mi}-u_{m}| \|_{H^{-1}(\mathbb{R}^n)} \\[2mm]
       &\leq&
       C \sum_{m=1}^{l}\sum_{j=1}^{l} \| u_{ji} \|_{H^{1}(\mathbb{R}^n)} \|u_{mi}-u_{m} \|_{H^{1}(\mathbb{R}^n)}\\[2mm]
       &+&C \sum_{m=1}^{l}\sum_{j=1}^{l} \|u_{j}\|_{H^{1}(\mathbb{R}^n)} \|u_{mi}-u_{m} \|_{H^{1}(\mathbb{R}^n)} \\[2mm]
       &\leq&
       C \left( M + \|u\|_{\mathbf{H}^{1}(\mathbb{R}^n)} \right) \sum_{m=1}^{l}  \|u_{mi}-u_{m} \|_{H^{1}(\mathbb{R}^n}.  \end{eqnarray*}
   We note that the right-hand side goes to $0$ as $i\to +\infty$, which implies that $|x|^{-b}f_k(\mathbf{u}_i) \to |x|^{-b}f_k(\mathbf{u})$ in $\mathbf{H}^{-1}(\mathbb{R}^{n})$.
\end{proof}

The next is a technical result which we use in the global solution analysis in the intercritical case.

  \begin{lemma}\label{lemma-inequality}
        Let $I$ an open interval with $0\in I$. Let $\alpha \in \mathbb{R}$ and $q>1$. Define $\gamma = (\beta q)^{- \frac{1}{q-1}}$ and $f(r) = \alpha - r + \beta r^q$, for $r\geq 0$. Let $G(t)$ a non-negative continuos function such that  $f \circ G \geq 0$ on $I$. Assume that  $\alpha < \left( 1 - \dfrac{1}{q} \right) \gamma$.
        \begin{itemize}
            \item[$(i)$] If $G(0) < \gamma$, then $G(t) < \gamma$ for all $t \in I$
            \item[$(ii)$] If $G(0) > \gamma$, then $G(t) > \gamma$ for all $t \in I$ 
        \end{itemize}
    \end{lemma}
\begin{proof}
    See, for instance  \cite[Lemma~5.2]{Begout}, \cite[Lemma~4.2]{Esfahani} or  \cite[Lemma~3.1]{Pastor}. 
\end{proof}

\bigskip
\section{Local and global well-posedness in \texorpdfstring{$L^{2}$}{L2} and \texorpdfstring{$H^{1}$}{H1}}

In this section, we analyse the dynamics of system \eqref{INLS} in the $L^2$ and $H^1$ frameworks. 
Using Duhamel's principle, we obtain the integral equation corresponding to the system  \eqref{INLS}:
%
\begin{equation}\label{Duhamel}
u_k(\cdot,t)=U_k(t)u_{k0}+i\int_{0}^{t}U_k(t-t') \dfrac{|x|^{-b}}{\alpha_k} f_k(u_1,\ldots,u_l)\,dt',\end{equation}
%
where $U_k$ is defined by
$$
U_k (t) \;=\; e^{i\frac{t}{\alpha_k}(\gamma_k \Delta - \beta_k)}, \;\;\;\;\;\;\;\;\;\; k=1,\ldots,l,
$$
and $e^{it \Delta}$ is the linear Schrödinger group.

We will discuss the well-posedness of system \eqref{INLS} in the  $L^2$-subcritial case and the $H^1$-subcritial  case. From now on, we will refer to the expression \emph{Well-posedness Theory} in the sense of Kato, according to the following definition:
%
%
\begin{definition}\label{WP}
The initial-value problem \eqref{INLS} is locally well-posed if for all $\mathbf{u}_0 \in \mathbf{H}^{s}(\mathbb{R}^{n})$, there exist $T>0,$ a subspace $X$ of $C\left([-T,T]; \mathbf{H}^{s}(\mathbb{R}^{n})\right) $ and  a unique solution  $\mathbf{u}$ such that
\begin{itemize}
\item[1.] $\mathbf{u}$ is the solution of the integral equation \eqref{Duhamel};
\item[2.] $\mathbf{u}\in X$ (Persistence);
\item[3.] The solution varies continuously depending  on the initial data (Continuous Dependence).
\end{itemize}
Global well-posedness requires that the same properties hold for all time $T > 0$.
\end{definition}

We prove, assuming (H1) and (H2),  the existence of a local solution to the IVP \eqref{INLS} by using the contraction mapping principle in a suitable space based on the well-known Strichartz estimates.  

\begin{proposition}\label{Strichartz}\emph{\textbf{(Strichartz`s inequality)}} Let $(q_1,r_1)$ and $(q_2,r_2)$ be two admissible pairs and $I =[-T,T]$ for some $T>0$. Then, for $k=1,\ldots,l$,
$$
\|U_k(t) f \|_{L_{\mathbb{R}}^{q_1}L_{x}^{r_1}} \leq C\| f \|_{L^{2}}
$$
and
$$
\left\| \int_{0}^{t} U_k(t-s) f(\cdot,s) \,ds \right\|_{L_{\mathbb{I}}^{q_1}L_{x}^{r_1}} \leq C\| f \|_{L_{\mathbb{I}}^{q_{2}^{'}}L_{x}^{r_{2}^{'}}},
$$
where $q_{2}^{'}$ and $r_{2}^{'}$ are the Holder conjugate of $q_2$ and $r_2$, respectively.
\end{proposition}
\begin{proof}
    See, for instance,  Cazenave \cite[Theorem~2.3.3.]{Cazenave} 
\end{proof}

\subsection{ Local existence of \texorpdfstring{$L^{2}$}{L2}-solutions}

In this section, we prove the existence of local $L^2$ solutions in the subcritical regime through  the integral equation \eqref{Duhamel} in $L^{2}$ with $n+2b<4$ and $0<b< \min\left\{2,\dfrac{n}{2}\right\}$. 
The results show here closely align with those in \cite{Hayashi} and \cite{Noguera}. For any $u_{10}, \ldots,u_{l0} \in L^{2}$, we solve \eqref{Duhamel} in the space
$$
X(I) = \mathcal{C}\cap L^{\infty}\left(I, L^2  \right) \cap S(L^2,I), \;\;\;\; n+2b<4, 
$$
for some time interval $I = [-T,T]$ with $T>0$. Here the norm in $X(I)$ is the same as in $S(L^2,I)$.

The following two lemmas are used to bound the integral term of \eqref{Duhamel}, and these are proved using Hölder's inequality in space and time variables and Strichartz's inequalities
\begin{lemma}
\label{lemmaL2}
    Assume $n+2b<4$ and $0<b<\min \left\{ 2, \dfrac{n}{2}\right\}$. Then 
    $$
    \| \,|x|^{-b} fg \|_{S'(L^2,I)} \leq C T^{\theta} \, \|f\|_{S(L^2,I)}\, \|g\|_{S(L^2,I)},
    $$
    where $I=[0,T]$, $\theta = \dfrac{4-n-2b}{4}$, and $C$ is a positive constant.
\end{lemma}
\begin{proof}
Let $R >0$ be chosen later and denote by $B_R$ the open ball of radius $R$ in $\mathbb{R}^{n}$. Consider the sets $B_{+}^{R}=B_R$ and $B_{-}^{R} = \mathbb{R}^{n}\setminus B_R$, and let $0< \epsilon < \min \left\{ \dfrac{4 - n - 2b}{2}, \; \dfrac{n+2b}{6}, \; \dfrac{n-b}{3}, \; b \right\}$. Define
$$
r=\dfrac{3n}{n-b} \;\;\; \mbox{and} \;\;\;  r_{\pm} = \dfrac{3n}{n-b \mp 3\epsilon}.
$$
From the choice of $\epsilon$, and since $n+2b<4$ and $0<b<1$, one gets that $2<r< \dfrac{2n}{n-2}$ and $2<r_{\pm}< \dfrac{2n}{n-2}$. Moreover, the following relation holds:
$$
\dfrac{1}{r'}= \dfrac{b\pm \epsilon}{n} + \dfrac{1}{r_{\pm}} + \dfrac{1}{r}.
$$
Therefore, from the Hölder inequality we obtain
\begin{eqnarray*}
    \| |x|^{-b} fg \|_{L^{r'}(B_{\pm}^{R})} &\leq& C \| |x|^{-b} \|_{L^{\frac{n}{b\pm \epsilon}}(B_{\pm}^{R})} \|f\|_{L^{r_{\pm}}}\|g\|_{L^{r}} \\[2mm]
    &\leq& C R^{\pm \epsilon}\|f\|_{L^{r_{\pm}}}\|g\|_{L^{r}}.
\end{eqnarray*}
Now, let $q$, $q_{\pm}$ be such that $(q,r)$, $(q_{+},r_{+})$ and $(q_{-},r_{-})$ are $L^{2}$-admissible pair. One can easily check that for $\theta_{\pm} = \dfrac{4-n-2b}{4} \mp \dfrac{\epsilon}{2}$, the following condition is satisfied:
$$
\dfrac{1}{q'} = \theta_{\pm} + \dfrac{1}{q_{\pm}} + \dfrac{1}{q}
$$
Then, by the Hölder inequality in the time variable, we have
$$\| |x|^{-b} fg \|_{L_{I}^{q'}L^{p'}} \leq C (T^{\theta_{+}}R^{\epsilon} + T^{\theta_{-}}R^{-\epsilon}) \|f\|_{S(L^2,I)}\, \|g\|_{S(L^2,I)}.$$
Finally, taking $R=T^{\frac{1}{2}}$, we have the desired estimate. 
\end{proof}


\begin{lemma}\label{l2-desigualdade2}
    Let $n+2b < 4$ and 
    $0<b<1$. Suppose $\mathbf{u}$, $\mathbf{u}^{'} \in \mathbf{X}(I) = S(L^2,I) \times \ldots \times S(L^2,I)$. Assume that \emph{(H1)} and \emph{(H2)} hold. Then
        $$
        \left\| \int_{0}^{t} U_{k}(t-t') \dfrac{|x|^{-b}}{\alpha_k} [f_k(\mathbf{u}) - f_k(\mathbf{u}^{'})]\, dt^{'} \right\|_{S(L^2,I)} 
        \leq CT^{\frac{4-n-2b}{4}} \left( \| \mathbf{u}\|_{\mathbf{X}(I)} + \|\mathbf{u}^{'}\|_{\mathbf{X}(I)} \right) \| \mathbf{u} - \mathbf{u}^{'}\|_{\mathbf{X}(I)}.
        $$
\end{lemma}

\begin{proof}
In view of Proposition \ref{Strichartz}, by 
Lemma \ref{Ademir.cor2.3.lem2.4}
and Lemma \ref{lemmaL2} we get
\begin{eqnarray*}
    \left\| \int_{0}^{t} U(t-t') \dfrac{|x|^{-b}}{\alpha_k} [f_k(\mathbf{u}) - f_k(\mathbf{u}')] \,dt' \right\|_{S(L^2,I)}
&\leq&
C \| |x|^{-b} [f_k(\mathbf{u}) - f_k(\mathbf{u}')] \|_{S'(L^2,I)} \\[2mm]
&\leq&
C \sum_{m=1}^{l}\sum_{j=1}^{l} \||x|^{-b}\,(|u_j| + |u_{j}^{'}|) |u_{m}-u_{m}^{'}| \|_{S'(L^2,I)} \\[2mm]
&\leq&
C T^{\theta} \sum_{m=1}^{l}\sum_{j=1}^{l} ( \|u_j\|_{S(L^2,I)} + \|u_{j}^{'}\|_{S(L^2,I)})\|u_m - u_{m}^{'}\|_{S(L^2,I)}  \\[2mm]
&\leq&
C T^{\theta} \left( \|\mathbf{u}\|_{\mathbf{X}(I)} + \|\mathbf{u}^{'}\|_{\mathbf{X}(I)} \right)\|\mathbf{u}-\mathbf{u}^{'}\|_{\mathbf{X}(I)},
\end{eqnarray*}
where $\theta = \dfrac{4-n-2b}{4}$.
\end{proof}

Now we present the proof of the first result.



\textbf{Proof of Theorem \ref{l2-existence solutions}:}
    We shall prove that for some $T>0$ the operator $\Gamma$ defined by
$$ \Gamma(\mathbf{u}) \;=\; (\Phi_1(\mathbf{u}),\ldots,\Phi_l(\mathbf{u})),$$
    where
\begin{eqnarray*}
    \Phi_k(\mathbf{u})(t) 
    \;=\; U_k(t)u_{k0} + i \int_{0}^{t} U_k(t-t') \dfrac{|x|^{-b}}{\alpha_k}f_k(\mathbf{u}) \,dt' \;\;\;\;\;\;\;\; k=1,\ldots,l.
\end{eqnarray*}
is a contraction on the complete metric space
$$
B(T,a) \;=\; 
\{ \mathbf{u} \in \mathbf{X}(I): \|\mathbf{u}\|_{\mathbf{X}(I)} := \sum_{j=1}^{l} \|u_j\|_{S(L^2,I)} \leq a \},
$$
with the metric $d(\mathbf{u},\mathbf{v}) =  \|\mathbf{u}-\mathbf{v}\|_{\mathbf{X}(I)}$.

Indeed, using Strichartz estimates Proposition \ref{Strichartz} and Lemma \ref{l2-desigualdade2}, 
we get
\begin{eqnarray*}
    \|\Phi_k(\mathbf{u})(t)\|_{S(L^2,I)} 
    &\leq& \|U_k(t)u_{k0}\|_{S(L^2,I)} + \left\| \int_{0}^{t} U_k(t-t') \dfrac{|x|^{-b}}{\alpha_k} f_k(\mathbf{u}) \,dt' \right\|_{S(L^2,I)} \\[2mm]
    &\leq& C\|u_{k0}\|_{L^2} + CT^{\frac{4-n-2b}{4}}\| \mathbf{u} \|_{\mathbf{X}(I)}^{2}.
\end{eqnarray*}
Summing the above inequality over $k$, since $\|\mathbf{u}_{0}\|_{\mathbf{L}^{2}} < r$ and $\mathbf{u} \in B(T,a)$, we have:
$$
 \|\Gamma(\mathbf{u})\|_{\mathbf{X}(I)} 
\;\leq\; Cr + CT^{\frac{4-n-2b}{4}} a^2.
$$
Let us choose $a=2Cr$ and take $T$ such that $CT^{\frac{4-n-2b}{4}} a < \dfrac{1}{4l}$. We obtain
$$
 \|\Gamma(\mathbf{u})\|_{\mathbf{X}(I)}\;\leq\; \dfrac{a}{2} + \dfrac{a}{4l} < a.
$$
Therefore $\Gamma = (\Phi_1, \ldots.,\Phi_l) : B(T,a) \to B(T,a)$ is well defined.\\

Now, we show  that $\Gamma$ is a contraction. Again using Strichartz estimates Proposition \ref{Strichartz} and Lemma \ref{l2-desigualdade2}, we have

\begin{eqnarray*}
    \|\Phi_k(\mathbf{u})(t) - \Phi_k(\mathbf{u}')(t)\|_{S(L^2,I)} 
    &\leq&  \left\| \int_{0}^{t} U_k(t-t') \dfrac{|x|^{-b}}{\alpha_k} [f_k(\mathbf{u}) - f_k(\mathbf{u}')] \,dt'\right\|_{S(L^2,I)} \\[2mm]
  &\leq&  
  CT^{\frac{4-n-2b}{4}} \left( \| \mathbf{u}\|_{\mathbf{X}(I)} + \|\mathbf{u}^{'}\|_{\mathbf{X}(I)} \right) \| \mathbf{u} - \mathbf{u}^{'}\|_{\mathbf{X}(I)} \\[2mm]
  &\leq&  
  2CT^{\frac{4-n-2b}{4}} a \| \mathbf{u} - \mathbf{u}^{'}\|_{\mathbf{X}(I)} \\[2mm]
  &\leq& 
  \dfrac{1}{2l} \| \mathbf{u} - \mathbf{u}^{'}\|_{\mathbf{X}(I)}.
  \end{eqnarray*}
Therefore, summing again, over $k$
$$
    \|\Gamma(\mathbf{u}) - \Gamma(\mathbf{u}^{'})\|_{\mathbf{X}(I)} 
\;\leq\;
\dfrac{1}{2} \| \mathbf{u} - \mathbf{u}^{'}\|_{\mathbf{X}(I)}.
$$
Finally, we will check the  continuous dependence. Again by Strichartz estimates Proposition \ref{Strichartz} and Lemma \ref{l2-desigualdade2}, we get

\begin{eqnarray*}
    \|u_k(t) - u_{k}^{'}(t)\|_{S(L^2,I)} 
    &\leq&
    \| U_k(t)(u_{k0} - u_{k0}^{'})\|_{S(L^2,I)}\\[2mm] 
    &+&
    \left\| \int_{0}^{t} U_k(t-t') \dfrac{|x|^{-b}}{\alpha_k}[f_k(\mathbf{u}) - f_k(\mathbf{u}')] \,dt'\right\|_{S(L^2,I)} \\[2mm]
   &\leq& 
    C \| u_{k0} - u_{k0}^{'} \|_{L^2}
    +
    2CT^{\frac{4-n-2b}{4}} a \| \mathbf{u} - \mathbf{u}^{'}\|_{\mathbf{X}(I)}
\end{eqnarray*}
Since $T$ is small, and summming over $k$, we get
$$\| \mathbf{u} - \mathbf{u}^{'}\|_{\mathbf{X}(I)}
\leq  \widetilde{K} \| \mathbf{u}_0 - \mathbf{u}_{0}^{'}\|_{\mathbf{X}(I)}
$$
and the result follows.
 \qed
\begin{remark}\label{l2-time}
    Note that the time $T>0$ obtained in  Theorem \ref{l2-existence solutions} depends solely on the $\mathbf{L}^{2}$ norm of the initial data. In fact, the existence time of the solution depends only on the  $\mathbf{L}^{2}$ norm of the initial data, not its position in space, that is, if $\mathbf{u}_0$, $\mathbf{v}_0 \in \mathbf{L}^{2}$ are such that $\|\mathbf{u}_0\|_{\mathbf{L}^{2}} = \|\mathbf{v}_0\|_{\mathbf{L}^{2}}$ then $T(\|\mathbf{u}_0\|_{\mathbf{L}^{2}}) = T(\|\mathbf{v}_0\|_{\mathbf{L}^{2}})$. Moreover, if $\|\mathbf{u}_0\|_{\mathbf{L}^{2}} \leq \|\mathbf{v}_0\|_{\mathbf{L}^{2}}$ then $T(\|\mathbf{u}_0\|_{\mathbf{L}^{2}}) \geq T(\|\mathbf{v}_0\|_{\mathbf{L}^{2}})$.
\end{remark}


\subsection{ Local existence of \texorpdfstring{$H^{1}$}{H1}-solutions}

 Now, we prove the existence of local $H^1$ solutions in the subcritical regime, more precisely, we study \eqref{Duhamel} in $H^{1}$ under the condition $n+2b<6$ and $0<b< \min\left\{2,\dfrac{n}{2}\right\}$. Thus, assuming that $u_{10}, \ldots,u_{l0} \in H^{1}$, we solve \eqref{Duhamel} in the space
$$
Y(I) = \mathcal{C}\cap L^{\infty}\left(I, H^1  \right) \cap S(H^1,I), \;\;\;\; n+2b<6, 
$$
for some time interval $I = [-T,T]$, with $T>0$. Here in $Y(I)$ the norm  is the same as $S(H^1,I)$.

Firstly, we recall the Sobolev embedding that will be useful to prove the following lemmas


\begin{proposition}\emph{\textbf{(Sobolev embedding)}} \label{thm:sobolev-embedding} Let $1 \leq p < + \infty$.
\begin{itemize}
    \item[$(i)$] If $n>2$ and $p<n$ then $H^{1,p}$ is continuosly embedded in $L^{r}$, where $\dfrac{1}{n} = \dfrac{1}{p} - \dfrac{1}{r}$. Moreover,
    $$
    \|f\|_{L^{r}} \leq C \| \nabla f \|_{L^{p}}.
    $$
    \item[$(ii)$] If $n=2$ then $H^{1} \subset L^{r}$ for all $r \in[2,+\infty)$. Futhermore,
    $$
    \|f\|_{L^{r}} \leq C \| f \|_{H^{1}}.
    $$
\end{itemize}
\end{proposition}
\begin{proof}
     See Bergh-Löfström \cite[Theorem~6.5.1]{Bergh}   (see also Linares-Ponce  \cite[Theorem~3.3.]{Linares-Ponce} 
 and Demenguel-Demenguel \cite[Theorem~4.18]{Demenguel} ).
\end{proof}

As in the $L^2$ case, using Hölder's inequality and Sobolev embedding, we obtain the following result

\begin{lemma}\label{H1-inequality}
    Let $n\geq 2$, $n+2b<6$ and $0<b< \min\left\{2, \dfrac{n}{2}\right\}$, then
    $$ \displaystyle     \| |x|^{-b} fg \|_{S^{'}(H^1,I)} \leq C T^{\theta(n)}  \| f\|_{S(H^1,I)} \| g \|_{S(H^1,I)},$$
where 
$$
\theta(n) = \left\{ \begin{array}{lc}
\dfrac{4-n-2b}{4}, \;\;\;\, \textrm{ if } \; n+2b < 4\\
\dfrac{6-n-2b}{4}, \;\;\;\,  \textrm{ if } \; 4 \leq n+2b <6
\end{array}\right. .
$$
\end{lemma}

\begin{proof}
    Let $R>0$ be chosen later and denote by $B_R$ the open ball of radius $R$ in $\mathbb{R}^{n}$. We consider the sets $B_{+}^{R} = B_R$ and $B_{-}^{R} = \mathbb{R}^{n}\setminus B_R$. Let $0< \eta <1$. We will divide the proof into three cases:
    \begin{itemize}
        \item[$(i)$] $4\leq n +2b < 6$

Let $\eta < \dfrac{6-n-2b}{6}$ and $0 < \epsilon < \min \left\{ \dfrac{6-n-2b}{6} - \eta, \; \eta + \dfrac{n+2b-4}{2}  \right\}$.

Define 
$$  r=\dfrac{2n}{n-2+2\eta} \;\;\; \mbox{and} \;\;\;  r_{\pm} = \dfrac{4n}{n-2b +4 -2\eta \mp 2\epsilon} .$$
From the choice of $\eta$ and  $\epsilon$, together with $4 \leq n+2b<6$ and $0<b<\min\left\{2, \dfrac{n}{2}\right\}$, one gets that $2<r< \dfrac{2n}{n-2}$ and $2<r_{\pm}<3 \leq n$. Moreover, the following relation holds:
\begin{eqnarray*}
\dfrac{1}{r'} = \dfrac{b + 1\pm \epsilon}{n} + 2 \left( \dfrac{1}{r_{\pm}} - \dfrac{1}{n} \right)  = \dfrac{b \pm \epsilon}{n} + \left( \dfrac{1}{r_{\pm}} - \dfrac{1}{n} \right) + \dfrac{1}{r_{\pm}}
\end{eqnarray*}
Therefore, from the Hölder inequality and the Sobolev embedding (Theorem \ref{thm:sobolev-embedding}), we obtain
\begin{eqnarray*} \| \nabla (|x|^{-b} fg) \|_{L^{r'}(B_{\pm}^{R})} &\leq& C \left( \| |x|^{-b} \|_{L^{\frac{n}{b\pm \epsilon}}(B_{\pm}^{R})} + \| |x|^{-b} \|_{L^{\frac{n}{b+1\pm \epsilon}}(B_{\pm}^{R})} \right) \|\nabla f\|_{L^{r_{\pm}}}\|\nabla g\|_{L^{r_{\pm}}} \\[2mm]
    &\leq& C R^{\pm \epsilon}\|\nabla f\|_{L^{r_{\pm}}}\|\nabla g\|_{L^{r_{\pm}}},
\end{eqnarray*}
and
\begin{eqnarray*} \||x|^{-b} fg \|_{L^{r'}(B_{\pm}^{R})} \leq  C R^{\pm \epsilon}\|\nabla f\|_{L^{r_{\pm}}}\| g\|_{L^{r_{\pm}}}.
\end{eqnarray*}

Now, let $q$, $q_{\pm}$ be such that $(q,r)$, $(q_{+},r_{+})$ and $(q_{-},r_{-})$ are $L^{2}$- admissible pair. We can easily check that $\theta_{\pm} = \dfrac{6-n-2b}{4} \mp \dfrac{\epsilon}{2}$, and the following condition is satisfied:
$$
\dfrac{1}{q'} = \theta_{\pm} + \dfrac{2}{q_{\pm}}
$$
then by the Hölder's inequality in the time variable, we have
$$\| |x|^{-b} fg \|_{S^{'}(H^1,I)} \leq C (T^{\theta_{+}}R^{\epsilon} + T^{\theta_{-}}R^{-\epsilon}) \|f\|_{S(H^1,I)}\, \|g\|_{S(H^1,I)}.$$
Finally, taking $R=T^{\frac{1}{2}}$, we have the desired estimate. 
        \item[$(ii)$]  $n=3$ and $0< b < \dfrac{1}{2}$.

Let $\eta < \dfrac{4-n-2b}{2} = \dfrac{1-2b}{2}$ and $0 < \epsilon < \dfrac{1-2b}{2} - \eta$. Define 
$$  r= \dfrac{6}{1+2\eta} = \dfrac{2n}{n-2+2\eta} \;\;\; \mbox{and} \;\;\;  r_{\pm} = \dfrac{3}{1-b-\eta \mp \epsilon} =  \dfrac{n}{1-b-\eta \mp \epsilon}  .$$
From the choice of $\eta$ and  $\epsilon$, together with $0 <b<\dfrac{1}{2}$, one gets that $2<r< 6=\dfrac{2n}{n-2}$ and $2<r_{\pm}<6=\dfrac{2n}{n-2}$. Moreover, the following relation holds:
\begin{eqnarray*}
\dfrac{1}{r'} = \dfrac{b + 1\pm \epsilon}{n} +  \dfrac{1}{r_{\pm}} + \left( \dfrac{1}{2} - \dfrac{1}{n} \right)  = \dfrac{b \pm \epsilon}{n}  + \dfrac{1}{r_{\pm}} + \dfrac{1}{2}
\end{eqnarray*}
Therefore, from the Hölder inequality and the Sobolev embedding, we obtain
\begin{eqnarray*} \| \nabla (|x|^{-b} fg) \|_{L^{r'}(B_{\pm}^{R})} &\leq& C \left( \| |x|^{-b} \|_{L^{\frac{n}{b\pm \epsilon}}(B_{\pm}^{R})} + \| |x|^{-b} \|_{L^{\frac{n}{b+1\pm \epsilon}}(B_{\pm}^{R})} \right) \| f\|_{W^{1,r_{\pm}}}\| g\|_{H^{1}} \\[2mm]
    &\leq& C R^{\pm \epsilon}\| f\|_{W^{1,r_{\pm}}}\|\nabla g\|_{H^{1}},
\end{eqnarray*}
and
\begin{eqnarray*} \||x|^{-b} fg \|_{L^{r'}(B_{\pm}^{R})} \leq  C R^{\pm \epsilon}\| f\|_{L^{r_{\pm}}}\| g\|_{L^{2}}.
\end{eqnarray*}

Now, let $q$, $q_{\pm}$ be such that $(q,r)$, $(q_{+},r_{+})$ and $(q_{-},r_{-})$ are $L^{2}$-admissible pair. One can easily check that $\theta_{\pm} = \dfrac{4-n-2b}{4} \mp \dfrac{\epsilon}{2}$, and the following condition is satisfied
$$
\dfrac{1}{q'} = \theta_{\pm} + \dfrac{1}{q_{\pm}}.
$$
Then by the Hölder's inequality in the time variable, we have
$$\| |x|^{-b} fg \|_{S^{'}(H^1,I)} \leq C (T^{\theta_{+}}R^{\epsilon} + T^{\theta_{-}}R^{-\epsilon}) \|f\|_{S(H^1,I)}\, \|g\|_{S(H^1,I)}.$$
Finally, taking $R=T^{\frac{1}{2}}$, we have the desired estimate.
        \item[$(iii)$]   $n=2$ and $0< b < 1$.

Consider $B_{+}^{1} = B_1$ and $B_{-}^{1} = \mathbb{R}^{n} \setminus B_1$. First, we will estimate on $B_{+}^{1}$. For this, let $0<\epsilon < \dfrac{1-b}{3}$. Define
$$  r= \dfrac{6}{1-b} , \;\;\; r_{+} = \dfrac{6}{1-b-3\epsilon}, \;\;\; \mbox{and} \;\;\;  p = \dfrac{6}{4-b}.$$

From the choice of $\epsilon$, 
$$
\dfrac{1}{r'} = \dfrac{b+ \epsilon}{2} + \dfrac{1}{r_{+}} + \dfrac{1}{p}.
$$
Hölder's inequality, Hardy's inequality, and the Sobolev embedding give 
\begin{eqnarray*} \| \nabla (|x|^{-b} fg) \|_{L^{r'}(B_{+}^{1})} &\leq&
C  \| |x|^{-b} \|_{L^{\frac{n}{b + \epsilon}}(B_{+}^{1})}  \left\| \dfrac{f}{|x|} \right\|_{L^{p}(B_{+}^{1})} \| g\|_{L^{r_{+}}(B_{+}^{1})} \\[2mm]
&+& \| |x|^{-b} \|_{L^{\frac{n}{b + \epsilon}}(B_{+}^{1})} \| \nabla f \|_{L^{p}(B_{+}^{1})} \| g\|_{L^{r_{+}}(B_{+}^{1})} \\[2mm]
&+& \| |x|^{-b} \|_{L^{\frac{n}{b + \epsilon}}(B_{+}^{1})} \| \nabla g \|_{L^{p}(B_{+}^{1})} \| f\|_{L^{r_{+}}(B_{+}^{1})}  \\[2mm]
    &\leq& C  \left( \dfrac{p}{2-p}\right) \left\| \nabla f \right\|_{L^{p}(B_{+}^{1})} \| g\|_{L^{r_{+}}(B_{+}^{1})} \\[2mm] &+& C \| \nabla f \|_{L^{p}(B_{+}^{1})} \| g\|_{L^{r_{+}}(B_{+}^{1})} + C \| \nabla g \|_{L^{p}(B_{+}^{1})} \| f\|_{L^{r_{+}}(B_{+}^{1})}  \\[2mm]
    &\leq& C \left( \dfrac{4-b}{1-b} \right) \| \nabla f \|_{L^{\frac{r}{2}}} \| g \|_{H^{1}} + C  \| \nabla g \|_{L^{\frac{r}{2}}} \| f \|_{H^{1}},
    \end{eqnarray*}
and
\begin{eqnarray*} \||x|^{-b} fg \|_{L^{r'}(B_{+}^{1})} \leq  C \| f\|_{H^1}\| g\|_{L^{\frac{r}{2}}}.
\end{eqnarray*}    
This holds since  $L^{\frac{r}{2}}(B_{+}^{1}) \subset L^{p}(B_{+}^{1})$, where $1<p<2<\dfrac{r}{2}$.

Now, let $q$, $\widetilde{q}$ be such that $(q,r)$ and $\left(\widetilde{q},\dfrac{r}{2}\right)$ are $L^{2}$-admissible pairs. One can easily check that $\theta = \dfrac{1-b}{2}$, and the following condition is satisfied:
$$
\dfrac{1}{q'} = \theta + \dfrac{1}{\widetilde{q}}.
$$
Then, by Hölder's inequality in the time variable, we have
$$\|\chi_{B_{+}^{1}}  |x|^{-b} fg \|_{S(H^1,I)} \leq CT^{\theta} \|f\|_{S(H^1,I)}\, \|g\|_{S(H^1,I)}.$$ 

We now estimate on $B_{-}^{1}$. Consider, without loss of generality,  $|I|<1$. Take $R$ such that $T< R^2 < 1$, let $0<\epsilon < \dfrac{1-b}{3}$; in this case we have the following inequality:
\begin{equation}\label{balls}
\| |x|^{-b-1} \|_{L^{\frac{n}{b-\epsilon}}(B_{-}^{1})} \leq \| |x|^{-b} \|_{L^{\frac{n}{b-\epsilon}}(B_{-}^{1})} \leq \| |x|^{-b} \|_{L^{\frac{n}{b-\epsilon}}(B_{-}^{R})} \leq CR^{-\epsilon}     
\end{equation}
Define
$$  r= \dfrac{6}{2-b} , \;\;\;  \mbox{and} \;\;\;r_{-} = \dfrac{6}{2-b+3\epsilon}.$$
From the choice of $\epsilon$, 
$$
\dfrac{1}{r'} = \dfrac{b+ \epsilon}{2} + \dfrac{1}{r_{-}} + \dfrac{1}{r}.
$$
From \eqref{balls}, Hölder's Inequality, and Sobolev embedding we have 
\begin{eqnarray*} \| \nabla (|x|^{-b} fg) \|_{L^{r'}(B_{-}^{1})} &\leq&
C  \| |x|^{-b-1} \|_{L^{\frac{n}{b - \epsilon}}(B_{-}^{1})}  \| f \|_{L^{r_{-}}} \| g\|_{L^{r}} \\[2mm]
&+& \| |x|^{-b} \|_{L^{\frac{n}{b - \epsilon}}(B_{-}^{1})} \| \nabla f \|_{L^{r_{-}}} \| g\|_{L^{r}} + \| |x|^{-b} \|_{L^{\frac{n}{b - \epsilon}}(B_{-}^{1})}   \| f\|_{L^{r_{-}}}\| \nabla g \|_{L^{r}}  \\[2mm]
    &\leq& C R^{-\epsilon}   \| f\|_{W^{1,r_{-}}}\|  g \|_{W^{1,r}},
    \end{eqnarray*}
    and
\begin{eqnarray*} \||x|^{-b} fg \|_{L^{r'}(B_{\pm}^{R})} \leq  C R^{- \epsilon}\| f\|_{L^{r_{-}}}\| g\|_{L^{r}}.
\end{eqnarray*}

Now, let $q$, $q_{-}$ be such that $(q,r)$ and $(q_{-},r_{-})$ are $L^{2}$-admissible pairs. One can easily check that $\theta = \dfrac{1-b}{2}$, and the following condition is satisfied
$$
\dfrac{1}{q'} = \theta + \dfrac{\epsilon}{2} + \dfrac{1}{q_{-}} + \dfrac{1}{q},
$$
then by Hölder's inequality in the time variable, we have
$$\|\chi_{B_{-}^{1}}   |x|^{-b} fg \|_{S(H^1,I)} \leq CT^{\theta + \frac{\epsilon}{2}}R^{-\epsilon}  \|f\|_{S(H^1,I)}\, \|g\|_{S(H^1,I)}
\leq CT^{\theta} \|f\|_{S(H^1,I)}\, \|g\|_{S(H^1,I)}.$$ 
Therefore
$$\| |x|^{-b} fg \|_{S(H^1,I)} \leq \|\chi_{B_{+}^{1}}  |x|^{-b} fg \|_{S(H^1,I)} + \|\chi_{B_{-}^{1}}  |x|^{-b} fg \|_{S(H^1,I)} \leq T^{\theta} \|f\|_{S(H^1,I)}\, \|g\|_{S(H^1,I)}.$$
    \end{itemize}
\end{proof}

Lemma \ref{H1-inequality} yields the following estimate for the integral component in the system  \eqref{Duhamel}.

\begin{lemma}\label{H1-integralpart}
    Let $n\geq 2$, 
    $0<b<\min\left\{2, \dfrac{n}{2}\right\}$ with $n+2b < 6$. Suppose $\mathbf{u}$, $\mathbf{u}^{'} \in \mathbf{Y}(I)$ and assume that \emph{(H1)} and \emph{(H2)} hold. Then
        $$
        \left\| \int_{0}^{t} U_{k}(t-t') \dfrac{|x|^{-b}}{\alpha_k} [f_k(\mathbf{u}) - f_k(\mathbf{u}^{'})]\, dt^{'} \right\|_{\mathbf{Y}(I)} 
        \leq CT^{\theta(n)} \left( \| \mathbf{u}\|_{\mathbf{Y}(I)} + \|\mathbf{u}^{'}\|_{\mathbf{Y}(I)} \right) \| \mathbf{u} - \mathbf{u}^{'}\|_{\mathbf{Y}(I)}.
        $$
\end{lemma}
\begin{proof}
    The proof is similar to  that of Lemma \ref{l2-desigualdade2}.
\end{proof}

Following the strategy of Theorem \ref{l2-existence solutions}, we use the contraction mapping theorem, together with Lemma \ref{H1-integralpart}, to establish the existence of local $H^1$  solutions.

\textbf{Proof of Theorem \ref{Existence H1}:}
    The proof is similar to  that of Theorem \ref{l2-existence solutions}.
%
\begin{remark}\label{H1-time}
    Note that the time $T>0$ obtained in Theorem \ref{Existence H1} depends  solely on the $\mathbf{H}^{1}$-norm of the initial data. In fact, the existence time of the solution depends only on the $\mathbf{H}^{1}$-norm of the initial data, not  its position in space, that is, if $\mathbf{u}_0$, $\mathbf{v}_0 \in \mathbf{H}^{1}$ satisfy $\|\mathbf{u}_0\|_{\mathbf{H}^{1}} = \|\mathbf{v}_0\|_{\mathbf{H}^{1}}$ then $T(\|\mathbf{u}_0\|_{\mathbf{H}^{1}}) = T(\|\mathbf{v}_0\|_{\mathbf{H}^{1}})$. Moreover, if $\|\mathbf{u}_0\|_{\mathbf{H}^{1}} \leq \|\mathbf{v}_0\|_{\mathbf{H}^{1}}$ then $T(\|\mathbf{u}_0\|_{\mathbf{H}^{1}}) \geq T(\|\mathbf{v}_0\|_{\mathbf{H}^{1}})$.
\end{remark}

Thus, we conclude this section with an important result concerning blow-up alternative in $H^1$.
\begin{corollary}\label{H1: Blow-up}
    Let $\mathbf{u} \in \mathbf{H}^{1}$ be a solution of system \eqref{INLS} as in Theorem \eqref{Existence H1}. There exists $T_{\ast}, T^{\ast} \in (0,\infty]$ such that the local solutions can be extended to the interval $(-T_{\ast}, T^{\ast})$; moreover, if $T^{\ast} <\infty$ (respect. $T_{\ast}<\infty$), then
    $$
    \lim_{t\to T^{\ast}}\|\mathbf{u}(t)\|_{\mathbf{H}^{1}} = \infty \;\;\;\;\; (\mbox{resp.,}\; \lim_{t\to -T_{\ast}}\|\mathbf{u}(t)\|_{\mathbf{H}^{1}} = \infty).
    $$
\end{corollary}
\begin{proof}
    Define $$T^{\ast} = \sup\{T; \;\;\mbox{  there exists a solution of } \eqref{INLS}\mbox{ in }[0,T]\}.$$
    From Theorem \ref{Existence H1}, there exists a unique solution $\mathbf{u}$ of \eqref{INLS} corresponding to the initial data $\mathbf{u}_{0}$ such that $\mathbf{u} \in \mathcal{C}([0,T^{\ast}), \mathbf{H}^{1}) \cap \mathbf{X}([0,T^{\ast}))$. If $T^{\ast}<\infty$, assume by contradiction that there exists a sequence $t_k \uparrow T^{\ast}$ and $M>0$ such that $\|\mathbf{u}(t_k)\|_{\mathbf{H}^{1}} \leq M$ for all $k$. By Remark \ref{H1-time}, for $\mathbf{u}^{k}(0) = \mathbf{u}(t_k)$, there exists $T(M)>0$ and a unique solution $\mathbf{u}^{k} \in \mathbf{X}([0,T(M)])$. By the uniqueness of the solution, it follows that  $\mathbf{u}(t)$ exists on $[0, t_k+ T(M)]$. Since $t_k \uparrow T^{\ast}$ and $T(M)>0$, we have $t_k+ T(M) > T^{\ast}$ for $k$ sufficiently large, which contradicts the maximality of $T^{\ast}$. The other case is analogous.
\end{proof}

\subsection{Global existence of \texorpdfstring{$L^{2}$}{L2}-solutions}

The key to obtaining global solutions is to establish an a priori estimate in $L^2$  for the local solution, which follows from conservation of the charge. 

The following derivation is formal, but the process can be made rigorous by approximating with sufficiently regular solutions and passing to the limit, or by using the strategy in \cite{Ozawa}.

\begin{lemma}\label{Q(u0)}
    Under assumptions \emph{(H3)} and \emph{(H4)}, the charge of the system \eqref{INLS}, defined as
    \begin{eqnarray*}
    Q(\mathbf{u}(t))\;=\;\sum_{k=1}^{l} {\alpha_k \sigma_k} \int_{\mathbb{R}^n} |u_k(x,t)|^2\, dx
    \end{eqnarray*}
    %
    is conserved.
\end{lemma}

\begin{proof}
    Multiplying \eqref{INLS} by $\overline{u_k}$, we obtain
    $$
\dfrac{\partial}{\partial t}Q(\mathbf{u}(t)) \;=\; 2 \int_{\mathbb{R}^n} |x|^{-b} \mbox{Im}\; \sum_{k=1}^{l} \sigma_k f_k(\mathbf{u})\overline{u}_{k} =0,
$$
where the last equality follows from  Remark \ref{remarkH3H4}.
\end{proof}
As an immediate consequence of the previous Lemma, we have the following

\begin{theorem}
    Let $1\leq n \leq 3$, $0<b<\min\left\{2,\dfrac{n}{2} \right\}$ with $n+2b <4$. Assume that \emph{(H1)-(H4)} hold. Then, for any $\mathbf{u}_0 \in \mathbf{L}^{2}$, system \eqref{INLS} has a unique global solution $\mathbf{u} \in \mathbf{X}(\mathbb{R})$. Moreover,
    $$
 Q(\mathbf{u}(t)) \;=\; Q(\mathbf{u}_0), \;\;\;\;\;\; \forall t \in \mathbb{R}. 
$$
\end{theorem}
\begin{proof}
Similary to Corollary \ref{H1: Blow-up},     define $$T^{\ast} = \sup\{T; \;\;\mbox{  there exists a solution of } \eqref{INLS}\mbox{ in }[0,T]\}.$$
    From Theorem \ref{l2-existence solutions}, there exists a unique solution $\mathbf{u}$ of \eqref{INLS} corresponding to the initial data $\mathbf{u}_{0} \in \mathbf{L}^{2}$ such that $\mathbf{u} \in \mathcal{C}([0,T^{\ast}), \mathbf{L}^{2}) \cap \mathbf{X}([0,T^{\ast}))$. Now, if $T^{\ast}<\infty$, assume by contradiction that there exists a sequence $t_k \uparrow T^{\ast}$. Since $Q(\cdot)$ and $\|\cdot\|_{\mathbf{L}^{2}}$ are equivalent norms on $\mathbf{L}^{2}$ and by Lemma \ref{Q(u0)}, we have 
    $\|\mathbf{u}(t_k)\|_{\mathbf{L}^{2}} \leq C Q(\mathbf{u}_0)$ for some positive constant $C$. By Remark \ref{l2-time}, for $\mathbf{u}^{k}(0) = \mathbf{u}(t_k)$ there exists $T_0:=T(Q(\mathbf{u}_0))>0$ and a unique solution $\mathbf{u}^{k} \in \mathbf{X}([0,T_0])$. From the uniqueness of the solution, it follows that  $\mathbf{u}(t)$ exists on $[0, t_k+ T_0]$. Since $t_k \uparrow T^{\ast}$ and $T_0>0$, we have $t_k+ T_0 > T^{\ast}$  for large $k$, contradicting the maximality of $T^{\ast}$. The case for negative times is analogous. 
\end{proof}

\subsection{Global existence of \texorpdfstring{$H^{1}$}{H1}-solutions}

The key to obtaining global solutions is to establish an a priori estimate in $H^1$ for the local solution, which follows from the conservation of energy (see Corollary \ref{H1: Blow-up}).

We begin by recalling an important tool for minimizing the Weinstein functional, which is fundamental in the study of INLS: a Gagliardo-Nirenberg type inequality, proved by Farah \cite{Farah} and by Combet and Genoud \cite{Combet}. This work utilizes the following version
\begin{theorem}\label{GN}
    Let $n+2b<6$ and $0 < b< \min \left\{2,n \right\}$. Then the inequality
    \begin{equation}
  \int_{\mathbb{R}^n}|x|^{-b}|u(x)|^{3}\,dx
  \leq
  C_{GN}\|\nabla u\|_{L^2}^{\frac{n+2b}{2}}\|u\|_{L^2}^{\frac{6-(n+2b)}{2}}
\end{equation}
    holds for some positive constant $C_{GN}$.
\end{theorem}
\begin{proof}
    See, for instance, Farah \cite[Theorem~1.2]{Farah}.
\end{proof}

This result enables us to estimate the functional $P$ defined in \eqref{K,L,P}. 

\begin{proposition}\label{P<=C0QK}
    Assume that \emph{(H1)-(H3)} and \emph{(H5)} hold. If $\mathbf{u} \in \mathbf{H}^{1}$, then $P(\mathbf{u})$, defined in \eqref{K,L,P}, 
    is finite; moreover, the following inequality holds:
    $$
\left| P(\mathbf{u}) \right| \leq C_0 Q(\mathbf{u})^{\frac{6-n-2b}{4}} K(\mathbf{u})^{\frac{n+2b}{4}},
$$
where $C_0$ is a positive constant depending on $\alpha_k$, $\sigma_k$, and $\gamma_k$, for $k=1,\ldots,l$. 
\end{proposition}

\begin{proof}
First, note that from the Gagliardo-Nirenberg inequality (Theorem \ref{GN}), for each $k=1,\ldots,l$,
\begin{eqnarray*}
  \int_{\mathbb{R}^n}|x|^{-b}|u_k(x)|^{3}\,dx
  &\leq&
  C_{GN}\|\nabla u_k\|_{L^2(\mathbb{R}^n)}^{\frac{n+2b}{2}}\|u_k\|_{L^2(\mathbb{R}^n)}^{\frac{6-(n+2b)}{2}} \\[2mm]
&\leq&
C \gamma_k^{-\frac{n+2b}{4}} (\alpha_k \sigma_k)^{\frac{n+2b-6}{4}} Q(\mathbf{u})^{\frac{6-n-2b}{2}} K(\mathbf{u})^{\frac{n +2b}{2}}.   
\end{eqnarray*}
Using Lemma \ref{Ademir.lem2.10.2.11.2.13}, we obtain
\begin{eqnarray*}
|P(\mathbf{u})| 
&\leq& \int |x|^{-b}|\mbox{Re}F(\mathbf{u})|\,dx  \leq
C \sum_{k=1}^{l} \int|x|^{-b}|u_k(x)|^{3}\,dx \\[2mm]
&\leq&
C_0 Q(\mathbf{u})^{\frac{6-n-2b}{2}} K(\mathbf{u})^{\frac{n +2b}{2}}.    
\end{eqnarray*}
\end{proof}

The conservation of energy for system \eqref{INLS} is established in the following lemma.

\begin{lemma}\label{E(u0)}
    If condition \emph{(H3)} holds, then the energy associated with \eqref{INLS}, given by 
    %
%
\begin{eqnarray*}
E(\mathbf{u}(t))= 
\sum_{k=1}^{l} \gamma_{k} \int_{\mathbb{R}^N} |\nabla u_k(x,t)|^2\, dx +
\sum_{k=1}^{l} \beta_{k} \int_{\mathbb{R}^N} | u_k(x,t)|^2\, dx -
2\emph{Re} \int_{\mathbb{R}^N} |x|^{-b} F(\textbf{u}(t))\,dx    
\end{eqnarray*} 
%
is a conserved quantity.
\end{lemma}

\begin{proof}
We already know that
\begin{eqnarray}\label{partial: u}
 \dfrac{\partial}{\partial t} \int_{\mathbb{R}^n} | u_k(t) |^{2} \,dx 
\; = \;
2 \mbox{Re} \int_{\mathbb{R}^n}  u_k \; \partial_{t} \overline{u}_k \, dx. 
\end{eqnarray}
    Now, for the gradient term
\begin{eqnarray}\label{partial: nabla u}
   \dfrac{\partial}{\partial t}\int_{\mathbb{R}^n} |\nabla u_k|^2 \, dx
   &=&
      \dfrac{\partial}{\partial t} \left( - \int_{\mathbb{R}^n} \Delta u_k\;  \overline{u}_k \, dx \right) \nonumber \\[2mm]
      &=&
      - \int_{\mathbb{R}^n} \Delta (\partial_t u_k) \;  \overline{u}_k \, dx      - \int_{\mathbb{R}^n} \Delta u_k\;  \partial_{t} \overline{u}_k \, dx \\[2mm]
    &=&
      - \int_{\mathbb{R}^n} \partial_t u_k \;  \Delta (\overline{u}_k) \, dx      - \int_{\mathbb{R}^n} \Delta u_k\;  \partial_{t} \overline{u}_k \, dx \nonumber \\[2mm]
      &=&
      -2 \mbox{Re} \int_{\mathbb{R}^n}  \Delta (u_k) \, \partial_t \overline{u_k} \;  dx. \nonumber
      \end{eqnarray}
On the other hand, note that
$$
   \dfrac{\partial}{\partial t} \int_{\mathbb{R}^n} |x|^{-b} F(\mathbf{u}) \;dx
   \;=\;
    \int_{\mathbb{R}^n} |x|^{-b} \sum_{k=1}^{l} \left[\dfrac{\partial F(\mathbf{u})}{\partial u_k} \partial_t u_k + \dfrac{\partial F(\mathbf{u})}{\partial \overline{u}_k} \partial_t \overline{u}_k \right] \, dx.
$$
Taking de real part
\begin{eqnarray}\label{partial: F}
   \dfrac{\partial}{\partial t} \mbox{Re} \int_{\mathbb{R}^n} |x|^{-b} F(\mathbf{u}) \;dx
   &=&
    \mbox{Re} \int_{\mathbb{R}^n} |x|^{-b} \sum_{k=1}^{l} \left[ \dfrac{\partial F(\mathbf{u})}{\partial u_k} \partial_t u_k + \dfrac{\partial F(\mathbf{u})}{\partial \overline{u}_k} \partial_t \overline{u}_k \right] \, dx \nonumber \\[2mm]
    &=&
    \mbox{Re} \int_{\mathbb{R}^n} |x|^{-b} \sum_{k=1}^{l} \left[ \overline{\dfrac{\partial F(\mathbf{u})}{\partial u_k} \partial_t u_k} + \dfrac{\partial F(\mathbf{u})}{\partial \overline{u}_k} \partial_t \overline{u}_k \right] \, dx \nonumber  \\[2mm]
    &=&
\mbox{Re} \int_{\mathbb{R}^n} |x|^{-b} \sum_{k=1}^{l} \left[ \overline{\dfrac{\partial F(\mathbf{u})}{\partial u_k}} \partial_t \overline{u}_k + \dfrac{\partial F(\mathbf{u})}{\partial_t \overline{u}_k} \partial_t \overline{u}_k \right] \, dx \nonumber  \\[2mm]
    &=&
\mbox{Re} \int_{\mathbb{R}^n} |x|^{-b} \sum_{k=1}^{l} \left[ \dfrac{\partial F}{\partial \overline{u}_k}(\mathbf{u}) + \overline{\dfrac{\partial F}{\partial u_k}}(\mathbf{u}) \right] \partial_t \overline{u}_k  \, dx \nonumber  \\[2mm]
&=&
\mbox{Re} \int_{\mathbb{R}^n} |x|^{-b} \sum_{k=1}^{l} f_k(\mathbf{u}) \partial_t\overline{u}_k  \, dx.
      \end{eqnarray}
The last equality is from $(H3)$. Therefore, from the \eqref{INLS}, \eqref{partial: u}, \eqref{partial: nabla u} and \eqref{partial: F}:
\begin{eqnarray*}
    \dfrac{\partial}{\partial t} E(\mathbf{u}(t))
    &=& 
    \sum_{k=1}^{l} \gamma_k \dfrac{\partial}{\partial t}\int_{\mathbb{R}^n} | \nabla u_k|^2 \, dx 
    + \sum_{k=1}^{l} \beta_k \dfrac{\partial}{\partial t}\int_{\mathbb{R}^n} |u_k|^2 \, dx
    - 2 \dfrac{\partial}{\partial t} \mbox{Re} \int_{\mathbb{R}^n} |x|^{-b} F(\mathbf{u}) \;dx \\[2mm]
    &=&
    2 \mbox{Re} \sum_{k=1}^{l} \int_{\mathbb{R}^n} (- \gamma_k \Delta u_k + \beta_k u_k - |x|^{-b}f_k(\mathbf{u}) \partial_t \overline{u}_k  \;dx \\[2mm]
    &=&
    2 \mbox{Re} \sum_{k=1}^{l} i \alpha_k \int_{\mathbb{R}^n} \partial_t u_k \partial_t \overline{u}_k  \;dx = 
    2 \mbox{Re} \;i \sum_{k=1}^{l}  \alpha_k \int_{\mathbb{R}^n}|\partial_t u_k|^{2}   \;dx =0. 
\end{eqnarray*}
Therefore, the energy is conserved.
\end{proof}


Next, we define the Weinstein functional
$$
    J(\mathbf{u}) = \dfrac{Q(\mathbf{u})^{\frac{6-n-2b}{4}} K(\mathbf{u})^{\frac{n+2b}{4}}}{|P(\mathbf{u})|}, 
$$
and the real number
\begin{equation}\label{xi0}
    \xi_0 = \inf\{J(\mathbf{u}):\; \mathbf{u} \in \mathbf{H}^{1},\;\; P(\mathbf{u})\neq 0\},
\end{equation}
where $Q$ is the charge defined in \eqref{Q(u)}.

\begin{remark}
    By Proposition \ref{P<=C0QK}, if  \emph{(H1)-(H3)} and \emph{(H5)} hold, then,  $\xi_0$ is a positive constant such that
\begin{equation}\label{P<QK}
    |P(\mathbf{u})| \leq \dfrac{1}{\xi_0} Q(\mathbf{u})^{\frac{6-n-2b}{4}} K(\mathbf{u})^{\frac{n+2b}{4}}.
\end{equation}
Indeed, if $P(\mathbf{u}) \neq 0$, then $|P(\mathbf{u})| >0$. Thus,
$$
0 < \dfrac{1}{C_0} \leq J(\mathbf{u}),
$$
and the conclusion follows.
\end{remark}

We now proceed to prove the existence of global $H^1$ solutions for system \eqref{INLS}.

\begin{remark}\label{K: bound}
    From Corollary \ref{H1: Blow-up}, we have that if $\| u(t) \|_{\mathbf{H}^{1}}$ is bounded, then the solution of \eqref{INLS} is global. On the other hand,  from Lemma \ref{Q(u0)} $\|u(t)\|_{\mathbf{L}^{2}}$ is already bounded. Therefore, the strategy for proving the existence of global $H^{1}$-solutions for \eqref{INLS} when $n \geq 2$ and $n+2b < 6$ is to show that $K(\textbf{u}(t))$ is also bounded.
\end{remark}

The following theorem is a non-sharpened version of the Theorem \ref{thrm: Global}; however, it serves as a first criterion for a global solution, and its proof is the inspiration for the sharper result.

\begin{theorem}\label{global-solution}
    Let $n \geq 2$, $n+2b<6$ with 
    $0<b<\min\left\{2, \dfrac{n}{2}\right\}$. Assume that conditions\emph{(H1)-(H5)} hold.
    \begin{itemize}
        \item[$(i)$] If $2 \leq n \leq 3$ and $n+2b<4$, then for any $\mathbf{u}_0 \in \mathbf{H}^1$, system \eqref{INLS} has a unique solution $\textbf{u} \in \textbf{Y}(\mathbb{R})$.
        \item[$(ii)$] If $n=3$ and $2b=1$, then for any $\mathbf{u}_0 \in \mathbf{H}^{1}$ satisfying 
        $$
        2 Q(\textbf{u}_0)^{\frac{1}{2}} < \xi_0,
        $$
        system \eqref{INLS} has a unique solution $\textbf{u} \in \mathbf{Y}(\mathbb{R})$.
        \item[$(iii)$] If \( 3 \leq n \leq 5 \), then for any \( \mathbf{u}_0 \in \mathbf{H}^{1} \) satisfying
\begin{equation}\label{QK<C}
    Q(\mathbf{u}_0)^{1-s_c}K(\mathbf{u}_0)^{s_c} < \left(\dfrac{\xi_0}{s_c+2} \right)^{2}
\end{equation}
and
\begin{equation}\label{QE<C}
    Q(\mathbf{u}_0)^{1-s_c}E(\mathbf{u}_0)^{s_c} < \left(\dfrac{s_c}{s_c+2} \right)^{s_c}\left(\dfrac{\xi_0}{s_c+2} \right)^{2}, \quad \text{where } s_c=\dfrac{n+2b-4}{2},
\end{equation}
    
        system \eqref{INLS} has a unique solution $\textbf{u} \in \mathbf{Y}(\mathbb{R})$.    
    \end{itemize}
\end{theorem}

\begin{proof}
    From Remark \ref{K: bound}, it suffices to obtain an a \textit{priori} bound for $K(\textbf{u})$. For $(i)$, using \eqref{P<QK} and Young's inequality, we can write for any $\epsilon>0$: $\alpha \beta \leq \epsilon \alpha^{p} + \dfrac{1}{q(\epsilon p)^{\frac{q}{p}}} \beta^{q}$, where $\dfrac{1}{p} + \dfrac{1}{q} = 1$.

In our case 
$$
p = \dfrac{4}{n+2b}>1 \;\;\; \mbox{and} \;\;\; q =\dfrac{4}{4-n-2b}
$$
Thus,
$$
2|P(\textbf{u})| \leq \epsilon K(\textbf{u}) + C(\epsilon) Q(\textbf{u})^{\frac{6-n-2b}{4-n-2b}}.
$$
Using the last inequality, the conservation of energy Lemma \ref{E(u0)}, the conservation of charge Lemma \ref{Q(u0)} and the fact that $-L(\textbf{u}) \leq 0$ we obtain an a \textit{priori} bound for $K(\textbf{u}).$ Indeed
\begin{eqnarray*}
    K(\textbf{u}) = E(\textbf{u}_0) - L(\textbf{u}) +2P(\textbf{u}) 
    \leq  E(\textbf{u}_0) +2|P(\textbf{u})| \leq   E(\textbf{u}_0) + \epsilon K(\textbf{u}) + C(\epsilon) Q(\textbf{u})^{\frac{6-n-2b}{4-n-2b}}.
\end{eqnarray*}
Therefore, if $0 < \epsilon < 1$, then
$$
K(\textbf{u})
\leq
\dfrac{1}{1-\epsilon} \left[ E(\textbf{u}_0)  +  C(\epsilon) Q(\textbf{u}_0)^{\frac{6-n-2b}{4-n-2b}} \right].
$$

For case $(ii)$,  we have  $n+2b =4$ ($n=3$ and $2b=1$), thus  \eqref{P<QK} can be rewritten as $$2|P(\textbf{u})| \leq K(\textbf{u}) \cdot \dfrac{2}{\xi_0} Q(\textbf{u}_0)^{\frac{1}{2}}.$$
Similarly to the previous argument, we have
$$
K(\textbf{u}) \leq E(\textbf{u}_0) + K(\textbf{u}) \cdot \dfrac{2}{\xi_0} Q(\textbf{u}_0)^{\frac{1}{2}}.
$$
Therefore,
$$
K(\textbf{u}) \leq \left[ 1 - \dfrac{2}{\xi_0} Q(\textbf{u}_0)^{\frac{1}{2}} \right]^{-1} E(\textbf{u}_0).
$$
This bound is valid since $1 - \dfrac{2}{\xi_0} Q(\textbf{u}_0)^{\frac{1}{2}} > 0$, which is equivalent to the condition $2 Q(\textbf{u}_0)^{\frac{1}{2}} < \xi_0$.

Finally, for case $(iii)$, we first note that
$$
K(\textbf{u}) \leq E(\textbf{u}_0) + \dfrac{2}{\xi_0} Q(\textbf{u}_0)^{\frac{6-n-2b}{4}} K(\textbf{u})^{\frac{n+2b}{4}}.
$$
Applying Lemma \ref{lemma-inequality} with
    $$
\alpha = E(\textbf{u}_0), \;\;\; \beta = \dfrac{2}{\xi_0} Q(\textbf{u})^{\frac{6-n-2b}{4}},\;\;\;\;
q = \dfrac{n+2b}{4}, \;\;\;\; \mbox{and}\;\;\;\;
G(t) = K(\mathbf{u}(t)). 
$$
we find that $$
\gamma \;=\; (\beta q)^{-\frac{1}{q-1}}
\;=\;
\left( \dfrac{2\xi_0}{n+2b}\right)^{\frac{4}{n+2b-4}} Q(\textbf{u}_0)^{- \frac{6-n-2b}{n+2b-4}}
$$
It is straightforward to verify that $\alpha < \left( 1 - \dfrac{1}{q} \right) \gamma$ is equivalent to \eqref{QE<C} and $G(0)< \gamma$ is equivalent to \eqref{QK<C}, as desired. 
\end{proof}

\bigskip

\section{Existence of ground state solutions}

Building on the  Theorem \ref{global-solution}, we will establish a sharp criterion for the dichotomy of global existence versus finite-time blow-up. This criterion, formulated in terms of the ground state of the corresponding elliptic system, for this reason, this section is devoted to proving the existence of ground state solutions as stated in Theorem \ref{GS-existence}.
From now on, we will assume that conditions (H1)-(H8) hold. 
First, let us introduce
the \emph{Weinstein functional}, for the elliptic system. Given $\boldsymbol{\psi} \in \mathbf{H}^{1}$ define
\begin{equation}\label{J}
    J(\boldsymbol{\psi}) := \dfrac{\mathcal{Q}(\boldsymbol{\psi})^{\frac{6-n-2b}{4}} K(\boldsymbol{\psi})^{\frac{n+2b}{4}}}{P(\boldsymbol{\psi})}; \;\;\;\;\;\; P(\boldsymbol{\psi}) \neq 0,
\end{equation}
where 
$$
K(\boldsymbol{\psi}) :=  \sum_{k=1}^{l} \gamma_k \|\nabla \psi_k\|_{L^2}^{2}, \;\;\;\;
\mathcal{Q}(\boldsymbol{\psi}) :=\sum_{k=1}^{l}  b_k\|\psi_k\|_{L^2}^{2}, \;\;\;\; \mbox{and} \;\;\;\;
P(\boldsymbol{\psi}) :=\int_{\mathbb{R}^{n}} |x|^{-b} F(\boldsymbol{\psi}).
$$
\begin{remark} Note that we can rewrite the action functional $I$, defined in \eqref{I=SumH1-P}, as 
\begin{equation}\label{I}
    I(\boldsymbol{\psi}) = \dfrac{1}{2}[K(\boldsymbol{\psi}) + \mathcal{Q}(\boldsymbol{\psi})] -P (\boldsymbol{\psi}).
\end{equation}
\end{remark}
\begin{remark}\label{QinG(1,0)}
    If $b_k=\alpha_k\sigma_k$ i.e. $\omega=1$ and $\beta_k=0$, for all $k$,  we have the functions $Q$ and $\mathcal{Q}$ are iquals.
\end{remark}
\begin{remark} Similary to Propostion \ref{P<=C0QK}, the following inequality holds
\begin{equation}\label{P<CQK}
|P(\mathbf{u})| \leq  C_0 \mathcal{Q}(\mathbf{u})^{\frac{6-n-2b}{2}} K(\mathbf{u})^{\frac{n +2b}{2}}.    
\end{equation} 
\end{remark}


We start showing that these functionals are indeed Fréchet differentiable. Here, the primes denote their Fréchet derivatives.

\begin{lemma}\label{K',Q',P'} If $\mathbf{g} \in \mathbf{H}^{1}$. Then
$$
K'(\boldsymbol{\psi})(\mathbf{g}) = 2 \sum_{k=1}^{l} \gamma_k \int_{\mathbb{R}^{n}} \nabla \psi_k\nabla g_k \, dx, \;\;\;\;
\mathcal{Q}'(\boldsymbol{\psi})(\mathbf{g})  = 2 \sum_{k=1}^{l}  b_k \int_{\mathbb{R}^{n}} \psi_k g_k \, dx,$$
and
$$
P'(\boldsymbol{\psi})(\mathbf{g})  = \sum_{k=1}^{l} \int_{\mathbb{R}^{n}} |x|^{-b} f_k(\boldsymbol{\psi})g_k \,dx.
$$    
\end{lemma}
\begin{proof}
   In view of our assumptions, the proof is quite standard, and we therefore omit the details.
\end{proof}

Lemma \ref{K',Q',P'} implies that $I$ is Fréchet differentiable and that its critical points (in the sense of Definition \ref{weak-solutions}) are precisely the solutions of \eqref{PDE-GroundState}.


In the following lemma we obtain Pohozaev-type  identities which are satisfied by any solution of \eqref{PDE-GroundState}.

\begin{lemma}\label{Pohozaev}\emph{\textbf{(Pohozaev-type  identities)}}\label{P=2I, K = (n+2b)I and Q=(6-n-2b)I}
    Let $\boldsymbol{\psi}$ be a solution of \eqref{PDE-GroundState}. Then
    \begin{eqnarray}
        P(\boldsymbol{\psi}) &=& 2 I(\boldsymbol{\psi}), \label{P=2I} \\[2mm]
        K(\boldsymbol{\psi}) &=& (n+2b) I(\boldsymbol{\psi}), \label{k=(n+2b)I} \\[2mm]
        \mathcal{Q}(\boldsymbol{\psi}) &=& (6-n-2b) I(\boldsymbol{\psi}). \label{Q=(6-n-2b)I} 
    \end{eqnarray}
\end{lemma}

\begin{proof}
    We first note that by letting $g_k=\psi_k$, for $k=1,\ldots,l$ in \eqref{weak} we obtain
    $$
    \gamma_{k}  \|\nabla \psi_k \|_{L^{2}}^{2} + b_k \| \psi_k \|_{L^{2}}^{2} \;=\;  \int_{\mathbb{R}^{n}} |x|^{-b} f_k(\boldsymbol{\psi})\psi_{k} \,dx.
    $$
    Summing over $k$ and using Lemma \ref{Ademir.lem2.10.2.11.2.13},
    we have
    \begin{equation}\label{K+Q=3P}
        K(\boldsymbol{\psi}) + \mathcal{Q}(\boldsymbol{\psi}) = 3 P(\boldsymbol{\psi}).
    \end{equation}
    Therefore,
    $$
    I(\boldsymbol{\psi}) = \dfrac{1}{2}[K(\boldsymbol{\psi}) + \mathcal{Q}(\boldsymbol{\psi})] - P(\boldsymbol{\psi}) = \dfrac{1}{2}P(\boldsymbol{\psi}),
    $$
 and \eqref{P=2I} follows.

To prove  \eqref{k=(n+2b)I}, define $(\delta_{\lambda} f)(x) = f\left( \dfrac{x}{\lambda} \right)$. Then the function $h \mapsto h(\lambda) = I(\delta_{\lambda} \psi)$ has a critical point at $\lambda =1$, or equivalently,
$$
h'(\lambda) \;=\; \left.\dfrac{d}{d \lambda}\right|_{\lambda=1} I(\delta_{\lambda} \boldsymbol{\psi}) =0.
$$
Note that
\begin{itemize}
    \item[$(i)$] $K(\delta_{\lambda} \boldsymbol{\psi} ) = \lambda^{n-2} K (\boldsymbol{\psi})$ so $K'(\delta_{\lambda} \boldsymbol{\psi}) = (n-2) \lambda^{n-3} K (\boldsymbol{\psi})$.
    \item[$(ii)$] $\mathcal{Q}(\delta_{\lambda} \boldsymbol{\psi} ) = \lambda^{n} \mathcal{Q} (\boldsymbol{\psi})$ so $\mathcal{Q}'(\delta_{\lambda} \boldsymbol{\psi} ) = n \lambda^{n-1} \mathcal{Q} (\boldsymbol{\psi})$.
    \item[$(iii)$] $P(\delta_{\lambda} \boldsymbol{\psi} ) = \lambda^{n-b} P (\boldsymbol{\psi})$ so $P'(\delta_{\lambda} \boldsymbol{\psi} ) = (n-b) \lambda^{n-b-1} P (\boldsymbol{\psi})$.
\end{itemize}
Thus,
\begin{equation}\label{h'=0}
h'(1) = \dfrac{n-2}{2}K( \boldsymbol{\psi}) + \dfrac{n}{2} \mathcal{Q}(\boldsymbol{\psi}) -(n-b)P(\boldsymbol{\psi}) =0.    
\end{equation}
On the other hand, using \eqref{P=2I} in \eqref{I}, we have
\begin{equation}\label{K+Q=6I}
    K(\boldsymbol{\psi}) +\mathcal{Q}(\boldsymbol{\psi}) = 6I(\boldsymbol{\psi}).
\end{equation}
Substituting \eqref{K+Q=6I} and \eqref{P=2I} into \eqref{h'=0}, we obtain
$$
3n I(\boldsymbol{\psi}) - K(\boldsymbol{\psi}) -2(n-b)I(\boldsymbol{\psi}) =0,
$$
This implies that \eqref{k=(n+2b)I} holds. Finally, substituting \eqref{k=(n+2b)I} into \eqref{K+Q=6I} yields \eqref{Q=(6-n-2b)I}.
\end{proof}

\begin{remark}\label{Q>0}
Since $\mathcal{Q}(\boldsymbol{\psi}) > 0$ for any $\boldsymbol{\psi} \neq \mathbf{0}$, it follows from \eqref{Q=(6-n-2b)I} that the system \eqref{k=(n+2b)I} has non-trivial solutions only if $6-n-2b > 0$. In addition, $\mathcal{Q}(\boldsymbol{\psi})$ remains constant on $\mathcal{G}(\omega, \boldsymbol{\beta})$, and $\boldsymbol{\psi}$ is a ground state if and only if $\mathcal{Q}(\boldsymbol{\psi})$ is minimal.
\end{remark}

We now prove the existence of at least one ground state solution for \eqref{PDE-GroundState}. Our strategy is to minimize the Weinstein-type functional $J$, which first requires the following preliminaries.
\begin{lemma}\label{C subset P}
Assume $n \geq 2$ with $n+2b < 6$ and define the set 
\[
\mathcal{P} := \{ \boldsymbol{\psi} \in \mathbf{H}^{1} \mid P(\boldsymbol{\psi}) > 0 \}.
\]
Then:
\begin{itemize}
    \item[$(i)$] $\mathcal{C} \subset \mathcal{P}$.
    \item[$(ii)$] $\xi_1 := \inf \, \{ J(\boldsymbol{\psi}) \mid \boldsymbol{\psi} \in \mathcal{P} \} > 0$.
    \end{itemize}
\end{lemma}
\begin{proof}
    For $(i)$: Let $  \boldsymbol{\psi} \in \mathcal{C}$. Then $ \boldsymbol{\psi} \neq 0$ and $P(\boldsymbol{\psi}) = 2I(\boldsymbol{\psi}) = \dfrac{2}{6-n-2b}\mathcal{Q}(\boldsymbol{\psi}) >0$. Therefore, $\boldsymbol{\psi} \in \mathcal{P}$.

    Now, for $(ii)$: By the Gagliardo-Nirenberg inequality Theorem \ref{GN} and Lemma \ref{Ademir.lem2.10.2.11.2.13}
    \begin{eqnarray*}
        |P(\boldsymbol{\psi})| &\leq& C \sum_{k=1}^{l} \int_{\mathbb{R}^{n}} |x|^{-b} |\psi_{k}|^3 dx \leq  C \sum_{k=1}^{l} \| \nabla u_k \|_{L^2}^{\frac{n+2b}{2}} \| u_k \|_{L^2}^{\frac{6 -n-2b}{2}}\\[2mm]
        &\leq&   C \gamma_{k}^{-\left(\frac{n+2b}{2}\right)} (\alpha_k \sigma_{k})^{-\left(\frac{6-n-2b}{4} \right)}K(\boldsymbol{\psi})^{\frac{n+2b}{4}} \mathcal{Q}(\boldsymbol{\psi})^{\frac{6-n-2b}{4}}  \\[2mm]
        &\leq& C_1 K(\boldsymbol{\psi})^{\frac{n+2b}{4}} \mathcal{Q}(\boldsymbol{\psi})^{\frac{6-n-2b}{4}}.
    \end{eqnarray*}
Thus, 
$$
0 < \dfrac{1}{C_1} \leq J(\boldsymbol{\psi}).
$$
Therefore, $\xi_1 >0$.
\end{proof}

The following lemma establishes a direct relationship between the functionals $J$ and $I$.

\begin{lemma}\label{J=I}
    Assume $2 \leq n \leq 5$ with $n+2b<6$. If $\boldsymbol{\psi}$ is a non-trivial solution of \eqref{PDE-GroundState}, then 
    $$
    J(\boldsymbol{\psi}) = \dfrac{(n+2b)^{\frac{n+2b}{4}} (6-n-2b)^{\frac{6-n-2b}{4}} }{2} I(\boldsymbol{\psi})^{\frac{1}{2}}.
    $$
    In particular, any non-trivial solution $\boldsymbol{\psi} \in \mathcal{P}$ of \eqref{PDE-GroundState} that minimizes of $J$ is a ground state of \eqref{PDE-GroundState}.
\end{lemma}
\begin{proof}
Using \eqref{Q=(6-n-2b)I}, \eqref{k=(n+2b)I} and \eqref{P=2I}, we obtain
    \begin{eqnarray*}
        J(\boldsymbol{\psi}) &=& \dfrac{\mathcal{Q}(\boldsymbol{\psi})^{\frac{6-n-2b}{4}} K(\boldsymbol{\psi})^{\frac{n+2b}{4}}}{P(\boldsymbol{\psi})} = \dfrac{(n+2b)^{\frac{n+2b}{4}} (6-n-2b)^{\frac{6-n-2b}{4}}}{2 I(\boldsymbol{\psi})}I(\boldsymbol{\psi})^{\frac{3}{2}} \\
        &=& \dfrac{(n+2b)^{\frac{n+2b}{4}} (6-n-2b)^{\frac{6-n-2b}{4}} }{2} I(\boldsymbol{\psi})^{\frac{1}{2}}.
    \end{eqnarray*}
\end{proof}

In what follows, given any non-negative function $f \in \mathbf{H}^{1}$, we denote by $f^{\ast}$ its symmetric-decreasing rearrangement (see, for instance, \cite{Leoni} or \cite{Lieb}).
Also, for any $\lambda > 0$, we define $(\delta_{\lambda}g)(x) = g\left(\dfrac{x}{\lambda}\right)$. Thus, for any $\mathbf{g} = (g_1,\ldots,g_l) \in \mathbf{H}^{1}$, we set $\mathbf{g}^{\ast} = (g_{1}^{\ast},\ldots,g_{l}^{\ast})$ and $(\delta_{\lambda}\mathbf{g})(x) = \left( g_{1}\left( \dfrac{x}{\lambda} \right), \ldots, g_{l}\left( \dfrac{x}{\lambda} \right) \right)$.
Regarding scaling transformations and symmetric-decreasing rearrangement, the functionals introduced in this section satisfy the following properties.

\begin{lemma}\label{Q,K,P,Q',K',P'}
    Let $n\geq 2$, $0<b< \min\left\{2, \dfrac{n}{2}\right\}$, $n+2b<6$, and $a, \lambda >0$. If $\boldsymbol{\psi} \in \mathcal{P}$ and $\mathbf{g} \in C_{0}^{\infty} (\mathbb{R}^{n})^{l}$, we have
\begin{itemize}
    \item[$(i)$] $\mathcal{Q}(a \delta_{\lambda} \boldsymbol{\psi}) = a^{2} \lambda^{n} \mathcal{Q}(\boldsymbol{\psi})$.
    \item[$(ii)$] $K(a \delta_{\lambda} \boldsymbol{\psi}) = a^{2} \lambda^{n-2} K(\boldsymbol{\psi})$.
    \item[$(iii)$] $P(a \delta_{\lambda} \boldsymbol{\psi}) = a^{3} \lambda^{n-b} P(\boldsymbol{\psi})$.
    \item[$(iv)$] $K'(a \delta_{\lambda} \boldsymbol{\psi})\mathbf{g} = a \lambda^{n-2} K'(\boldsymbol{\psi})(\delta_{\lambda^{-1}}\mathbf{g} )$.
    \item[$(v)$] $\mathcal{Q}'(a \delta_{\lambda} \boldsymbol{\psi})\mathbf{g} = a \lambda^{n} \mathcal{Q}'(\boldsymbol{\psi})(\delta_{\lambda^{-1}}\mathbf{g} )$.
    \item[$(vi)$] $P'(a \delta_{\lambda} \boldsymbol{\psi})\mathbf{g} = a^{2} \lambda^{n-b} P'(\boldsymbol{\psi})(\delta_{\lambda^{-1}}\mathbf{g} )$.
\end{itemize}
In addition, if $\psi_k$ is non-negative for $k=1,\ldots,l$, then 
\begin{itemize}
    \item[$(vii)$] $\mathcal{Q}(\boldsymbol{\psi}^{\ast}) = \mathcal{Q}(\boldsymbol{\psi})$.
    \item[$(viii)$] $K(\boldsymbol{\psi}^{\ast}) \leq K(\boldsymbol{\psi})$.
    \item[$(ix)$] $P(\boldsymbol{\psi}^{\ast}) \geq P(\boldsymbol{\psi})$.
\end{itemize}
\end{lemma}
\begin{proof}
Properties $(i)$--$(vi)$ follow immediately from the definitions, taking into account Lemma \ref{K',Q',P'}. Properties $(vii)$ and $(viii)$ follow from the facts that (see, for instance, \cite[Theorems 16.10 and 16.17]{Leoni})
\[
\| \psi_{k}^{\ast} \|_{{L}^{2}} = \| {\psi_k} \|_{{L}^{2}} \quad \text{and} \quad \| \nabla \psi_{k}^{\ast} \|_{{L}^{2}} \leq \| \nabla {\psi_k} \|_{{L}^{2}}, \quad k=1,\ldots,l.
\]
Let $\widetilde{F} = \widetilde{F}_1 + \ldots + \widetilde{F}_m$, where each $\widetilde{F}_s(\mathbf{y}) = F_s(\mathbf{y}) - \sum_{i=1}^{m}F_s(y_i \mathbf{e}_i)$ for $s=1,\ldots,m$ is supermodular and $\widetilde{F}_s$ is nondecreasing in each variable, since $F_s$ is supermodular. From Lemma \ref{PropH8}, this is a necessary and sufficient condition for 
\begin{equation}\label{G}
    \int_{\mathbb{R}^{n}} G(|x|, \boldsymbol{\psi}) \, dx \leq \int_{\mathbb{R}^{n}} G(|x|, \boldsymbol{\psi}^{\ast}) \, dx
\end{equation}
to hold. See \cite[Theorem 1]{Burchard} and \cite[Proposition 3.1]{Hajaiej}.

From $(H5)$ and by Lemma \ref{Ademir.lem2.10.2.11.2.13},
$F( \psi_i \mathbf{e_i}) = C \psi_{i}^{3}$, where $C= F(\mathbf{e_i}) \geq 0$. Thus, we way apply the same Theorem to obtain
\begin{equation}\label{Gi}
     \int_{\mathbb{R}^{n}} |x|^{-b}F( \psi_i \mathbf{e_i}) \leq \int_{\mathbb{R}^{n}} |x|^{-b}F( \psi_i^{\ast} \mathbf{e_i}), \;\;\;\;\; i=1,\ldots,n.
\end{equation}
Therefore, from \eqref{G} and \eqref{Gi}
\begin{eqnarray*}
    \int_{\mathbb{R}^{n}} |x|^{-b} F( \boldsymbol{\psi}) &=&
    \int_{\mathbb{R}^{n}} |x|^{-b} \widetilde{F} (\boldsymbol{\psi})
    + \sum_{i=1}^{m} \int_{\mathbb{R}^{n}} |x|^{-b} F_(\psi_i \mathbf{e}_i) \\[2mm]
    &=&
     \int_{\mathbb{R}^{n}} D(|x|, \boldsymbol{\psi}) +   \sum_{i=1}^{m} \int_{\mathbb{R}^{n}} |x|^{-b} F_(\psi_i \mathbf{e}_i) \\[2mm]
     &\leq& 
      \int_{\mathbb{R}^{n}} D(|x|, \boldsymbol{\psi}^{\ast}) +  \sum_{i=1}^{m} \int_{\mathbb{R}^{n}} |x|^{-b} F_(\psi_i \mathbf{e}_i) \\[2mm]
      &=&
      \int_{\mathbb{R}^{n}} |x|^{-b} F (\boldsymbol{\psi}^{\ast})
    - \sum_{i=1}^{m} \left[ \int_{\mathbb{R}^{n}} |x|^{-b}  F_(\psi_i^{\ast} \mathbf{e}_i)  - \int_{\mathbb{R}^{n}} |x|^{-b} F_(\psi_i \mathbf{e}_i)  \right] \\[2mm]
    &\leq&  \int_{\mathbb{R}^{n}} |x|^{-b} F (\boldsymbol{\psi}^{\ast}).
\end{eqnarray*}
\end{proof}

The next lemma describes the behavior of $J$ under scaling and symmetric-decreasing rearrangement.

\begin{lemma}\label{J,J'}
    Under the assumptions of Lemma \ref{Q,K,P,Q',K',P'}:
    \begin{itemize}
        \item[$(i)$] $J(a \delta_{\lambda} \boldsymbol{\psi}) = J(\boldsymbol{\psi})$.
        \item[$(ii)$] $J(|\boldsymbol{\psi}|) \leq J(\boldsymbol{\psi})$, where $|\boldsymbol{\psi}| = (|\psi_1|,\ldots, |\psi_l|)$.
        \item[$(iii)$] $J'(a \delta_{\lambda} \boldsymbol{\psi})\mathbf{g} = a^{-1} J'(\boldsymbol{\psi})(\delta_{\lambda^{-1}} \mathbf{g})$.
        
        In addition, if $\psi_k$ is non-negative, for $k=1,\ldots,l$, then
        \item[$(iv)$] $J(\boldsymbol{\psi}^{\ast}) \leq J(\boldsymbol{\psi})$.
    \end{itemize}
\end{lemma}
\begin{proof}
    The proofs of $(i)$, $(iii)$, and $(iv)$ are immediate consequences of Lemma \ref{Q,K,P,Q',K',P'}. For $(ii)$, we must use assumption (H6).  
\end{proof}

Equipped with the preceding lemmas, we are now able to present our main result on the existence of ground states. We adopt the following conventions: a function $\boldsymbol{\psi} \in \mathbf{H}^{1}$ is said to be positive (resp., non-negative), denoted $\boldsymbol{\psi} > 0$ ($\boldsymbol{\psi} \geq 0$), if each of its components is positive (resp., non-negative). Furthermore, $\boldsymbol{\psi}$ is radially symmetric if all of its components are radially symmetric functions. Here, we establish a more complete version of Theorem \ref{GS-existence}.

\begin{theorem}\emph{ \textbf{(Existence of ground state solutions)}}\label{Existence-GS} Assume that conditions \emph{(H1) - (H8)} hold. For $n \geq 2$, $n+2b<6$, and $0 <b < \min\left\{2, \dfrac{n}{2}\right\}$, the infimum
\begin{eqnarray}\label{xi_1}
    \xi_1 \;=\; \inf_{\boldsymbol{\psi} \in \mathcal{P}} J(\boldsymbol{\psi})
\end{eqnarray}
    is attained at a function $\boldsymbol{\psi}_0 \in \mathcal{P}$ such that:
\begin{itemize}
    \item[$(i)$] $\boldsymbol{\psi}_0$ is non-negative and radially symmetric function.
    \item[$(ii)$] There exist $t_0 >0$ and $\lambda_0 >0$ such that $\boldsymbol{\psi} = t_0 \delta_{\lambda_0} \boldsymbol{\psi}_0$ is a positive ground state solution of \eqref{PDE-GroundState}. In addition, if $\widetilde{\boldsymbol{\psi}}$ is any ground state of \eqref{PDE-GroundState} then,
   \begin{equation}\label{xi_1=Q}
    \xi_1 \;=\; 
\dfrac{(n+2b)^{\frac{n+2b}{4}} (6-n-2b)^{\frac{4-n-2b}{4}}}{2} \mathcal{Q}(\widetilde{\boldsymbol{\psi}})^{\frac{1}{2}}.
\end{equation}
\end{itemize} 
\end{theorem}

\begin{proof}
    Let $(\boldsymbol{\psi}_j) \subset \mathcal{P}$ be a minimizing sequence for \eqref{xi_1}, i.e.,
$$
\lim_{j \to +\infty} J(\boldsymbol{\psi}_j) = \xi_1.
$$
Replacing $\boldsymbol{\psi}_j$ by $|\boldsymbol{\psi}_{j}|^{\ast}$ via Lemma \ref{J,J'} we may assume that $\boldsymbol{\psi}_j$ consists of radially symmetric and non-increasing functions in $\mathbf{H}^{1}$.

Define $\widetilde{\boldsymbol{\psi}}_{j} = t_j \delta_{\lambda_j} \boldsymbol{\psi}_j$, where
$$
t_j \;=\; \dfrac{\mathcal{Q}(\boldsymbol{\psi}_j)^{\frac{n-2}{4}}}{K(\boldsymbol{\psi}_j)^{\frac{n}{4}}}
\;\;\;\;\;\;\;
\mbox{and}
\;\;\;\;\;\;\;
\lambda_j \;=\; \left( \dfrac{K(\boldsymbol{\psi}_j)}{\mathcal{Q}(\boldsymbol{\psi}_j)} \right)^{\frac{1}{2}}.
$$
This choice ensures that
$$
K(\widetilde{\boldsymbol{\psi}}_{j}) = \mathcal{Q}(\widetilde{\boldsymbol{\psi}}_{j})=1 
\;\;\;\;\;\;\;
\mbox{and}
\;\;\;\;\;\;\;
J(\widetilde{\boldsymbol{\psi}}_{j}) = J(\boldsymbol{\psi}_j).
$$
Thus,
$$
\dfrac{1}{P(\widetilde{\boldsymbol{\psi}}_{j})}
\;=\;
J(\widetilde{\boldsymbol{\psi}}_{j})
\;=\;
J(\boldsymbol{\psi}_j) \rightarrow \xi_1 >0.
$$

Since $K(\widetilde{\boldsymbol{\psi}}_{j}) = \mathcal{Q}(\widetilde{\boldsymbol{\psi}}_{j}) =1$, the sequence  $(\widetilde{\boldsymbol{\psi}}_{j})$ is bounded in $\mathbf{H}_{rd}^{1}$.
Here, $\mathbf{H}^{1}_{rd}$ denotes the subspace of $\mathbf{H}^1(\mathbb{R}^n)$ consisting of radially symmetric functions, that is,
\[
\mathbf{H}^{1}_{rd}
:= \{\, u \in \mathbf{H}^1(\mathbb{R}^n) \;:\; u(x) = u(|x|) \,\}.
\]
Since $\mathbf{H}^{1}_{rd}$ is a reflexive space, there exists a subsequence, still denoted by $(\widetilde{\boldsymbol{\psi}}_{j})$, and $\boldsymbol{\psi}_0 \in \mathbf{H}_{rd}^{1}$ such that
\begin{eqnarray*}
    & & \widetilde{\boldsymbol{\psi}}_{j} \rightharpoonup \boldsymbol{\psi}_0 \;\; \mbox{in} \;\; \mathbf{H}^{1}.\\[2mm]
    & & \widetilde{\boldsymbol{\psi}}_{j} \rightarrow \boldsymbol{\psi}_0 \;\; \mbox{in} \;\; \mathbf{L}^{r}(B(0,R)) \;\; \mbox{where} \;\; 2 \leq r < 2^{\ast} \;\; \mbox{for any } R.\\[2mm]
    & & \widetilde{\boldsymbol{\psi}}_{j} \rightarrow \boldsymbol{\psi}_0 \;\; \mbox{a.e. in} \;\; \mathbb{R}^n.
\end{eqnarray*}
Note that $\displaystyle \lim_{j \to + \infty} P(\widetilde{\boldsymbol{\psi}}_j) = P(\boldsymbol{\psi}_0)$. Indeed, by Lemma \ref{Ademir.lem2.10.2.11.2.13},
\begin{eqnarray*}
    |P(\widetilde{\boldsymbol{\psi}_j}) - P(\boldsymbol{\psi}_0)| 
    &\leq&
    \int_{\mathbb{R}^{n}} |x|^{-b} |F(\widetilde{\boldsymbol{\psi}_j}) - F(\boldsymbol{\psi}_0)| \, dx \\[2mm]
    &\leq& C \sum_{m=1}^{l}\sum_{k=1}^{l}
    \int_{\mathbb{R}^{n}} |x|^{-b}\left(     
    |\widetilde{\psi}_{kj}|^{2}  +     |{\psi}_{k0}|^{2} \right)     |\widetilde{\psi}_{mj} - \psi_{m0} | \, dx \\[2mm]
&\leq& C \sum_{m=1}^{l}\sum_{k=1}^{l} \left\{ 
\left( \int_{B(0,R)^{C}} |x|^{-b \gamma'}  \right)^{\frac{1}{\gamma'}} (\|\widetilde{\psi}_{kj}\|_{L^{\gamma}}^{2} + \|{\psi}_{k0}\|_{L^{\gamma}}^{2}  ) \|\widetilde{\psi}_{mj} - \psi_{m0} \|_{L^{\gamma}} \right\} \\[2mm]
&+& C \sum_{m=1}^{l}\sum_{k=1}^{l} \left\{ 
\left( \int_{B(0,R)} |x|^{-b r'}  \right)^{\frac{1}{r'}} (\|\widetilde{\psi}_{kj}\|_{L^{r}}^{2} + \|{\psi}_{k0}\|_{L^{r}}^{2}  ) \|\widetilde{\psi}_{mj} - \psi_{m0} \|_{L^{r}} \right\}
    \end{eqnarray*}
We used Holder's inequality, where
$$
1 = \dfrac{1}{\gamma'} + \dfrac{3}{\gamma} = \dfrac{1}{r'} + \dfrac{3}{r}.
$$
Let $\epsilon < \min\{b,2-b\}$. Taking $\gamma' = \dfrac{n}{b-\epsilon}$ and   $r' = \dfrac{n}{b+\epsilon}$, we have $ 2 \leq \gamma \leq 2^{\ast}$ and $ 2 \leq r \leq 2^{\ast}$.

We deduce that 
$$ P(\boldsymbol{\psi}_0) = \displaystyle \lim_{j \to + \infty} P(\widetilde{\boldsymbol{\psi}}_j) = \dfrac{1}{\xi_1} >0,$$
which means that $\boldsymbol{\psi}_0 \in \mathcal{P}$.

    On the other hand, the lower semi-continuity of the weak convergence gives
$$
K(\boldsymbol{\psi}_0) \leq \liminf_{j} K(\widetilde{\boldsymbol{\psi}}_{j}) \;=\;1
\;\;\;\;\;
\mbox{and}
\;\;\;\;\;
\mathcal{Q}(\boldsymbol{\psi}_0) \leq \liminf_{j} \mathcal{Q}(\widetilde{\boldsymbol{\psi}}_{j}) \;=\;1.
$$
Therefore,
$$
\xi_1 \leq J(\boldsymbol{\psi}_0) = \dfrac{ \mathcal{Q}(\boldsymbol{\psi}_0)^{\frac{6-n-2b}{4}}
K(\boldsymbol{\psi}_0)^{\frac{n+2b}{4}}}{P(\boldsymbol{\psi}_0)} 
\leq \dfrac{1}{P(\boldsymbol{\psi}_0)} = \xi_1.
$$
From the above inequality, we conclude that $$ J(\boldsymbol{\psi}_0) = \xi_1 \;\;\;\;\;
\mbox{and}
\;\;\;\;\; K(\boldsymbol{\psi}_0) = \mathcal{Q}(\boldsymbol{\psi}_0) =1.
 $$

The last assertion, combined with the weak convergence $\widetilde{\boldsymbol{\psi}}_j \rightharpoonup \boldsymbol{\psi}_0$ in $\mathbf{H}^{1}$, implies that $\widetilde{\boldsymbol{\psi}}_j \to \boldsymbol{\psi}_0$ strongly in $\mathbf{H}^{1}$, thus establishing part $(i)$ of the theorem.

For part $(ii)$, we note that for sufficiently small  $t$ and any $\mathbf{u} \in \mathbf{H}^{1}$, we have $(\boldsymbol{\psi}_0 + t \mathbf{u}) \in \mathcal{P}$. Since $\boldsymbol{\psi}_0$ is a minimizer of $J$ on $\mathcal{P}$, it follows that
\begin{eqnarray*}
    0 &=& 
    \left. \dfrac{d}{dt}\right|_{t=0} J(\boldsymbol{\psi}_0 + t \mathbf{u}) \\[2mm]
    &=&
    \left( \dfrac{6-n-2b}{4}\right)J(\boldsymbol{\psi}_0) \dfrac{\mathcal{Q}'(\boldsymbol{\psi}_0)(\mathbf{u})}{\mathcal{Q}(\boldsymbol{\psi}_0)} +
    \left( \dfrac{n+2b}{4}\right)J(\boldsymbol{\psi}_0) \dfrac{K'(\boldsymbol{\psi}_0)(\mathbf{u})}{K(\boldsymbol{\psi}_0)} - J(\boldsymbol{\psi}_0) \dfrac{P'(\boldsymbol{\psi}_0)(\mathbf{u})}{P(\boldsymbol{\psi}_0)}.
\end{eqnarray*}
Since $K(\boldsymbol{\psi}_0) = \mathcal{Q}(\boldsymbol{\psi}_0) = 1$ and $\dfrac{1}{P(\boldsymbol{\psi}_0)} = \xi_1$, it follows that
\begin{equation}\label{k' + cQ' = CP'}
    K'(\boldsymbol{\psi}_0)(\mathbf{u})
    + \left( \dfrac{6-n-2b}{n+2b} \right)\mathcal{Q}'(\boldsymbol{\psi}_0)(\mathbf{u})
    \;=\; \dfrac{4 \xi_1}{n+2b}P'(\boldsymbol{\psi}_0)(\mathbf{u}).
\end{equation}
Next, define $\boldsymbol{\psi} = t_0 \delta_{\lambda_0} \boldsymbol{\psi}_0$ and $\mathbf{v} = \delta_{\lambda_{0}^{-1}}\mathbf{u}$. By Lemma \ref{Q,K,P,Q',K',P'}, we obtain
\begin{eqnarray*}
    I'(\boldsymbol{\psi})(\mathbf{u})
    &=&
    \dfrac{1}{2} [K'(\boldsymbol{\psi})(\mathbf{u}) + \mathcal{Q}'(\boldsymbol{\psi})(\mathbf{u})] - P'(\boldsymbol{\psi})(\mathbf{u}) \\
    &=&
    \dfrac{t_0}{2} [\lambda_{0}^{n-2} K'(\boldsymbol{\psi}_0)(\mathbf{v}) +\lambda_{0}^{n} \mathcal{Q}'(\boldsymbol{\psi}_0)(\mathbf{v})] - t_{0}^{2}\lambda_{0}^{n-b} P'(\boldsymbol{\psi}_0)(\mathbf{v}) \\[2mm]
    &=&
    \dfrac{t_0 \lambda_{0}^{n-2}}{2} [ K'(\boldsymbol{\psi}_0)(\mathbf{v}) +\lambda_{0}^{2} \mathcal{Q}'(\boldsymbol{\psi}_0)(\mathbf{v}) - t_{0}\lambda_{0}^{2-b} P'(\boldsymbol{\psi}_0)(\mathbf{v})] .
\end{eqnarray*}

Choosing $\lambda_{0} \;=\; \left(\dfrac{6-n-2b}{n+2b}\right)^{\frac{1}{2}}$, and $t_0 \;=\; \dfrac{4\xi_1}{n+2b} \left(\dfrac{6-n-2b}{n+2b} \right)^{\frac{b-2}{2}}$, we obtain $I'(\boldsymbol{\psi}) = 0$, then $\boldsymbol{\psi}$ is a solution of the problem.

Now from Lemmas \ref{J,J'} and \ref{J=I}, we have $\boldsymbol{\psi}$ is also a critical point of $J$ with $J(\boldsymbol{\psi}) = J(\boldsymbol{\psi}_0)$. Since $\boldsymbol{\psi}_0$ is a minimizer of $J$, so is $\boldsymbol{\psi}$.

Another application of Lemma \ref{J=I} shows that $\boldsymbol{\psi}$ is a ground state of \eqref{PDE-GroundState}. To see that $\boldsymbol{\psi}$ is positive, we note that
$$
\Delta \psi_k - \dfrac{b_k}{\gamma_k} \psi_k = - \dfrac{|x|^{-b}}{\gamma_k}f_k(\boldsymbol{\psi}) \leq 0.
$$
Since $\gamma_k >0$, $\psi_k$ are non-negative and, $f_k$ satisfies (H7).

Therefore, by the strong maximum principle \cite[Theorem 3.5]{Gilbarg}, we obtain the positivity of $\boldsymbol{\psi}$.

Finally, we prove \eqref{xi_1=Q}. Indeed, if $\boldsymbol{\psi}$ is as in part $(ii)$, then \eqref{Q=(6-n-2b)I} and Lemma \ref{Q,K,P,Q',K',P'} imply 

\begin{eqnarray*}
    \xi_1
    \;=\;
    J(\boldsymbol{\psi})
    \;=\;
    \dfrac{(n+2b)^{\frac{n+2b}{4}} (6-n-2b)^{\frac{6-n-2b}{4}}}{2} I(\boldsymbol{\psi})^{\frac{1}{2}}
    \;=\;
    \dfrac{(n+2b)^{\frac{n+2b}{4}} (6-n-2b)^{\frac{4-n-2b}{4}}}{2} \mathcal{Q}(\boldsymbol{\psi})^{\frac{1}{2}}.
\end{eqnarray*}
Therefore, if  $\widetilde{\boldsymbol{\psi}} \in \mathcal{G}(\omega,0)$, then Remark \ref{Q>0} yields
$$     \xi_1 \;=\; 
\dfrac{(n+2b)^{\frac{n+2b}{4}} (6-n-2b)^{\frac{4-n-2b}{4}}}{2} \mathcal{Q}(\widetilde{\boldsymbol{\psi}})^{\frac{1}{2}}.
$$
\end{proof}

Theorem \ref{Existence-GS} yields the sharp constant for the inequality \eqref{P<CQK}. Specifically, we have

\begin{corollary}
    Let $2 \leq n \leq 5$ with $n+2b<6$. The inequality 
    $$
    P(\mathbf{u}) \leq C_{op} \mathcal{Q}(\mathbf{u})^{\frac{6-n-2b}{4}} K(\mathbf{u})^{\frac{n+2b}{4}}
    $$
    holds for any $\mathbf{u} \in \mathcal{P}$, with
    $$
    C_{op} = \dfrac{2 (6-n-2b)^{\frac{n+2b-4}{4}}}{(n+2b)^{\frac{n+2b}{4}}}\cdot \dfrac{1}{\mathcal{Q}(\boldsymbol{\psi})^{\frac{1}{2}}},
    $$
    where $\boldsymbol{\psi} \in \mathcal{G}(\omega,\boldsymbol{\beta})$.
\end{corollary}





\bigskip

\section{Global solutions versus blow-up}

Here we study the dichotomy global existence versus blow-up, improve the Theorem \ref{global-solution}, in order to obtain a sharp result.

\subsection{ Global existence in  \texorpdfstring{$H^{1}$}{H1}}

Theorem \ref{global-solution} established that solutions of system \eqref{INLS} are global in $H^1$ for $4 \leq n+2b < 6$ with $0<b< \min\left\{2,\dfrac{n}{2}\right\}$, provided the initial data is sufficiently small. We now determine the precise smallness required. This is achieved by using a specific set of ground states to derive sharp sufficient conditions for global existence.

The ground states of interest are those with $b_k = \alpha_k \sigma_k$, that is,  those satisfying the system
\begin{equation}\label{PDE-GS}
    -\gamma_k \Delta \psi_k + \alpha_k\sigma_k \psi_k = |x|^{-b}f_k(\boldsymbol{\psi}), \;\;\;\;\;\; k=1,\ldots,l.
\end{equation}
\begin{remark}
    In view of Theorem \ref{Existence-GS}, ground states for \eqref{PDE-GS} do exist. In addition, they can be seen as elements of the set $\mathcal{G}(1, \mathbf{0})$.
\end{remark}

 The sharp criterion for global well-posedness is formulated using these ground states. Specifically,     parts $(ii)$ and $(iii)$ of Theorem \ref{global-solution} can be reformulated as Theorem \ref{thrm: Global}.


\textbf{Proof of Theorem \ref{thrm: Global}:}
First, note that
by Remark \ref{QinG(1,0)}, $Q = \mathcal{Q}$ on $\mathcal{G}(1, \mathbf{0
})$. The result follows by applying the Theorem \ref{global-solution}, where $$\xi_{0}^{2}=\xi_{1}^{2} = \dfrac{(n+2b)^{\frac{n+2b}{2}} (6-n-2b)^{\frac{4-n-2b}{2}}}{4}Q(\boldsymbol{\psi}).$$ 
    \begin{itemize}
        \item[$(i)$] holds directly from Theorem \ref{global-solution}.
        \item[$(ii)$]  In case $n+2b=4$, by Theorem \ref{global-solution}, the result holds if 
        $$
        Q(\mathbf{u}_0) < \dfrac{\xi_{1}^{2}}{4} = Q(\boldsymbol{\psi}).
        $$
        \item[$(iii)$] Another application of Theorem \ref{global-solution}, together with the Pohozaev identities from Lemma \ref{Pohozaev}, yields the results:
        \begin{eqnarray*}
            \left(\dfrac{\xi_1}{s_c +2}\right)^{2} &=&
            \left(\dfrac{2 \xi_1}{n+2b}\right)^{2}
            = 
            \left(\dfrac{n+2b}{6-n-2b}\right)^{\frac{n+2b-4}{2}}Q(\boldsymbol{\psi}) 
            =
            \left(\dfrac{n+2b}{6-n-2b}\right)^{s_c} Q(\boldsymbol{\psi}) \\[2mm]
            &=&
            Q(\boldsymbol{\psi})^{1-s_c} \left(\dfrac{n+2b}{6-n-2b} Q(\boldsymbol{\psi})\right)^{s_c}
            = 
            Q(\boldsymbol{\psi})^{1-s_c}K(\boldsymbol{\psi})^{s_c}
        \end{eqnarray*}
        and
        \begin{eqnarray*}
            \left( \dfrac{s_c}{s_c +2} \right)^{s_c} \left(\dfrac{\xi_1}{s_c+2}\right)^{2} &=&
            \dfrac{s_{c}^{s_c}}{(s_c+2)^{s_c+2}} \xi_{1}^{2} \\[2mm]
            &=& \dfrac{4(n+2b-4)^{\frac{n+2b-4}{2}}}{(n+2b)^{\frac{(n+2b)}{2}}} \dfrac{(n+2b)^{\frac{n+2b}{2}} (6-n-2b)^{\frac{4-n-2b}{2}}}{4}Q(\boldsymbol{\psi}) \\[2mm]
            &=& \left(\dfrac{n+2b-4}{6-n-2b}\right)^{s_c} Q(\boldsymbol{\psi}) 
            =
            Q(\boldsymbol{\psi})^{1-s_c} \left(\dfrac{n+2b-4}{6-n-2b} Q(\boldsymbol{\psi})\right)^{s_c} \\[2mm] &=&Q(\boldsymbol{\psi})^{1-s_c}  \mathcal{E}(\boldsymbol{\psi})^{s_c}.            
        \end{eqnarray*}
    \end{itemize}
\qed

    


\subsection{Virial identity}

In order to prove the sharp condition of Theorem \ref{thrm: Global}, we will introduce the \emph{Virial} with a smooth cut-off function.
\begin{theorem}\label{V,V',V''}
    Let $n\geq 2$, $0 < b< \min\left\{2, \dfrac{n}{2}\right\}$ with $n+2b<6$. Assume $\mathbf{u}_0 \in \mathbf{H}^{1}$, and let $\mathbf{u}$ be the corresponding solution given by Theorem \ref{Existence H1}. Assume $\varphi \in C_{0}^{\infty}(\mathbb{R}^{n})$ and define
\begin{eqnarray}\label{V}
    V(t) = \int_{\mathbb{R}^{n}} \varphi(x) \left(  \sum_{k=1}^{l} \dfrac{\alpha_{k}^{2}}{\gamma_k} |u_k(x,t)|^{2} \right) \, dx.
\end{eqnarray}
Then, for $m_k = \dfrac{\alpha_k}{2\gamma_k}$, we have
\begin{equation}\label{V'}
V'(t) \;=\; \mathcal{R}(t) - 4 \int_{\mathbb{R}^{n}} \varphi(x) \mbox{Im} \sum_{k=1}^{l} m_k |x|^{-b} f_k(\mathbf{u})\overline{u}_k \, dx
\end{equation}
where
\begin{equation}\label{R}
    \mathcal{R}(t) = 2 \sum_{k=1}^{l} \alpha_k \mbox{Im} \int_{\mathbb{R}^n} \nabla \varphi \nabla u_k \overline{u}_k\, dx = 2 \sum_{k=1}^{l} \alpha_k \mbox{Im} \int_{\mathbb{R}^n} \varphi' \dfrac{x \cdot \nabla u_k}{r}\overline{u_k}\, dx,
\end{equation}
and
\begin{equation}\label{V''}
V''(t) \;=\; \mathcal{R}^{'}(t) - 4 \dfrac{d}{dt} \left[ \int_{\mathbb{R}^{n}} \varphi(x) \mbox{Im} \sum_{k=1}^{l} m_k |x|^{-b} f_k(\mathbf{u})\overline{u}_k \, dx \right],
\end{equation}
where
\begin{eqnarray}\label{R'}
    \mathcal{R}'(t) &=& 4 \sum_{j=1}^{n}\sum_{m=1}^{n} \int_{\mathbb{R}^{n}} \mbox{Re} \left( \sum_{k=1}^{l} \gamma_k \partial_{x_m}\mathbf{u} \partial_{x_j}\overline{\mathbf{u}} \right) \dfrac{\partial^{2} \varphi}{\partial_{x_j}\partial_{x_m}} \, dx - \sum_{k=1}^{l} \gamma_k  \int_{\mathbb{R}^{n}} |u_k|^{2} \Delta^{2} \varphi \, dx \nonumber \\[2mm]
    &-& 2 \int_{\mathbb{R}^{n}} |x|^{-b} \mbox{Re} F(\mathbf{u}) \Delta \varphi \, dx + 4 \mbox{Re} \int_{\mathbb{R}^{n}} \nabla(|x|^{-b}) \nabla \varphi F(\mathbf{u})\,dx.  
\end{eqnarray}
\end{theorem}

\begin{proof}
A direct calculation shows that the solution $\mathbf{u}$ of the system \eqref{INLS} satisfies
\begin{eqnarray*}
V'(t) &=&
\sum_{k=1}^{l} \dfrac{\alpha_{k}^{2}}{\gamma_k} \int_{\mathbb{R}^{n}} \varphi(x) 2 \mbox{Re} \,\partial_t u_k \overline{u}_k \, dx \;=\;
2 \mbox{Re}\sum_{k=1}^{l} \dfrac{\alpha_{k}}{\gamma_k} \int_{\mathbb{R}^{n}} \varphi(x) (\alpha_k \partial_t u_k) \overline{u}_k \, dx \\[2mm]
&=&
2 \mbox{Re} \,i\,\sum_{k=1}^{l} \dfrac{\alpha_{k}}{\gamma_k} \int_{\mathbb{R}^{n}} \varphi(x) (-i\alpha_k \partial_t u_k) \overline{u}_k \, dx \\[2mm]
&=&
2 \mbox{Re} \,i\,\sum_{k=1}^{l} \dfrac{\alpha_{k}}{\gamma_k} \int_{\mathbb{R}^{n}} \varphi(x) [\gamma_k \Delta u_k - \beta_k u_k + |x|^{-b}f_k(\mathbf{u})] \overline{u}_k \, dx \\[2mm]
&=&
- 2 \mbox{Im} \left\{\sum_{k=1}^{l} \dfrac{\alpha_{k}}{\gamma_k} \int_{\mathbb{R}^{n}} \varphi(x) [\gamma_k \Delta u_k - \beta_k u_k + |x|^{-b}f_k(\mathbf{u})] \overline{u}_k \, dx \right\} \\[2mm]
&=&
- 2 \mbox{Im} \left\{\sum_{k=1}^{l}  \int_{\mathbb{R}^{n}} \alpha_k \Delta u_k \overline{u}_k \varphi(x)  - \dfrac{\alpha_{k} \beta_k }{\gamma_k} |u_k|^{2}\varphi(x) + |x|^{-b} \varphi(x)f_k(\mathbf{u}) \overline{u}_k \, dx \right\} \\[2mm]
&=&
2 \mbox{Im} \sum_{k=1}^{l} \alpha_k  \int_{\mathbb{R}^{n}} \nabla u_k \nabla( \overline{u}_k \varphi(x)) \, dx  - 2  \int_{\mathbb{R}^{n}} \varphi(x)|x|^{-b}  \mbox{Im} \sum_{k=1}^{l}  m_k f_k(\mathbf{u}) \overline{u}_k \, dx  \\[2mm]
&=&
2 \mbox{Im} \sum_{k=1}^{l} \alpha_k  \int_{\mathbb{R}^{n}} \nabla u_k ( \nabla \overline{u}_k \varphi(x) + \overline{u}_k \nabla \varphi(x)) \, dx  \\[2mm]
& & - 2  \int_{\mathbb{R}^{n}} \varphi(x)|x|^{-b}  \mbox{Im} \sum_{k=1}^{l}  m_k f_k(\mathbf{u}) \overline{u}_k \, dx \\[2mm]
&=&
2 \mbox{Im} \sum_{k=1}^{l} \alpha_k  \int_{\mathbb{R}^{n}} \nabla \varphi(x) \nabla u_k  \overline{u}_k  \, dx  - 2  \int_{\mathbb{R}^{n}} \varphi(x)|x|^{-b}  \mbox{Im} \sum_{k=1}^{l}  m_k f_k(\mathbf{u}) \overline{u}_k \, dx.  
\end{eqnarray*}

Now, calculating the derivative of $\mathcal{R}$,
\begin{eqnarray*}
  \mathcal{R}'(t)
  &=& 
  2 \sum_{k=1}^{l}  \alpha_k
  Im \int_{\mathbb{R}^n} \nabla \varphi \cdot (\nabla \partial_t u_k \, \overline{u}_k  + \nabla u_k \, \partial_t \overline{u}_k) \, dx  \\[2mm]
  &=& 
  2 \sum_{k=1}^{l}  \alpha_k Im \left\{   \sum_{j=1}^{n} \int_{\mathbb{R}^n} \left[   \partial_t\overline{u}_k \partial_{x_j}u_k \partial_{x_j}\varphi - \partial_t u_k \partial_{x_j}(\overline{u}_k \partial_{x_j}\varphi)   \right]\,dx   \right\} \\[2mm]
&=& 
  2 \sum_{k=1}^{l}  \alpha_k Im \left\{   \sum_{j=1}^{n} \int_{\mathbb{R}^n} \left[   \partial_t\overline{u}_k \partial_{x_j}u_k \partial_{x_j}\varphi -  \partial_t u_k \partial_{x_j}\overline{u}_k \partial_{x_j}\varphi - \partial_t u_k \overline{u}_k \partial_{x_j}^{2}\varphi \right]\,dx   \right\} \\[2mm]
&=& 
  4 \sum_{k=1}^{l}  Im \left\{ \alpha_k    \int_{\mathbb{R}^n} \sum_{j=1}^{n}   \partial_t\overline{u}_k \partial_{x_j}u_k \partial_{x_j}\varphi\,dx   \right\}
  - 
  2 \sum_{k=1}^{l}   Im \left\{  \alpha_k \int_{\mathbb{R}^n}  \sum_{j=1}^{n}  \partial_t u_k \overline{u}_k \partial_{x_j}^{2}\varphi \,dx  \right\} \\[2mm]
&=&
    4\sum_{k=1}^{l} I_k + 2 \sum_{k=1}^{l}J_k.
\end{eqnarray*}
Let us start with $J_k$, using \eqref{INLS}, we obtain 
\begin{eqnarray*}
    J_k &=& -    Im \left\{  \alpha_k \int_{\mathbb{R}^n}  \sum_{j=1}^{n}  \partial_t u_k \overline{u}_k \partial_{x_j}^{2}\varphi \,dx   \right\} \;=\; \sum_{j=1}^{n}  \mbox{Im}   \, i \int_{\mathbb{R}^n} i \alpha_k   \partial_t u_k \overline{u}_k \partial_{x_j}^{2}\varphi \,dx  \\[2mm] 
    &=& \sum_{j=1}^{n}   \mbox{Re}   \, \int_{\mathbb{R}^n} - [\gamma_k \Delta u_k - \beta_k u_k +|x|^{-b}f_k(\mathbf{u})] \overline{u}_k \partial_{x_j}^{2} \varphi \,dx  \\[2mm]
    &=& - \mbox{Re} \sum_{j=1}^{n}\sum_{m=1}^{n}   \, \gamma_k \int_{\mathbb{R}^n}  \partial_{x_m}^{2} u_k \overline{u}_k \partial_{x_j}^{2}\varphi \,dx \\[2mm]
    &+&  \mbox{Re} \, \sum_{j=1}^{n}   
 \beta_k \int_{\mathbb{R}^n}   |u_k|^{2} \partial_{x_j}^{2}\varphi \,dx \;-\;   \mbox{Re} \, \sum_{j=1}^{n} \int_{\mathbb{R}^n} |x|^{-b}f_k(\mathbf{u}) \overline{u}_k \partial_{x_j}^{2}\varphi \,dx\\[2mm]
    &=&  \mbox{Re} \sum_{j=1}^{n}\sum_{m=1}^{n}   \, \gamma_k \int_{\mathbb{R}^n}  |\partial_{x_m} u_k|^2 \partial_{x_j}^{2}\varphi \,dx \;+\; 
     \mbox{Re} \sum_{j=1}^{n}\sum_{m=1}^{n}   \, \gamma_k \int_{\mathbb{R}^n}  \partial_{x_m} u_k \overline{u}_k \dfrac{ \partial^{3} \varphi}{\partial_{x_m}\partial_{x_j}^{2}} \,dx 
    \\[2mm]
    &+&     
 \beta_k \int_{\mathbb{R}^n}   |u_k|^{2} \Delta \varphi \,dx \;-\;   \mbox{Re} \,  \int_{\mathbb{R}^n} |x|^{-b}f_k(\mathbf{u}) \overline{u}_k \Delta \varphi \,dx \\[2mm]
    &=&    \gamma_k \int_{\mathbb{R}^n}  |\nabla u_k|^2 \Delta \varphi \,dx \;+\;  \beta_k \int_{\mathbb{R}^n}   |u_k|^{2} \Delta \varphi \,dx \;-\;   \mbox{Re} \,  \int_{\mathbb{R}^n} |x|^{-b}f_k(\mathbf{u}) \overline{u}_k \Delta \varphi \,dx \\[2mm]
    &+&
     \mbox{Re} \sum_{j=1}^{n}\sum_{m=1}^{n}   \, \gamma_k \int_{\mathbb{R}^n}  \partial_{x_m} u_k \overline{u}_k \dfrac{ \partial^{3} \varphi}{\partial_{x_m}\partial_{x_j}^{2}} \,dx .  
\end{eqnarray*}
Since $\partial_{x_m} |u_k|^{2} = 2 \mbox{Re} \, \partial_{x_m}u_k \overline{u}_k$, then
$$
\mbox{Re} \int_{\mathbb{R}^n}  \partial_{x_m} u_k \overline{u}_k \dfrac{ \partial^{3} \varphi}{\partial_{x_m}\partial_{x_j}^{2}} \,dx \;=\; 
\dfrac{1}{2} \int_{\mathbb{R}^n}  \partial_{x_m}(|u_k|^{2}) \dfrac{ \partial^{3} \varphi}{\partial_{x_m}\partial_{x_j}^{2}} \,dx \;=\; -
\dfrac{1}{2} \int_{\mathbb{R}^n}  |u_k|^{2} \dfrac{ \partial^{4} \varphi}{\partial_{x_m}^{2}\partial_{x_j}^{2}} \,dx.
$$

Therefore,
\begin{eqnarray*}
    J_k  &=&    \gamma_k \int_{\mathbb{R}^n}  |\nabla u_k|^2 \Delta \varphi \,dx \;+\;  \beta_k \int_{\mathbb{R}^n}   |u_k|^{2} \Delta \varphi \,dx \;-\;   \mbox{Re} \,  \int_{\mathbb{R}^n} |x|^{-b}f_k(\mathbf{u}) \overline{u}_k \Delta \varphi \,dx \\[2mm]
    &-&
\dfrac{\gamma_k}{2} \int_{\mathbb{R}^n}  |u_k|^{2} \Delta^{2} \varphi \,dx.
\end{eqnarray*}
Summing over $k$, we obtain
\begin{eqnarray}\label{2J}
   2 \sum_{k=1}^{l} J_k  &=&  2 \sum_{k=1}^{l}   \gamma_k \int_{\mathbb{R}^n}  |\nabla u_k|^2 \Delta \varphi \,dx \;+\; 2 \sum_{k=1}^{l}  \beta_k \int_{\mathbb{R}^n}   |u_k|^{2} \Delta \varphi \,dx  \\[2mm]
   &-& 2 \sum_{k=1}^{l}    \mbox{Re} \,  \int_{\mathbb{R}^n} |x|^{-b}f_k(\mathbf{u}) \overline{u}_k \Delta \varphi \,dx \;-\;  \sum_{k=1}^{l} 
\gamma_k \int_{\mathbb{R}^n}  |u_k|^{2} \Delta^{2} \varphi \,dx. \nonumber
\end{eqnarray}

Now, we study carefully $I_k$, since $\mbox{Im}z = - \mbox{Im}\overline{z}$, equation \eqref{INLS} implies:
\begin{eqnarray*}
    I_k &=& Im \left\{ \alpha_k    \int_{\mathbb{R}^n} \sum_{j=1}^{n}   \partial_t\overline{u}_k \partial_{x_j}u_k \partial_{x_j}\varphi\,dx   \right\} \;=\; - Im \left\{ \alpha_k    \int_{\mathbb{R}^n} \sum_{j=1}^{n}   \partial_t u_k \partial_{x_j}\overline{u}_k \partial_{x_j}\varphi\,dx   \right\} \\[2mm] &=&  Im\, i \,    \int_{\mathbb{R}^n} \sum_{j=1}^{n}  i \alpha_k  \partial_t u_k \partial_{x_j}\overline{u}_k \partial_{x_j}\varphi\,dx \;+\; \mbox{Re}\, \sum_{j=1}^{n}   \int_{\mathbb{R}^n} i \alpha_k  \partial_t u_k \partial_{x_j}\overline{u}_k \partial_{x_j}\varphi\,dx \\[2mm]
    &=& - \mbox{Re}\, \sum_{j=1}^{n}  \int_{\mathbb{R}^n}   [ \gamma_k \Delta u_k - \beta_k u_k +|x|^{-b}f_k(\mathbf{u}) ] \partial_{x_j}\overline{u}_k \partial_{x_j}\varphi\,dx \\[2mm]
    &=&
    - \mbox{Re}\,  \gamma_k \sum_{j=1}^{n}\sum_{m=1}^{n}  \int_{\mathbb{R}^n}    \partial_{x_m}^{2}u_k  \partial_{x_j}\overline{u}_k \partial_{x_j}\varphi\,dx \;+\;
    \beta_k \sum_{j=1}^{n}  \int_{\mathbb{R}^n}  \mbox{Re}\,( \partial_{x_j}\overline{u}_k  u_k) \partial_{x_j}\varphi\,dx \\[2mm]
    &-&  \sum_{j=1}^{n}  \int_{\mathbb{R}^n} |x|^{-b} \mbox{Re}(f_k(\mathbf{u}) \partial_{x_j}\overline{u}_k) \partial_{x_j}\varphi\,dx \\[2mm]
    \\[2mm]
    &=&
    - \mbox{Re}\,  \gamma_k \sum_{j=1}^{n}\sum_{m=1}^{n}  \int_{\mathbb{R}^n}    \partial_{x_m}^{2}u_k  \partial_{x_j}\overline{u}_k \partial_{x_j}\varphi\,dx \;+\;
    \dfrac{\beta_k}{2} \sum_{j=1}^{n}  \int_{\mathbb{R}^n}   \partial_{x_j}(|u_k|^{2}) \partial_{x_j}\varphi\,dx \\[2mm]
    &-&  \sum_{j=1}^{n}  \int_{\mathbb{R}^n} |x|^{-b} \mbox{Re}(f_k(\mathbf{u}) \partial_{x_j}\overline{u}_k) \partial_{x_j}\varphi\,dx
\\[2mm]
    &=&
    - \mbox{Re}\,  \gamma_k \sum_{j=1}^{n}\sum_{m=1}^{n}  \int_{\mathbb{R}^n}    \partial_{x_m}^{2}u_k  \partial_{x_j}\overline{u}_k \partial_{x_j}\varphi\,dx \;-\;
    \dfrac{\beta_k}{2}   \int_{\mathbb{R}^n}   |u_k|^{2} \Delta \varphi\,dx \\[2mm]
    &-&  \sum_{j=1}^{n}  \int_{\mathbb{R}^n} |x|^{-b} \mbox{Re}(f_k(\mathbf{u}) \partial_{x_j}\overline{u}_k) \partial_{x_j}\varphi\,dx
\end{eqnarray*}
Since, 
$$
\partial_{x_j}F(\mathbf{u}) = \sum_{k=1}^{l} \left( \dfrac{\partial F}{\partial u_k} \dfrac{\partial u_k}{\partial_{x_j}} + \dfrac{\partial F}{\partial \overline{u}_k} \dfrac{\partial \overline{u}_k}{\partial_{x_j}}  \right).
$$
Taking the real part, and using (H3), we have
\begin{eqnarray*}
\mbox{Re} \,\partial_{x_j}F(\mathbf{u}) &=& \mbox{Re} \sum_{k=1}^{l} \left( \overline{\dfrac{\partial F}{\partial u_k}} \dfrac{\partial \overline{u}_k}{\partial_{x_j}} + \dfrac{\partial F}{\partial \overline{u}_k} \dfrac{\partial \overline{u}_k}{\partial_{x_j}}  \right)
=\mbox{Re} \sum_{k=1}^{l} \left( \overline{\dfrac{\partial F}{\partial u_k}} + \dfrac{\partial F}{\partial \overline{u}_k}  \right) \partial_{x_j}\overline{u}_k \\[2mm]
&=&
\mbox{Re} \sum_{k=1}^{l} f_k(\overline{u})\partial_{x_j} \overline{u}_k. 
\end{eqnarray*}
Therefore, taking de sum over k, we have
\begin{eqnarray*}
    \sum_{k=1}^{l} I_k &=&     \sum_{k=1}^{l} - \gamma_k   \mbox{Re}\,  \sum_{j=1}^{n}\sum_{m=1}^{n}  \int_{\mathbb{R}^n}    \partial_{x_m}^{2}u_k  \partial_{x_j}\overline{u}_k \partial_{x_j}\varphi\,dx \\[2mm]
    &-&  \sum_{k=1}^{l}
    \dfrac{\beta_k}{2}   \int_{\mathbb{R}^n}   |u_k|^{2} \Delta \varphi\,dx \;-\; \sum_{j=1}^{n}  \int_{\mathbb{R}^n} |x|^{-b} \partial_{x_j} {\mbox{Re}F(\mathbf{u})}  \partial_{x_j}\varphi\,dx \\[2mm]
    &=:&  \sum_{k=1}^{l} A_k - \sum_{k=1}^{l}
    \dfrac{\beta_k}{2}   \int_{\mathbb{R}^n}   |u_k|^{2} \Delta \varphi\,dx +B.
\end{eqnarray*}

Integration by Parts $A_k$:
$$
A_k \;=\;
  \gamma_k   \mbox{Re}\,  \sum_{j=1}^{n}\sum_{m=1}^{n}  \int_{\mathbb{R}^n}    \partial_{x_m}u_k  \dfrac{\partial^{2} \overline{u}_k}{\partial_{x_m}\partial_{x_j}} \partial_{x_j}\varphi\,dx
+
\gamma_k   \mbox{Re}\,  \sum_{j=1}^{n}\sum_{m=1}^{n}  \int_{\mathbb{R}^n}    \partial_{x_m}u_k  \partial_{x_j}\overline{u}_k \dfrac{\partial^{2} \varphi}{\partial_{x_m}\partial_{x_j}}\,dx,
$$
note that 
$$
 \int_{\mathbb{R}^n}    \dfrac{\partial^{2} \overline{u}_k}{\partial_{x_m}\partial_{x_j}} \partial_{x_m}u_k\partial_{x_j}\varphi\,dx =
 - \int_{\mathbb{R}^n}    \dfrac{\partial^{2} u_k}{\partial_{x_m}\partial_{x_j}} \partial_{x_m}\overline{u}_k\partial_{x_j}\varphi\,dx - \int_{\mathbb{R}^n} |\partial_{x_m} u_k|^{2} \partial_{x_j}^{2} \varphi \, dx. 
$$
Then
$$
\mbox{Re}\,  \sum_{j=1}^{n}\sum_{m=1}^{n}  \int_{\mathbb{R}^n}    \partial_{x_m}u_k  \dfrac{\partial^{2} \overline{u}_k}{\partial_{x_m}\partial_{x_j}} \partial_{x_j}\varphi\,dx = - \dfrac{1}{2} \int_{\mathbb{R}^n} |\nabla u_k|^{2} \Delta \varphi \, dx.
$$
Therefore,
$$
A_k \;=\;
  - \dfrac{\gamma_k}{2} \int_{\mathbb{R}^n} |\nabla u_k|^{2} \Delta \varphi \, dx+
\gamma_k   \mbox{Re}\,  \sum_{j=1}^{n}\sum_{m=1}^{n}  \int_{\mathbb{R}^n}    \partial_{x_m}u_k  \partial_{x_j}\overline{u}_k \dfrac{\partial^{2} \varphi}{\partial_{x_m}\partial_{x_j}}\,dx.
$$
Now, Integrating by parts $B$:
\begin{eqnarray*}
    B &=& \sum_{j=1}^{n}  \int_{\mathbb{R}^n} \dfrac{\partial (|x|^{-b})}{\partial_{x_j}} \; \mbox{Re}F(\mathbf{u})  \partial_{x_j}\varphi\,dx + \sum_{j=1}^{n}  \int_{\mathbb{R}^n} |x|^{-b}  \mbox{Re}F(\mathbf{u})  \partial_{x_j}^{2}\varphi\,dx \\[2mm]
    &=&   \int_{\mathbb{R}^n} \nabla (|x|^{-b}) \; \nabla\varphi  \mbox{Re}F(\mathbf{u})  \,dx +   \int_{\mathbb{R}^n} |x|^{-b}  \mbox{Re}F(\mathbf{u})  \Delta\varphi\,dx.
\end{eqnarray*}
Consequently,
\begin{eqnarray}\label{4I}
    4 \sum_{k=1}^{l} I_k &=& 4\sum_{k=1}^{l} A_k-4\sum_{k=1}^{l}
    \dfrac{\beta_k}{2}   \int_{\mathbb{R}^n}   |u_k|^{2} \Delta \varphi\,dx +4B \nonumber \\[2mm]
    &=&
    -2\sum_{k=1}^{l} \gamma_k \int_{\mathbb{R}^n} |\nabla u_k|^{2} \Delta \varphi \, dx+ 4 \sum_{k=1}^{l}
  \mbox{Re}\,  \sum_{j=1}^{n}\sum_{m=1}^{n} \gamma_k  \int_{\mathbb{R}^n}    \partial_{x_m}u_k  \partial_{x_j}\overline{u}_k \dfrac{\partial^{2} \varphi}{\partial_{x_m}\partial_{x_j}}\,dx \nonumber \\[2mm]
&-&2\sum_{k=1}^{l}
    \beta_k   \int_{\mathbb{R}^n}   |u_k|^{2} \Delta \varphi\,dx + 4\mbox{Re} \int_{\mathbb{R}^n} \nabla (|x|^{-b}) \; \nabla\varphi  F(\mathbf{u})  \,dx\\[2mm]
    &+& 4 \mbox{Re}  \int_{\mathbb{R}^n} |x|^{-b} F(\mathbf{u})  \Delta\varphi\,dx. \nonumber
\end{eqnarray}
The conclusion follows from \eqref{2J} and \eqref{4I}.
\end{proof}

We finish this section with a technical lemma that will be used in the next section:

\begin{lemma}\label{lemma: varphi radial}
    Assume that $n\geq 1$. Let $r=|x|$, for $x \in \mathbb{R}^{n}$. Let $R>0$. There exists a radial smooth function 
    $$
    \varphi(x) = \left\{ \begin{array}{lc}
|x|^{2} \;\;\;\;\;\;\;\; \text{if} \;\; 0 \leq |x| \leq R\\
\;\;0\;\;\;\;\;\;\;\;\;\; \text{if} \;\;\;\;\;\; |x| \geq 2R
\end{array},\right.
    $$
    with $0 \leq \varphi(x) \leq |x|^{2}$, \;\; $\partial_{r}^{2} \varphi \leq 2$ \;\; and, \;\; $\partial_{r}^{4} \varphi \leq \dfrac{4}{R^2}$, \;\; for all $x \in \mathbb{R}^{n}$.
\end{lemma}
\begin{proof}
    The lemma follows by a straightforward calculation.
\end{proof}

\subsection{ Blow-up results}


Under specific assumptions on the coefficients of \eqref{INLS}, we now establish the sharpness of condition $(iii)$ in Theorem \ref{thrm: Global}. This sharpness is demonstrated by constructing initial data that violate this condition and lead to finite-time blow-up.

\subsubsection{Intercritical case}

This subsection aims to show the existence of blow-up solutions for \eqref{INLS}. To give the precise statement of our result, we set $\mathcal{G}(1, \mathbf{0})$. Thus, the main result of this subsection is Theorem \ref{blow-up}.


Let us start by introducing the ``\emph{Pohozaev}'' functional
\begin{eqnarray}\label{T}
    \mathcal{T}_n(\mathbf{u}(t)) = K(\mathbf{u}(t)) - \dfrac{n+2b}{2}P(\mathbf{u}(t)). 
\end{eqnarray}
We recall the definition of the energy functional:$$
E(\mathbf{u}(t)) \;=\; K(\mathbf{u}(t)) + L(\mathbf{u}(t)) - 2P(\mathbf{u}(t)),
$$
where
$$
L(\mathbf{u}(t)) = \sum_{k=1}^{l} \beta_k \int_{\mathbb{R}^{n}} |u_k|^2  \, dx.
$$
Thus, 
$$
P(\mathbf{u}(t)) \;=\; \dfrac{1}{2} \left[ K(\mathbf{u}(t)) + L(\mathbf{u}(t)) - E(\mathbf{u}(t)) \right]
$$
so
\begin{eqnarray}\label{T_n}
    \mathcal{T}_n (\mathbf{u}(t)) \;=\;
    \dfrac{n+2b}{4} E(\mathbf{u}(t)) - \left( \dfrac{n+2b-4}{4} \right)K(\mathbf{u}(t)) - \dfrac{n+2b}{4}L(\mathbf{u}(t)).
\end{eqnarray}

We note that the Pohozaev functional is strictly negative.

\begin{lemma}\label{T_n < - delta} Under the assumptions of Theorem \ref{blow-up}, there exists $\delta >0$ such that 
    $$
\mathcal{T}_n (\mathbf{u}(t)) \; \leq -\delta < 0, \quad t \in I.
$$
\end{lemma}

\begin{proof}
    Remember, by Pohozaev's identities, Lemma \ref{Pohozaev}: $K(\boldsymbol{\psi}) = (n+2b) I(\boldsymbol{\psi})$, $\mathcal{Q}(\boldsymbol{\psi}) \;=\; (6-n-2b)I(\boldsymbol{\psi})$ and $P(\boldsymbol{\psi}) = 2 I(\boldsymbol{\psi})$. Thus
\begin{eqnarray}\label{K=E}
K(\boldsymbol{\psi}) \;=\; \dfrac{n+2b}{n+2b -4} \mathcal{E}(\boldsymbol{\psi}). 
\end{eqnarray}
Since $\boldsymbol{\psi} \in \mathcal{G}(1,\mathbf{0})$ the functionals $Q$ and $\mathcal{Q}$ are the same. Therefore, from (H6) and \eqref{P<QK}
\begin{eqnarray}\label{K<E+CQK}
    K(\mathbf{u}(t)) &=& E(\mathbf{u}_0) - L(\mathbf{u}(t)) + 2P(\mathbf{u}(t)) \;\leq\; E(\mathbf{u}_0) + 2|P(\mathbf{u}(t))| \nonumber \\[2mm]
    &\leq& E(\mathbf{u}_0) + \dfrac{2}{\xi_1} Q(\mathbf{u}(t))^{\frac{6-n-2b}{4}} K(\mathbf{u}(t))^{\frac{n+2b}{4}}.
\end{eqnarray}
Now in the notation of Lemma \ref{lemma-inequality}, if we take
$$G(t) = K(\mathbf{u}(t)), \;\;\;\; \alpha = E(\mathbf{u}_0), \;\;\;\; \beta = \dfrac{1}{\xi_1} Q(\mathbf{u}_0)^{\frac{6-n-2b}{4}},$$
where $\dfrac{1}{\xi_1} = \dfrac{2(6-n-2b)^{\frac{n+2b-4}{4}}}{(n+2b)^{\frac{n+2b}{4}}} \dfrac{1}{Q(\boldsymbol{\psi})^{\frac{1}{2}}}$ \;\; and \;\; $q=\dfrac{n+2b}{4}$, then  
\begin{eqnarray}\label{gamma}
\gamma = (\beta q)^{-\frac{1}{q-1}} = \dfrac{n+2b}{6-n-2b} \dfrac{Q(\boldsymbol{\psi})^{\frac{2}{n+2b-4}}}{{Q}(\mathbf{u}_{0})^{\frac{6-n-2b}{n+2b-4}}}.   
\end{eqnarray}
Also from \eqref{K<E+CQK} $f \circ G \geq 0$. Moreover, by using \eqref{gamma} a direct calculation gives
$a < \left(1 - \dfrac{1}{q} \right) \gamma$ is equivalent to 
\begin{eqnarray}\label{EQ<QQ}
E(\mathbf{u}(t))^{n+2b-4}{Q}(\mathbf{u}_{0})^{6-n-2b} < \left( \dfrac{n+2b-4}{6-n-2b} Q(\boldsymbol{\psi})\right)^{n+2b-4} Q(\boldsymbol{\psi})^{6-n-2b},    
\end{eqnarray}
since 
\begin{eqnarray}\label{E=K=Q}
\mathcal{E}(\boldsymbol{\psi}) = \dfrac{n+2b-4}{n+2b}K(\boldsymbol{\psi}) = \dfrac{n+2b-4}{6-n-2b}Q(\boldsymbol{\psi}).
\end{eqnarray}
Replacing \eqref{E=K=Q} in \eqref{EQ<QQ}, we have
\begin{eqnarray*}
    E(\mathbf{u}(t))^{s_c}{Q}(\mathbf{u}_{0})^{1-s_c} < \mathcal{E}(\boldsymbol{\psi})^{s_c} Q(\boldsymbol{\psi})^{1-s_c} .
\end{eqnarray*}
Similarly, we see that $G(0) > \gamma $ is equivalent to
\begin{eqnarray*}
    K(\mathbf{u}_0)^{s_c}{Q}(\mathbf{u}_{0})^{1-s_c} > K(\boldsymbol{\psi})^{s_c}Q(\boldsymbol{\psi})^{1-s_c}.
\end{eqnarray*}
Hence, an application of Lemma \ref{lemma-inequality} yields
\begin{eqnarray}\label{QK}
K(\mathbf{u}(t))^{s_c}{Q}(\mathbf{u}_{0})^{1-s_c} > K(\boldsymbol{\psi})^{s_c}Q(\boldsymbol{\psi})^{1-s_c}.    
\end{eqnarray}
Thus, from \eqref{E(u_0)Q(u_0)}, \eqref{E=K=Q} and \eqref{QK} we have
\begin{eqnarray*}
    \left[ \left( \dfrac{n+2b}{n+2b-4} E(\mathbf{u}(t)) \right) \right]^{s_c} 
    &=& \left[ \left( \dfrac{n+2b}{n+2b-4} E(\mathbf{u}_0) \right) \right]^{s_c} \\[2mm]
    &<& \left[ \left( \dfrac{n+2b}{n+2b-4} \mathcal{E}(\boldsymbol{\psi})^{1-s_c})  \right) \right]^{s_c} \dfrac{Q(\boldsymbol{\psi})^{1-s_c} }{Q(\mathbf{u}_0)^{1-s_c}} \\[2mm]
    &=& K(\boldsymbol{\psi})^{1-s_c})^{s_c} \dfrac{Q(\boldsymbol{\psi})^{1-s_c} }{Q(\mathbf{u}_0)^{1-s_c}},
\end{eqnarray*}
and then 
$$
(n+2b)E(\mathbf{u}(t)) < (n+2b-4)K(\mathbf{u}(t)).
$$
This combined with \eqref{T_n}
$$
4 \mathcal{T}_n(\mathbf{u}(t)) \leq (n+2b)E(\mathbf{u}(t)) - (n+2b-4)K(\mathbf{u}(t)) \leq 0.
$$
\textbf{Claim:} There exists $\sigma_0 >0$ such that 
\begin{eqnarray}\label{Tn < sigma_0 K}
    \mathcal{T}_n(\mathbf{u}(t)) < -\sigma_0 K(\mathbf{u}(t)), \quad t \in I.
\end{eqnarray}
Indeed, if $E(\mathbf{u}_0) \leq 0$ from \eqref{T_n} we can promptly take $\sigma_0 = \dfrac{n+2b-4}{4}$ (Remark $4 < n+2b <6$).
Now suppose $E(\mathbf{u}_0) >0$ and assume by contradiction that \eqref{Tn < sigma_0 K} does not hold. Then we can find sequences $(t_m) \subset I$ and $(\sigma_m) \subset \mathbb{R}_{+}$ with $\sigma_m \to 0$ such that
$$
- \sigma_m \left( \dfrac{n+2b-4}{4} \right)K(\mathbf{u}(t_m)) \leq \mathcal{T}_{n}(\mathbf{u}(t_m)) < 0.
$$
Thus, the last inequality and \eqref{T_n} give
\begin{eqnarray*}
    E(\mathbf{u}(t_m)) &=& \dfrac{4}{n+2b}\mathcal{T}_n(\mathbf{u}(t_m)) + \dfrac{n+2b-4}{n+2b} K(\mathbf{u}(t_m)) + L(\mathbf{u}(t_m)) \nonumber\\[2mm]
    &\geq& - \sigma_{m} \left( \dfrac{n+2b-4}{n+2b} \right) K(\mathbf{u}(t_m)) + \dfrac{n+2b-4}{n+2b} K(\mathbf{u}(t_m)) + L(\mathbf{u}(t_m)) \nonumber \\[2mm]
    &\geq& (1 - \sigma_{m}) \left( \dfrac{n+2b-4}{n+2b} \right) K(\mathbf{u}(t_m)).
\end{eqnarray*}
From this, the conservation of the energy, \eqref{QK} and \eqref{E=K=Q} we get
\begin{eqnarray*}
    Q(\mathbf{u}_0)^{1-s_c} E(\mathbf{u}_0)^{s_c} &\geq& (1-\sigma_m)^{s_c}\left( \dfrac{n+2b-4}{n+2b} \right)^{s_c} K(\mathbf{u}(t))^{s_c} Q(\mathbf{u}_0)^{1-s_c} \\[2mm]
    &>& (1-\sigma_m)^{s_c}\left( \dfrac{n+2b-4}{n+2b}  K(\boldsymbol{\psi})\right)^{s_c} Q(\boldsymbol{\psi})^{1-s_c} \\[2mm]
    &=&
    (1-\sigma_m)^{s_c} \mathcal{E}(\boldsymbol{\psi})^{s_c} Q(\boldsymbol{\psi})^{1-s_c}.
\end{eqnarray*}
Taking $m\to +\infty$ in the last inequality we obtain
$$
Q(\mathbf{u}_0)^{1-s_c} E(\mathbf{u}_0)^{s_c} \geq  \mathcal{E}(\boldsymbol{\psi})^{s_c} Q(\boldsymbol{\psi})^{1-s_c},
$$
which is a contradiction, so the claim is proved.

Finally note that \eqref{QK} gives $$
K(\mathbf{u}(t))^{s_c} > K(\boldsymbol{\psi})^{s_c} \dfrac{Q(\boldsymbol{\psi})^{1-s_c}}{Q(\mathbf{u}_0)^{1-s_c}} =: \epsilon_{0}^{s_c}.
$$ 
Thus $-\sigma_0 K(\mathbf{u}(t)) < - \sigma_0 \epsilon_0$. Therefore the result is proved.
\end{proof}
We now proceed to the proof of Theorem \ref{blow-up}.

\textbf{Proof of Theorem \ref{blow-up}:}

Suppose that the maximal existence interval is $(-T_{\ast}, T^{\ast})$. We proceed by contradiction, without loss of generality assume that $T^{\ast} = + \infty$. Assume $\varphi$ is radial, since $r=|x|$ then $r^2 = |x|^2$ thus $2 r \dfrac{\partial r}{\partial x_j} = 2 x_j \Rightarrow \dfrac{\partial r}{\partial x_j} = \dfrac{x_j}{r}$, therefore
$$
\dfrac{\partial \varphi}{\partial x_j} = \varphi'(r) \dfrac{\partial r}{\partial x_j} = \varphi'(r) \dfrac{x_j}{r} \;\;\;\; \mbox{thus} \;\;\;\; \nabla \varphi = \varphi' \dfrac{x}{r}.
$$
and
$$
\dfrac{\partial^{2} \varphi}{\partial x_m \partial x_j} = \varphi^{''}(r) \dfrac{x_m x_j}{r^2} + \delta_{jm} \dfrac{\varphi'(r)}{r} - \varphi' \dfrac{x_m x_j}{r^3}.
$$
Then, let $V(t)$ be defined in \eqref{V}
\begin{equation*}
    V(t) = \sum_{k=1}^{l} \int_{\mathbb{R}^{n}} \dfrac{\alpha_{k}^{2}}{\gamma_k} \varphi(x)|u_k(x,t)|^2\,dx,
\end{equation*}
and then, by Theorem~\ref{V,V',V''}
$$
V'(t) \;=\; \mathcal{R}(t) - 4 \int_{\mathbb{R}^{n}} \varphi(x) \mbox{Im} \sum_{k=1}^{l} m_k |x|^{-b} f_k(\mathbf{u})\overline{u}_k \, dx
$$
where $m_k = \dfrac{\alpha_k}{2 \gamma_k}$ and
\begin{equation*}
    \mathcal{R}(t) = 2 \sum_{k=1}^{l} \alpha_k \mbox{Im} \int_{\mathbb{R}^n} \nabla \varphi \nabla u_k \overline{u}_k\, dx = 2 \sum_{k=1}^{l} \alpha_k \mbox{Im} \int_{\mathbb{R}^n} \varphi' \dfrac{x \cdot \nabla u_k}{r}\overline{u_k}\, dx.
\end{equation*}
Moreover,
\begin{eqnarray*}
    \mathcal{R}'(t) &=& 4 \sum_{j=1}^{n} \sum_{m=1}^{n} \int_{\mathbb{R}^n}  \mbox{Re} \left(\sum_{k=1}^{l} \gamma_k \partial_{x_m}u_k \partial_{x_j}\overline{u}_k \right) \left(  \varphi^{''} \dfrac{x_m x_j}{r^2} + \delta_{jm} \dfrac{\varphi'}{r} - \varphi' \dfrac{x_m x_j}{r^3}\right) \, dx \nonumber \\[2mm]
    &-& \sum_{k=1}^{l} \gamma_k \int_{\mathbb{R}^n} |u_k|^{2} \Delta^{2} \varphi\, dx - 2 \int_{\mathbb{R}^n} |x|^{-b} \mbox{Re} F(\mathbf{u}) \Delta \varphi \, dx -4b \mbox{Re} \int_{\mathbb{R}^n} |x|^{-b} F(\mathbf{u}) \dfrac{\varphi'}{r}\, dx,
\end{eqnarray*}
where $\nabla(|x|^{-b}) \nabla \varphi =-b|x|^{-b} \dfrac{\varphi'}{r}$.
$$
\dfrac{\partial }{\partial x_j}(|x|^{-b}) = \dfrac{\partial }{\partial x_j}((|x|^2)^{-\frac{b}{2}}) = - \dfrac{b}{2} (|x|^2)^{-\frac{b}{2}-1} 2 x_j = -b |x|^{-b-2}  x_j 
$$ and
$$
\dfrac{\partial }{\partial x_j}(|x|^{-b}) \dfrac{\partial \varphi }{\partial x_j} = -b|x| ^{-b-2}  x_j \varphi' \dfrac{x_j}{r} \;\;\;\; \mbox{thus} \;\;\;\; \nabla(|x|^{-b}) \nabla \varphi \;=\; -b |x|^{-b}\dfrac{\varphi'}{r}.
$$
Using \eqref{T}, we can re-write $\mathcal{R}'$
\begin{eqnarray*}
    \mathcal{R}'(t)
    &=&
     8 \mathcal{T}_n(u) 
     \;+\;
     4 \int_{\mathbb{R}^n} \left( \dfrac{\varphi'}{r} -2 \right)\left( \sum_{k=1}^{l} \gamma_k |\nabla u_k|^2 \right) \, dx \\[2mm]
     &+&
        4 \int_{\mathbb{R}^n} \left( \dfrac{\varphi^{''}}{r^2} - \dfrac{\varphi'}{r^3} \right)  \left( \sum_{k=1}^{l} \gamma_k |x \cdot \nabla u_k|^2 \right) \, dx 
        \;-\;
        \sum_{k=1}^{l}  \gamma_k \int_{\mathbb{R}^n}|u_k|^2  \Delta^{2}\varphi \, dx\\[2mm]
        &-&  2\,Re\, \int_{\mathbb{R}^n} \left[  \Delta \varphi - 2(n+2b) +  2b \dfrac{\varphi'}{r} \right]\, |x|^{-b} F(\mathbf{u})\,dx\\[2mm]
        &=:& 8 \mathcal{T}_n(u) + \mathcal{R}_{1}(t) + \mathcal{R}_{2}(t)  + \mathcal{R}_{3}(t),
\end{eqnarray*}
where
\begin{eqnarray*}
     \mathcal{R}_{1} &=&  4 \int_{\mathbb{R}^n} \left( \dfrac{\varphi'}{r} -2 \right)\left( \sum_{k=1}^{l} \gamma_k |\nabla u_k|^2 \right) \, dx \;+\;
        4 \int_{\mathbb{R}^n} \left( \dfrac{\varphi^{''}}{r^2} - \dfrac{\varphi'}{r^3} \right)  \left( \sum_{k=1}^{l} \gamma_k |x \cdot \nabla u_k|^2 \right) \, dx;  \\[2mm]
     \mathcal{R}_{2} &=& - \sum_{k=1}^{l}  \gamma_k \int_{\mathbb{R}^n}|u_k|^2\Delta^{2}\varphi \, dx;  \\[2mm]
     \mathcal{R}_{3} &=& - 2\,\mbox{Re}\, \int_{\mathbb{R}^n} \left[  \Delta \varphi - 2(n+2b) +  2b \dfrac{\varphi'}{r} \right]\, |x|^{-b} F(\mathbf{u})\,dx.
\end{eqnarray*}
 Let $\varphi(x)$ as in Lemma~\ref{lemma: varphi radial}, then
$0 \leq \varphi \leq |x|^{2}$, $\partial_{r}^{2} \varphi(x) \leq 2$ and $\partial_{r}^{4} \varphi(x) \leq \dfrac{4}{R^{2}}$, $\forall x \in \mathbb{R}^{n}$. Here, $R$ is seen as a parameter that be chosen later.
\paragraph{}

\underline{\textbf{Step $1$:}} We will show $\mathcal{R}_1 \leq 0$.

Let be the set $\Omega \;=\; \left\{x\in \mathbb{R}^{n}; \;\;\; \dfrac{\varphi^{''}}{r^2} - \dfrac{\varphi'}{r^3} \leq 0, \quad \mbox{where} \; r=|x| \right\}$. 

Since $\varphi'' \leq 2$ then $\varphi' \leq 2 r$, $\forall r \in \mathbb{R}$, we get:
\begin{eqnarray*}
     \mathcal{R}_{1} &=&  4 \int_{\Omega} \left( \dfrac{\varphi'}{r} -2 \right)\left( \sum_{k=1}^{l} \gamma_k |\nabla u_k|^2 \right) \, dx \;+\;
        4 \int_{\Omega} \left( \dfrac{\varphi^{''}}{r^2} - \dfrac{\varphi'}{r^3} \right)  \left( \sum_{k=1}^{l} \gamma_k |x \cdot \nabla u_k|^2 \right) \, dx  \\[2mm]
        &+& 4 \int_{\mathbb{R}^n \setminus \Omega} \left( \dfrac{\varphi'}{r} -2 \right)\left( \sum_{k=1}^{l} \gamma_k |\nabla u_k|^2 \right) \, dx \\[2mm] &+&
        4 \int_{\mathbb{R}^n \setminus \Omega} \left( \dfrac{\varphi^{''}}{r^2} - \dfrac{\varphi'}{r^3} \right)  \left( \sum_{k=1}^{l} \gamma_k |x \cdot \nabla u_k|^2 \right) \, dx  \\[2mm]
        &\leq& 4 \int_{\Omega} \left( \dfrac{\varphi'}{r} -2 \right)\left( \sum_{k=1}^{l} \gamma_k |\nabla u_k|^2 \right) \, dx \;+\; 4 \int_{\mathbb{R}^n \setminus \Omega} \left( \varphi^{''} - 2 \right)  \left( \sum_{k=1}^{l} \gamma_k | \nabla u_k|^2 \right) \, dx \\[2mm]
        &\leq& 0.
\end{eqnarray*}

\underline{\textbf{Step $2$:}} Bound $\mathcal{R}_{2}$.

Using that $|\Delta^2 \varphi| \leq \dfrac{C}{R^2}$, we have
\begin{eqnarray}\label{R2}
    \mathcal{R}_{2} &=&
     -   \sum_{k=1}^{l}  \gamma_k \int_{\mathbb{R}^n}|u_k|^2\Delta^{2}\varphi \, dx 
     \;\leq\;
 \int_{\mathbb{R}^n}  |\Delta^{2}\varphi| \left( \sum_{k=1}^{l}  \gamma_k |u_k|^2\right) \, dx\\[2mm]
    &\leq& \dfrac{C}{R^2}   \sum_{k=1}^{l}  \gamma_k \int_{\mathbb{R}^n}  |u_k|^2 \, dx \nonumber \;\leq\; \dfrac{C}{R^2} Q(\mathbf{u}(t)) \;=\; \dfrac{C}{R^2} Q(\mathbf{u}_0).
\end{eqnarray}

\underline{\textbf{Step $3$:}} Estimate $\mathcal{R}_{3}$.
Note that if $0 \leq |x| \leq R$, then $\Delta \varphi = 2n$ and $\varphi' = 2r$. Thus, 
$$
\Delta \varphi - 2(n+2b) +  2b \dfrac{\varphi'}{r} =0 \;\;\;\;\; \mbox{in} \;\; 0 \leq |x| \leq R.
$$
Also, there exists a positive constant $C$ such that
$$
\mathcal{R}_{3}(t) 
\;\leq \;
C \,\mbox{Re}\, \int_{\mathbb{R}^n} |x|^{-b} |F(\mathbf{u})|\,dx
\;\leq \;
C \sum_{k=1}^{l} \int_{|x|> R} |x|^{-b}  |u_k|^{3}\;dx.
$$
In the last inequality, we use Lemma \ref{Ademir.lem2.10.2.11.2.13}.

Recall that for any radial function $f \in H^{1}(\mathbb{R}^{n})$ (see, for instance,  \cite[equation~(3.7)]{Ogawa} )
$$
\int_{|x| \geq R} |f|^{3}\, dx \leq C R^{-\frac{(n-1)}{2}} \|f\|_{L^{2}(|x|\geq R)}^{\frac{5}{2}}\|\nabla f\|_{L^{2}(|x|\geq R)}^{\frac{1}{2}}.
$$
Hence, from Young's inequality we can write, for any $\epsilon >0$
\begin{eqnarray*}
\sum_{k=1}^{l} \int_{|x|> R} |x|^{-b}  |u_k|^{3}\;dx 
&\leq& R^{-b} \sum_{k=1}^{l} \int_{|x|> R} |u_k|^{3}\;dx \\[2mm]
&\leq& C R^{-\frac{(n +2b-1)}{2}} \sum_{k=1}^{l}   \|u_k\|_{L^{2}(|x|\geq R)}^{\frac{5}{2}}\|\nabla u_k\|_{L^{2}(|x|\geq R)}^{\frac{1}{2}} \\[2mm]
&\leq&
\sum_{k=1}^{l}  C(\epsilon) R^{-\frac{2(n +2b-1)}{3}} \|u_k\|_{L^{2}(|x|\geq R)}^{\frac{10}{3}}  \\[2mm]
&+& \sum_{k=1}^{l} 2(n+2b-4)\epsilon \gamma_k \|\nabla u_k\|_{L^{2}(|x|\geq R)}^{2}  \\[2mm]
&\leq& C_{\epsilon}R^{-\frac{2(n +2b-1)}{3}} Q(\mathbf{u}_0)^{\frac{5}{3}} + 2(n+2b-4)\epsilon K(\mathbf{u}(t)),
\end{eqnarray*}
where $C_\epsilon$ is a positive constant depending on $\epsilon$, $\alpha_k$, $\gamma_k$ and $\sigma_k$. 

Gathering together the estimates for $\mathcal{R}_1$, $\mathcal{R}_2$ and $\mathcal{R}_3$, we obtain
\begin{eqnarray*}
    \mathcal{R}'(t)
    &\leq&
    8 \mathcal{T}_n(\mathbf{u}(t)) + CR^{-2} Q(\mathbf{u}_0) + 2(n+2b-4)\epsilon K(\mathbf{u}(t)) + C_{\epsilon} R^{-\frac{2(n+2b-1)}{3}} Q(\mathbf{u}_0)^{\frac{5}{3}}
\\[2mm]
&\leq& 
8(1-\epsilon) \mathcal{T}_n(\mathbf{u}(t)) + 8\epsilon \mathcal{T}_n(\mathbf{u}(t)) + CR^{-2} Q(\mathbf{u}_0) + 2(n+2b-4)\epsilon K(\mathbf{u}(t)) \\[2mm] & & + C_{\epsilon} R^{-\frac{2(n+2b-1)}{3}} Q(\mathbf{u}_0)^{\frac{5}{3}}.
\end{eqnarray*}
Since $\mathcal{T}_n(\mathbf{u}(t)) \leq -\delta$ and 
$$
\mathcal{T}_n(\mathbf{u}(t))  \;+\: \dfrac{(n+2b-4)}{4} K(\mathbf{u}(t)) \;\leq\; \dfrac{n+2b}{4}|E(\mathbf{u}_0)|,
$$
we have 
\begin{eqnarray}\label{R'<=}
    \mathcal{R}'(t)
\;\leq\;
-8(1-\epsilon) \delta + 2(n+2b) \epsilon |E(\mathbf{u}_0)| + CR^{-2} Q(\mathbf{u}_0) + C_{\epsilon} R^{-\frac{2(n+2b-1)}{3}} Q(\mathbf{u}_0)^{\frac{5}{3}}.    
\end{eqnarray}
Choosing $\epsilon$ and $R$ such that,
$0< \epsilon < \min \left\{ \dfrac{1}{8},\; \dfrac{\delta}{2(n+2b) |E(\mathbf{u}_0)|}  \right\}$ and $\dfrac{1}{R^2} <  \min \left\{ \dfrac{\delta}{C Q(\mathbf{u}_0)},\; \left[ \dfrac{\delta}{C_{\epsilon} Q(\mathbf{u}_0)^{\frac{5}{3}}} \right]^{\frac{3}{n+2b-1}}\right\}$, we have
\begin{equation}\label{R'<-2delta}
\mathcal{R}'(t) \leq - 2 \delta.    
\end{equation}
Integrating \eqref{R'<-2delta} from $[0,t)$, we have
$$
\mathcal{R}(t) - \mathcal{R}(0) \;=\; \int_{0}^{t} \mathcal{R}'(s)\,ds \;\leq\; \int_{0}^{t} -2\delta \, ds \;=\; -2\delta \,t,
$$
and therefore,
$$
\mathcal{R}(t) \leq - 2 \delta t + \mathcal{R}(0).
$$
On the other hand, using \eqref{R}
\begin{eqnarray*}
    |\mathcal{R}(t)|
    =
    \left| 
    2\,\mbox{Im}\, \sum_{k=1}^{l} \alpha_k \int_{\mathbb{R}^{n}} \nabla \varphi \cdot \nabla u_k \, \overline{u}_k\, dx
    \right| \leq
    2 \sum_{k=1}^{l} \alpha_k \int_{\mathbb{R}^{n}} |\nabla \varphi |\, |\nabla u_k| \, |u_k|\, dx.
\end{eqnarray*}
Observe that $|\nabla \varphi| \leq CR$, from Hölder's inequality we have
\begin{eqnarray*}
    |\mathcal{R}(t)| 
    &\leq& C R \sum_{k=1}^{l} \int_{\mathbb{R}^{n}} |\nabla u_k| \, |u_k|\, dx
    \leq C R \sum_{k=1}^{l} \|\nabla u_k\|_{L^2}\| u_k\|_{L^2} \\[2mm]
    &\leq& 
    C  R Q(\mathbf{u}_0)^{\frac{1}{2}} K(\mathbf{u}(t))^{\frac{1}{2}}.
\end{eqnarray*}
Now, we may choose $T_0$ sufficiently large such that  $\dfrac{\mathcal{R}(0)}{\delta} \;<\; T_0$, and thus,
$$
\mathcal{R}(t) \leq - 2 \delta \, t + \mathcal{R}(0) \;=\; -  \delta \, t + ( \mathcal{R}(0) -  \delta \, t) < -  \delta \, t < 0,\;\;\; t \geq T_0, 
$$
and therefore,
$$
0 < \delta \,t \leq - \mathcal{R}(t) = |\mathcal{R}(t)|  \leq CR Q(\mathbf{u}_0)^{\frac{1}{2}} K(\mathbf{u}(t))^{\frac{1}{2}}, \;\;\; t \geq T_0.
$$
Equivalently,
$$
\delta^2 t^2 \leq C^2 R^2 Q(\mathbf{u}_0) K(\mathbf{u}(t)),  \;\;\; t \geq T_0.
$$
Thus
\begin{eqnarray}\label{K>C0t^2}
K(\mathbf{u}(t)) \geq C_0 t^2,  \;\;\; t \geq T_0,    
\end{eqnarray}
for some positive constant $C_0$.\\
Moreover, taking into account that $\epsilon$ is sufficiently small (less than $1/2$ is enough) by \eqref{R'<=} and \eqref{T_n}
\begin{eqnarray*}
    \mathcal{R}'(t)
    &\leq&
    8 \mathcal{T}_n(\mathbf{u}(t)) + CR^{-2} Q(\mathbf{u}_0) + 2(n+2b-4)\epsilon K(\mathbf{u}(t)) + C_{\epsilon} R^{-\frac{2(n+2b-1)}{3}} Q(\mathbf{u}_0)^{\frac{5}{3}}
\\[2mm]
&=& 
8 \left(
\dfrac{n+2b}{4} E(\mathbf{u}_0) - \dfrac{(n+2b-4)}{4} K(\mathbf{u}(t)) - \dfrac{n+2b}{4} L(\mathbf{u}(t)) \right) \\[2mm]
& & + CR^{-2} Q(\mathbf{u}_0) + 2(n+2b-4)\epsilon K(\mathbf{u}(t)) + C_{\epsilon} R^{-\frac{2(n+2b-1)}{3}} Q(\mathbf{u}_0)^{\frac{5}{3}} \\[2mm]
&=&
2(n+2b)  E(\mathbf{u}_0) - 2(n+2b-4)(1-\epsilon)  K(\mathbf{u}(t)) - 2(n+2b) L(\mathbf{u}(t)) \\[2mm]
& & + CR^{-2} Q(\mathbf{u}_0) + C_{\epsilon} R^{-\frac{2(n+2b-1)}{3}} Q(\mathbf{u}_0)^{\frac{5}{3}} \\[2mm]
&\leq& 
- (n+2b-4)K(\mathbf{u}(t)) + 2(n+2b)  E(\mathbf{u}_0) +  CR^{-2} Q(\mathbf{u}_0) + C_{\epsilon} R^{-\frac{2(n+2b-1)}{3}} Q(\mathbf{u}_0)^{\frac{5}{3}}. 
\end{eqnarray*}
In the inequality above,  we have used the conservation of the energy and the fact that $L(\mathbf{u}(t)) \geq 0$. Note that the last three terms of the above equation do not depend on $t$. Then, we may take $T_1 > T_0$ such that
$$
C_0 \dfrac{(n+2b-4)}{2} T_{1}^{2} \geq 2(n+2b)  E(\mathbf{u}_0) +  CR^{-2} Q(\mathbf{u}_0) + C_{\epsilon} R^{-\frac{2(n+2b-1)}{3}}Q(\mathbf{u}_0)^{\frac{5}{3}},  
$$
where $C_0$ is the constant appearing in \eqref{K>C0t^2}. Thus, 
$$
    \mathcal{R}'(t)
\leq  - (n+2b-4)K(\mathbf{u}(t)) + C_0 \dfrac{(n+2b-4)}{2} T_{1}^{2}.
$$
Since $K(\mathbf{u}(t)) \geq C_0 t^2 \geq C_0 T_{1}^{2} \geq T_{1}$, we obtain
\begin{eqnarray*}
     \mathcal{R}'(t)
&\leq&
- (n+2b-4)K(\mathbf{u}(t)) + \dfrac{(n+2b-4)}{2} K(\mathbf{u}(t)) \\[2mm]
&=& - \dfrac{(n+2b-4)}{2} K(\mathbf{u}(t)), \;\;\;\; t>T_1.
\end{eqnarray*}
Now, integrating the last inequality on $[T_1,t)$ gives us
$$
\mathcal{R}(t) - \mathcal{R}(T_1) \;=\; \int_{T_1}^{t} 
\mathcal{R}'(s) \, ds \; \leq \; - \dfrac{(n+2b-4)}{2} \int_{T_1}^{t}K(\mathbf{u}(s))\, ds.
$$
Since $\mathcal{R}(t) <0$, for $t \geq  T_0$ and $T_1 > T_0$, and then,
$$
\mathcal{R}(t)  \;\leq\;  - \dfrac{(n+2b-4)}{2} \int_{T_1}^{t}K(\mathbf{u}(s))\, ds. 
$$
Therefore,
$$
\dfrac{(n+2b-4)}{2} \int_{T_1}^{t}K(\mathbf{u}(s))\, ds
\leq - \mathcal{R}(t) = | \mathcal{R}(t)| \leq C R Q(\mathbf{u}_0)^{\frac{1}{2}} K(\mathbf{u}(t))^{\frac{1}{2}}.
 $$
Now, define
$$
\eta(t):= \int_{T_1}^{t} K(\mathbf{u}(s))\,ds \;\;\;\;\;\; \mbox{and}\;\;\;\;\;\; A:= \dfrac{(n+2b-4)^{2}}{4C^2R^2 Q(\mathbf{u}_0)}.
$$
Since $K(\mathbf{u}(t)) \geq C_0 t^2$, we have $\displaystyle \eta (t) \geq \int_{T_1}^{t} C_0\, s^2\,ds > 0$, for $t >T_1$.\\
Thus $A \leq \dfrac{\eta^{'}(t)}{\eta(t)^2}$, for $t>T_1$. Finally, taking $T_1 < T^{'}$ and integrating on $[T^{'}, t)$, we obtain
$$
A(t - T^{'}) \;=\; \int_{T^{'}}^{t} A \,ds \leq \int_{T^{'}}^{t} \dfrac{\eta^{'}(s)}{\eta(s)^2} \,ds \;=\;\left. - \dfrac{1}{\eta(s)} \right|_{s=T^{'}}^{s=t} = \dfrac{1}{\eta(T^{'})} - \dfrac{1}{\eta(t)} \leq \dfrac{1}{\eta(T^{'})}. 
$$
In the last inequality we have used  $\eta(t)>0$ when $t > T_1$. The above inequality gives
$$
0 < \eta(T^{'}) \leq \dfrac{1}{A(t - T^{'})},
$$
taking $t \to +\infty$, we obtain $0 < \eta(T^{'}) \leq 0$, which is a contradiction. Hence the proof of Theorem \ref{blow-up} is completed.

 \qed


\textbf{Funding} M. C. is supported by the National Council for Scientific and Technological Development (CNPq), Brazil, Grant No. 307099/2023-7, and by financial support from the CAPES/COFECUB Program, Grant No. 88887.879175/2023-00. L. G. is supported by a scholarship (PDSE 88881.934182/2024-01) from the Brazilian Federal Agency for Support and Evaluation of Graduate Education (CAPES).

\textbf{Acknowledgements}
The first author  acknowledges the support of the 
Conselho Nacional de Desenvolvimento Científico e Tecnológico (CNPq), Brazil and the CAPES/COFECUB Program, Brazil.
The second author gratefully acknowledges the Federal Institute of Minas Gerais (IFMG) for granting leave to pursue doctoral studies, which made this work possible. This project was largely carried out while  L. G. was visiting the Department of Mathematics and Statistics at Florida International University, Miami, FL, during the third year of his PhD program. He thanks the Department and the University for their hospitality and support. The authors thank Svetlana Roudenko for her valuable comments and suggestions, which helped improve the manuscript. This study was financed in part by the Coordenação de Aperfeiçoamento de Pessoal de Nível Superior – Brasil (CAPES) – Finance Code 001.

\addcontentsline{toc}{section}{References}
\bibliographystyle{abbrv}

\bibliography{sn-bibliography}

@article{Bai,
  title={Finite time/infinite time blow-up behaviors for the inhomogeneous nonlinear {S}chr{\"o}dinger equation},
  author={Bai, Ruobing and Li, Bing},
  journal={Nonlinear Analysis},
  volume={232},
  pages={113266},
  year={2023},
  publisher={Elsevier}
}

@article{Begout,
  title={Necessary conditions and sufficient conditions for global existence in the nonlinear
{S}chr{\"o}dinger equation},
  author={Pascal B{\'e}gout},
  journal={Adv. Math. Sci Appl.},
  volume={12},
  pages={817--827},
  year={2002}
}

@article{Bergh,
  title={Interpolation spaces. {A}n introduction},
  author={Bergh, J},
  journal={Grundlehren der Mathematischen Wissenschaften},
  volume={223},
  year={1976}
}

@article{Burchard,
  title={Rearrangement inequalities for functionals with monotone integrands},
  author={Burchard, Almut and Hajaiej, Hichem},
  journal={Journal of Functional Analysis},
  volume={233},
  number={2},
  pages={561--582},
  year={2006},
  publisher={Elsevier}
}

@article{Campos,
  title={Scattering of radial solutions to the inhomogeneous nonlinear {S}chr{\"o}dinger equation},
  author={Campos, Luccas},
  journal={Nonlinear Analysis},
  volume={202},
  pages={112118},
  year={2021},
  publisher={Elsevier}
}

@article{Cardoso,
year = {2022},
publisher = {IOP Publishing},
volume = {35},
number = {8},
pages = {4426--4436},
author = {Cardoso, Mykael and Farah, Luiz Gustavo},
title = {Blow-up of non-radial solutions for the ${{L}^2}$ critical inhomogeneous {NLS} equation},
journal = {Nonlinearity}
}

@article{Cardoso-Murphy,
  title={Scattering below the ground state for the intercritical non-radial inhomogeneous {NLS}},
  author={Cardoso, Mykael and Farah, Luiz Gustavo and Guzm{\'a}n, Carlos M and Murphy, Jason},
  journal={Nonlinear Analysis: Real World Applications},
  volume={68},
  pages={103687},
  year={2022},
  publisher={Elsevier}
}

@book{Cazenave,
  title     = {Semilinear {Schr\"odinger} Equations},
  author    = {Cazenave, Thierry},
  series    = {Courant Lecture Notes in Mathematics},
  volume    = {10},
  year      = {2003},
  publisher = {American Mathematical Society},
  address   = {Providence, RI}
}

@article{Combet,
year = {2016},
volume = {16},
number = {2},
pages = {483--500},
author = {Combet, Vianney
and Genoud, François},
title = {Classification of minimal mass blow-up solutions for an ${L^{2}}$critical inhomogeneous NLS},
journal = {Journal of Evolution Equations}
}

@book{Demenguel,
  title={Functional spaces for the theory of elliptic partial differential equations},
  author={Demengel, Fran{\c{c}}oise},
  year={2012},
  journal={Springer}
}

@article{Dinh,
  title={Blowup of ${H^{1}}$ solutions for a class of the focusing inhomogeneous nonlinear
schrodinger equation},
  author={Dinh, Van Duong},
  journal={Nonlinear Anal. Theory Methods Appl},
  volume={174},
  pages={169--188},
  year={2018}
}

@article{Dinh.2,
  title={Energy scattering for a class of inhomogeneous nonlinear Schr{\"o}dinger equation in two dimensions},
  author={Dinh, Van Duong},
  journal={Journal of Hyperbolic Differential Equations},
  volume={18},
  number={01},
  pages={1--28},
  year={2021},
  publisher={World Scientific}
}

@article{Doung,
  title={A system of inhomogeneous {NLS} arising in optical media with a ${\chi^{(2)}}$ nonlinearity, part {I}: {D}ynamics},
  author={Esfahani, Amin and Dinh, Van Duong  },
  journal={Communications on Pure and Applied Analysis},
  volume={24},
  number={4},
  pages={584--625},
  year={2025},
  publisher={Communications on Pure and Applied Analysis}
}

@article{Esfahani,
  title={Sharp constant of an anisotropic {G}agliardo--{N}irenberg-type inequality and applications},
  author={Esfahani, Amin and Pastor, Ademir},
  journal={Bulletin of the Brazilian Mathematical Society, New Series},
  volume={48},
  number={1},
  pages={171--185},
  year={2017},
  publisher={Springer}
}

@article{Farah,
  title={Global well-posedness and blow-up on the energy space for the inhomogeneous nonlinear {S}chr{\"o}dinger equation},
  author={Farah, Luiz G},
  journal={Journal of Evolution Equations},
  volume={16},
  number={1},
  pages={193--208},
  year={2016},
  publisher={Springer}
}

@article{Guzman,
  title={Scattering for the radial {3D} cubic focusing inhomogeneous nonlinear {S}chr{\"o}dinger equation},
  author={Farah, Luiz Gustavo and Guzm{\'a}n, Carlos M},
  journal={Journal of Differential Equations},
  volume={262},
  number={8},
  pages={4175--4231},
  year={2017},
  publisher={Elsevier}
}

@article{Guzman.2,
  title={Scattering for the radial focusing inhomogeneous {NLS} equation in higher dimensions},
  author={Farah, Luiz Gustavo and Guzm{\'a}n, Carlos M},
  journal={Bulletin of the Brazilian Mathematical Society, New Series},
  volume={51},
  number={2},
  pages={449--512},
  year={2020},
  publisher={Springer}
}

@article{Genoud,
  title={Schrodinger equations with a spatially decaying nonlinearity: existence and stability of standing waves},
  author={Genoud, Fran{\c{c}}ois and Stuart, Charles A},
  journal={Discrete and Continuous Dynamical Systems},
  volume={21},
  number={1},
  pages={137--186},
  year={2008},
  publisher={AIMS PRESS}
}

@book{Gilbarg,
  title     = {Elliptic Partial Differential Equations of Second Order},
  author    = {Gilbarg, David and Trudinger, Neil S.},
  edition   = {2},
  year      = {1983},
  publisher = {Springer-Verlag},
  address   = {Berlin}
}

@article{Hajaiej,
  title={On the necessity of the assumptions used to prove {H}ardy--{L}ittlewood and {R}iesz rearrangement inequalities},
  author={Hajaiej, Hichem},
  journal={Archiv der Mathematik},
  volume={96},
  number={3},
  pages={273--280},
  year={2011},
  publisher={Springer}
}

@article{Hayashi,
title = {On a system of nonlinear {S}chrödinger equations with quadratic interaction},
journal = {Annales de l'Institut Henri Poincaré C, Analyse non linéaire},
volume = {30},
number = {4},
pages = {661-690},
year = {2013},
issn = {0294-1449},
author = {Nakao Hayashi and Tohru Ozawa and Kazunaga Tanaka}
}

@article{Koynov,
  title={Nonlinear phase shift via multistep ${\chi^{(2)}}$ cascading},
  author={Koynov, K and Saltiel, S},
  journal={Optics communications},
  volume={152},
  number={1-3},
  pages={96--100},
  year={1998},
  publisher={Elsevier}
}

@book{Leoni,
  title     = {A First Course in Sobolev Spaces},
  author    = {Leoni, Giovanni},
  series    = {Graduate Studies in Mathematics},
  volume    = {105},
  year      = {2009},
  publisher = {American Mathematical Society},
  address   = {Providence, RI}
}

@book{Lieb,
  title     = {Analysis},
  author    = {Lieb, Elliott H. and Loss, Michael},
  series    = {Graduate Studies in Mathematics},
  volume    = {14},
  year      = {2001},
  publisher = {American Mathematical Society},
  address   = {Providence, RI}
}

@book{Linares-Ponce,
   title     = {Introduction to Nonlinear Dispersive Equations},
  author    = {Linares, Felipe and Ponce, Gustavo},
  series    = {Universitext},
  year      = {2015},
  publisher = {Springer},
  address   = {New York}
}

@article{Miao,
  title={Scattering for the non-radial inhomogeneous {NLS}},
  author={Miao, Changxing and Murphy, Jason and Zheng, Jiqiang},
  journal={Mathematical Research Letters},
  volume={28},
  number={5},
  pages={1481--1504},
  year={2021}
}

@article{Murphy,
  title={A simple proof of scattering for the intercritical inhomogeneous {NLS}},
  author={Murphy, Jason},
  journal={Proceedings of the American Mathematical Society},
  volume={150},
  number={3},
  pages={1177--1186},
  year={2022}
}

@article{Noguera2,
  title={Blow-up solutions for a system of {S}chr{\"o}dinger equations with general quadratic-type nonlinearities in dimensions five and six},
  author={Noguera, Norman and Pastor, Ademir},
  journal={Calculus of Variations and Partial Differential Equations},
  volume={61},
  number={3},
  pages={111},
  year={2022},
  publisher={Springer}
}

@article{Noguera,
  title={A system of {S}chr{\"o}dinger equations with general quadratic-type nonlinearities},
  author={Noguera, Norman and Pastor, Ademir},
  journal={Communications in contemporary mathematics},
  volume={23},
  number={04},
  pages={2050023},
  year={2021},
  publisher={World Scientific}
}

@article{Ogawa,
  title={Blow-up of ${H^{1}}$ solution for the nonlinear {S}chr{\"o}dinger equation},
  author={Ogawa, Takayoshi and Tsutsumi, Yoshio},
  journal={Journal of Differential Equations},
  volume={92},
  number={2},
  pages={317--330},
  year={1991},
  publisher={Elsevier}
}

@article{Ozawa,
  title={Remarks on proofs of conservation laws for nonlinear {S}chr{\"o}dinger equations.},
  author={Ozawa, T},
  journal={Calculus of Variations \& Partial Differential Equations},
  volume={25},
  number={3},
pages={403--408},
  year={2006}
}

@article{Pastor2,
  title={On three-wave interaction {S}chr{\"o}dinger systems with quadratic nonlinearities: {G}lobal well-posedness and {S}tanding waves.},
  author={Pastor, Ademir},
  journal={Communications on Pure \& Applied Analysis},
  volume={18},
  number={5},
  year={2019}
}

@article{Pastor,
  title={Weak concentration and wave operator for a {3D} coupled nonlinear {S}chr{\"o}dinger system},
  author={Pastor, Ademir},
  journal={Journal of Mathematical Physics},
  volume={56},
  number={2},
  year={2015},
  publisher={AIP Publishing}
}


\end{document}